\documentclass[11pt]{amsart}
\usepackage{url}
\usepackage{hyperref,amssymb}
\hypersetup{colorlinks=true,linkcolor=blue,
urlcolor=cyan,citecolor=magenta} 
\usepackage[hyperpageref]{backref}
\usepackage{frcursive,calrsfs,mathrsfs}
\usepackage{xcolor} 
\usepackage{appendix}
\newcommand{\Q}{{\mathbb{Q}}}
\newcommand{\Z}{{\mathbb{Z}}}

\newcommand{\ds}{\displaystyle}

\newcommand{\ov}{\overline}
\newcommand{\wt}{\widetilde}

\newcommand{\ft}{\footnotesize}
\newcommand{\ns}{\normalsize}
\newcommand{\BN}{{\bf N}}
\newcommand{\BJ}{{\bf J}}
\newcommand{\BH}{{\bf H}}
\newcommand{\Bnu}{\hbox{\Large $\nu$}}
\newcommand{\CA}{{\mathcal A}}
\newcommand{\CH}{{\mathcal H}}
\newcommand{\CR}{{\mathcal R}}
\newcommand{\CW}{{\mathcal W}}
\newcommand{\CU}{{\mathcal U}}
\newcommand{\CT}{{\mathcal T}}
\newcommand{\CC}{{\mathcal C}}
\newcommand{\CP}{{\mathcal P}}
\newcommand{\CN}{{\mathcal N}}

\newcommand{\CX}{{\mathcal X}}
\newcommand{\CI}{{\mathcal I}}
\newcommand{\order}{\raise0.8pt \hbox{${\scriptstyle \#}$}}
\newcommand{\lien}{\mathrel{\mkern-4mu}}
\newcommand{\too}{\relbar\lien\rightarrow}
\newcommand{\tooo}{\relbar\lien\relbar\lien\too}
\newcommand{\toooo}{\relbar\lien\relbar\lien\tooo}

\newcommand{\plus}{\ds\mathop{\raise 0.5pt \hbox{$\bigoplus$}}\limits}
\newcommand{\prd}{\ds\mathop{\raise 1.0pt \hbox{$\prod$}}\limits}
\newcommand{\sm}{\ds\mathop{\raise 1.0pt \hbox{$\sum$}}\limits}
\newcommand{\ffrac}[2]{\hbox{\ft $\displaystyle\frac{#1}{#2}$}}
\newcommand{\limproj}{\mathop{\lim_{\longleftarrow}}}
\newcommand{\Gal}{{\rm Gal}}
\newcommand{\Lbda}{{\bf \Lambda}}
\newcommand{\rk}{{\rm rank}}
\newcommand{\tor}{{\rm tor}}
\newcommand{\Ker}{{\rm Ker}}
\newcommand{\pr}{{\rm pr}}
\newcommand{\ta}{{\rm ta}}
\newcommand{\ram}{{\rm ram}}
\newcommand{\ab}{{\rm ab}}
\newcommand{\spl}{{\rm spl}}
\newcommand{\nr}{{\rm nr}}

\newcommand{\lc}{{\rm lc}}
\newcommand{\lac}{{\rm lac}}
\newcommand{\bp}{{\rm bp}}
\newcommand{\cyc}{{\rm c}}
\newcommand{\acyc}{{\rm ac}}
\newcommand{\iw}{{\rm Iw}}
\newcommand{\hk}{{\rm h}_k}
\newcommand{\hp}{{\rm h}_p}
\newcommand{\knr}{\wt {k}^{\raise -1.0pt\hbox{${}^\nr$}}}
\newcommand{\Ccl}{c\hskip-0.1pt{\ell}}
\newcommand{\Bcl}{\hbox{$\rm c\hskip-0.1pt{{l}}$}}

\newcommand{\Log}{{\rm Log}}
\newcommand{\Val}{{\rm V\!al}}
\newcommand{\Rad}{{\rm Rad}}
\newcommand{\Ram}{{\rm Ram}}
\newcommand{\Nram}{{\rm Nram}}
\newcommand{\expo}{{\rm exp}}
\newcommand{\fffrac}[2]{\hbox{$\frac{#1}{#2}$}}
\newtheorem{theorem}{Theorem}[section]
\newtheorem{lemma}[theorem]{Lemma}
\newtheorem{corollary}[theorem]{Corollary}
\newtheorem{conjecture}[theorem]{Conjecture}
\newtheorem{proposition}[theorem]{Proposition}

\newtheorem{definition}[theorem]{Definition}

\newtheorem{remark}[theorem]{Remark}
\newtheorem{remarks}[theorem]{Remarks}

\numberwithin{equation}{section}
\usepackage[english]{babel}
\title[On the $\Z_p$-extensions of a totally $p$-adic imaginary quadratic field]
{On the $\Z_p$-extensions of a totally $p$-adic \\ imaginary 
quadratic field \\ \vspace{0.2cm}
\small{With an Appendix by Jean-Fran\c cois Jaulent}}

\author{Georges Gras}

\address{Villa la Gardette, 4 chemin de Ch\^ateau Gagni\`ere, 
F-38520 Le Bourg d'Oisans}
\email{g.mn.gras@wanadoo.fr}

\urladdr{\url{http://orcid.org/0000-0002-1318-4414}}



\thanks{{\bf Acknowledgments}. 
I sincerely thank the anonymous Referee for his patience, thorough 
reading of the manuscript(s), and valuable comments and suggestions 
that led to a better presentation of the article. I also thank the Editor, 
Christine Huyghe, for her understanding during the evolution of the 
versions. I would like to thank Jean-Fran\c cois Jaulent for crucial discussions 
about the logarithmic class group and for agreeing to write, in Appendix \ref{A}, 
the general result giving the Theorem \ref{fq} for imaginary quadratic fields. 
My friendly thoughts to Christian Maire and Jean-Fran\c cois Jaulent for so many 
reasons.}

\begin{document}

\date{April 24, 2026}

\begin{abstract}
Let $k = \Q(\sqrt {-m})$ and $p \geq 3$ split in $k$. We prove new properties 
of the $\Z_p$-extensions $K/k$, distinct from the cyclotomic one; we do not 
assume $K/k$ totally ramified, nor the triviality of the $p$-class group of $k$. 
These properties are governed by the $p$-valuation $\delta_p(k)$ of a 
Fermat quotient of the fundamental $p$-unit $x$ of $k$, which 
also yields the order of the logarithmic class group $\wt \CH_k$ 
(Thm.\,\ref{fq} extended in App.\,\ref{A}
to the case of imaginary abelian fields of prime-to-$p$ degree), and 
allows to generalize the Gold--Sands criterion (Sec.\,\ref{gold}). 
These results are related to the first two elements, 
$\CH_{K_n}^1$ and $\CH_{K_n}^2$, of the filtrations of the $p$-class 
groups in $K = \cup_n K_n$, without any 
argument of Iwasawa's theory, and provide new perspectives since 
$\order (\CH_{K_n}^2/\CH_{K_n}^1) = \order \wt \CH_k$ for $n$ 
large enough (Thm.\,\ref{stability}). We give  a short proof generalizing 
a result of Kundu--Washington (Thm.\,\ref{noncyclic}) on the 
$p$-class groups in the anti-cyclotomic $\Z_p$-extension $k^\acyc$. 
We compute, Sec.\,\ref{initial}, for $p = 3$,
the first layer $k_1^\acyc$ of $k^\acyc$, using the $\Log_p$-function, 
and show (Thms.\,\ref{capitule},\,\ref{capituleK}) that capitulation of 
suitable ``classes'' is possible in $k^\acyc$, suggesting 
Conjecture\,\ref{mainconj}. Finally, we generalize (Thms.\,\ref{H},\,\ref{nongal}) 
a result of Ozaki giving large $\lambda$'s invariants. 
Calculations and programs are gathered App.\,\ref{C}.

\bigskip\noindent
{\sc R\'esum\'e.} 
Soit $k = \Q(\sqrt {-m})$ et $p \geq 3$ d\'ecompos\'e dans $k$. Nous 
d\'emontrons de nouvelles propri\'et\'es des $\Z_p$-extensions $K/k$, 
distinctes de la cyclotomique~; nous ne supposons pas $K/k$ totalement 
ramifi\'ee ni la trivialit\'e du $p$-groupe des classes de $k$. Ces propri\'et\'es 
sont gouvern\'ees par la $p$-valuation $\delta_p(k)$ d'un quotient de 
Fermat de la $p$-unit\'e fondamentale $x$ de $k$, 
qui donne aussi l'ordre du groupe des classes logarithmiques 
$\wt \CH_k$ (Thm.\,\ref{fq} g\'en\'eralis\'e dans 
l'App.\,\ref{A} au cas des corps ab\'eliens 
imaginaires de degr\'e \'etranger \`a $p$), et permet la g\'en\'eralisation 
du crit\`ere de Gold--Sands (Sec.\,\ref{gold}). Ces r\'esultats sont 
en relation avec les deux premiers \'el\'ements, $\CH_{K_n}^1$ et $\CH_{K_n}^2$, 
des filtrations des $p$-groupes  de classes dans $K = \cup_n K_n$, 
sans aucun argument de la th\'eorie d'Iwasawa, et fournissent 
de nouvelles perspectives car $\order (\CH_{K_n}^2/\CH_{K_n}^1) = 
\order \wt \CH_k$ pour $n$ assez grand (Thm.\,\ref{stability}). Nous donnons 
une courte preuve g\'en\'eralisant un r\'esultat de Kundu--Washington 
(Thm.\,\ref{noncyclic}) sur les $p$-groupes des classes dans la 
$\Z_p$-extension anti-cyclotomique $k^\acyc$. 
Nous calculons, Sec.\,\ref{initial}, pour $p = 3$, 
le premier \'etage $k_1^\acyc$ de $k^\acyc$, via 
la fonction $\Log_p$, et montrons (Thms.\,\ref{capitule},\,\ref{capituleK}) que la 
capitulation de certaines ``classes'' est possible dans $k^\acyc$, sugg\'erant la 
Conjecture\,\ref{mainconj}. Enfin, nous g\'en\'eralisons (Thms.\,\ref{H},\,\ref{nongal}) 
un r\'esultat de Ozaki donnant de grands $\lambda$-invariants. 
Calculs et programmes sont rassembl\'es App.\,\ref{C}.
\end{abstract}

\maketitle

\tableofcontents

\section{Introduction} 

This work, in connection with articles of Gold (1974), Carroll--Kisilevsky 
(1976), Horie (1987), Dummit--Ford--Kisilevsky--Sands (1991), Sands 
(1991--1093), Greenberg (1976--2016), 
Ozaki (2001--2004), then Fujii (2013), Kataoka (2017),
Hubbard--Washington (2010--2018), Hajir--Maire--Ramakrishna (2019), 
Jaulent (2019--2025), Li--Qiu (2023), Kundu--Washington (2023--2024), 
among others, falls more precisely within the framework 
of Jaulent's study of ``totally $p$-adic'' number fields \cite[Section 3]
{Jau2024a}, restricted, in particular for us, to the case of imaginary 
quadratic fields.

\smallskip
In other words, {\it $k$ is a totally $p$-adic number field} 
means that $p$ totally splits in $k$, contrary to some studies
concerning the easier case of a single $p$-place of $k$, totally 
ramified in a $\Z_p$-extension $K/k$ (e.g., \cite[Question p.\,2]{KW2023}, 
\cite[Conjecture 1.3]{EJV2011}, for statistical aspects); of course, the split 
case is also considered in the literature, but is much more complex, essentially 
due to the fact that prime ideals of $k$ dividing $p$ can be non-$p$-principal,
and also give non-trivial $p$-units having a great influence in abelian
$p$-ramification theory.

\smallskip
Our approach is not of the Iwasawa theory type and uses computable arguments of 
classical class field theory by means of the Chevalley--Herbrand formula \cite{Che1933}, 
with the help of some properties of the filtrations\,\footnote{\label{filtration0} 
Let $K = \cup_n K_n$, $\Gal(K_n/k) =: \langle \sigma_n \rangle$, and
$\CH_{K_n}$, the $p$-class group of $K_n$; 
the filtration is the sequence of subgroups $\big\{\CH_{K_n}^i \!:=\!
\{c \in \CH_{K_n},\, c^{(1-\sigma_n)^i} = 1\} \big\}_{i = 0}^{i = b_n}$, with
$\CH_{K_n}^{b_n} = \CH_{K_n}$. Formula \eqref{formule} for the ``jump''
$\CH_{K_n}^{i+1}/\CH_{K_n}^i$ generalizes 
that of Chevalley--Herbrand.} 
attached to the $p$-class groups in a cyclic (or pro-cyclic) 
$p$-extension $K/k$, properties that we gave in the $1994$'s; this point 
of view allows to find some new and precise results, and to understand 
algorithmic $p$-adic aspects that purely algebraic ones do not describe. 
Given the breadth of the subject, this article also provides a survey with 
a large, albeit incomplete, bibliography.

\smallskip
The reader more familiar with the cohomological  Artin--Tate 
presentation of class field theory will have no difficulties since our approach 
is elementary, the more elaborate objects or notions being the Artin symbols, 
and the Hasse norm theorem claiming that in a cyclic extension $M/F$, 
any element of $F$, everywhere local norm (except at a single place 
because of the ``product formula''), is a global norm in $M/F$. 
For a cohomological description of abelian $p$-ramification theory, 
we refer to \cite{NQD1986}, then to \cite{Gra2019c} for a survey and 
practical aspects. For a $p$-adic logarithmic class field theory approach, 
we refer to \cite{Jau1998, Jau1994, BJ2016, DJPS2005, Jau2024b, Jau2026}.

\smallskip
For the so important Chevalley--Herbrand formula, it is clear that the 
$1933$'s proof (in french), may offer some difficulties; so we refer to 
the cohomological Lang's proof \cite{Lan1986}. All variations of this 
formula are classical (\cite{Jau1986, Gra2017a, LiYu2020} and a lot 
of papers). 

\smallskip
It is now accepted that deep $p$-adic obstructions to the knowledge 
of the arithmetic in a $\Z_p$-extension, due in particular to the 
units and/or Fermat's quotients of algebraic numbers, cannot be 
removed by means of a strict algebraic way, and a majority 
of articles supplement the usual techniques of $\Z_p[[T]]$-modules 
with arguments of class field theory type, most often using, in a more 
or less hidden way, the genus theory form of Chevalley--Herbrand 
formula, improperly referenced through articles of Iwasawa or others
(see \cite[Footnote 1, p.46]{Gra2026} for more historical details). 

\smallskip
But, if basic arguments of class field theory are not always sufficient 
either, the simple use, for $k = \Q(\sqrt{-m})$, of the crucial second jumps 
$\CH_{K_n}^2/\CH_{K_n}^1$, deduced from the filtrations, yields many 
results with elementary proofs, which literature sometimes obtains 
with difficulty. Surprisingly, this second jump has, for $n$ large enough,
same order as that of the logarithmic class group $\wt \CH_k$ 
(Theorems \ref{fq}, \ref{stability}):
\begin{equation}\label{firstclog}
\order \wt \CH_k = p^{\delta_p(k)} \times \order \CH_k^{S_k},
\end{equation} 

\noindent
where $S_k$ is the set of the two primes dividing $p$,
$\CH_k^{S_k} := \CH_k/\Ccl(\langle S_k \rangle)$ is the $S_k$-class group,
$\CH_k$ being the $p$-class group of $k$ and $\Ccl({\mathfrak a})$ 
the image of an ideal ${\mathfrak a}$ in $\CH_k$; then $\delta_p(k)$ is
a deep random number, specific of the imaginary split case, elucidating 
the order of the norm residue symbol at $p$ of the fundamental $S_k$-unit 
of $k$. This is generalized in the semi-simple case, to imaginary abelian fields, 
in Appendix \ref{A} and \cite{Jau2024b}, in the context of a Gold criterion, 
by analogy with ordinary class groups. 

\smallskip
We think that, beyond the second 
element $\CH_{K_n}^2$, arithmetic becomes random, which explains 
that only statistics, densities results, may be accessible with the help 
of computations. When adding some hypothesis 
on some invariants, particular results may be stated.

\smallskip
One gets some general results of Gold's criterion type, characterizing 
$\lambda_p(K/k) = 1$, $\mu_p(K/k) = 0$, for Iwasawa's invariants
\cite{Gold1974, San1993, JS1995}. More precisely we obtain, 
in a simpler way, the classical Gold's criterion, and new generalizations 
which apply to any $\Z_p$-extension $K/k$ ramified at the $p$-places from 
some layers, with non-trivial $p$-class group for $k$ (recall that Gold's 
techniques were applied to the cyclotomic $\Z_p$-extension $k^\cyc/k$, 
and are essentially the first ones within the scope of class field theory). 

\smallskip
We notice that, in an imaginary base field $k$, the non-cyclotomic 
$\Z_p$-extensions $K/k$ behave differently from $k^\cyc/k$ because unit 
groups of $K$ do not reduce to those of $\Q^\cyc$ whose norm properties in 
the tower are well-known and, in some sense, do not intervene in $k^\cyc/k$.
But we show some similarities with the totally real case related to the tough 
Greenberg's conjectures \cite{Gree1976, Gree1998}, in particular since
there is a priori no obstruction for the phenomena of capitulations in $K$, as 
for \cite[Theorem 1]{Gree1976}. For some examples of capitulations, see 
Theorems \ref{capitule} and \ref{capitule2}, for $p$ split and non-split, 
respectively, then Program \ref{P7} with computing the first layer 
$k_1^\acyc/k$, of the the anti-cyclotomic $\Z_3$-extension $k^\acyc/k$
(more complete programs are given in \cite{Gra2026}).

\smallskip
A large part of the paper is devoted to the relations between the minimality 
of Iwasawa's invariants and phenomena of capitulation, in $K/k$, of  ``class 
groups'', the best notion of classes, adapted to the context (real or imaginary, 
$p$-split or not), needing to be analyzed. We only consider Iwasawa's invariants 
attached to usual $p$-class groups (eventually $\CH^S$-class groups); for a 
broad generalization to $\CH_T^S$- class groups, for $S$-ramification and 
$T$-decomposition, see \cite{JMP2013} giving the corresponding Iwasawa theory.

\section{Why the logarithmic class group \texorpdfstring{$\wt \CH_k$}{Lg}
for capitulation criteria ?}\label{why}

Nevertheless, capitulation in $K$ of the $p$-class group $\CH_k$ does not 
always characterize the minimality of Iwasawa's invariant
in both real and imaginary context (that is $\mu = 0$, then $\lambda = 0$ in the 
real case, $\lambda \in \{0, 1\}$ in the imaginary case, depending on the 
number of $p$-places ramified in $K/k$).
Indeed, if $\CH_k = 1$ in the split case of $p$ in 
$k = \Q(\sqrt{-m})$, Ozaki--Minardi results \cite[Theorem~1]{Oza2001} 
suggest the existence of exceptional non-Galois $\Z_p$-extensions with 
non-minimal Iwasawa's invariants; either these $\Z_p$-extensions do not 
exist or $\CH_k$ is not the good object for a general capitulation criterion. 
If, at the opposite, $\CH_k \ne 1$ and the $p$-Hilbert class field $H_k^\nr$ is 
contained in $K$, then $\CH_k$ trivially capitulates in $K$ but in this case, we
obtain $p^e \leq \lambda_p(K/k) < 2 p^e$, where $p^e := \order \CH_k$
(Theorems \ref{H}, \ref{nongal}).

\smallskip
The viewpoint of logarithmic class groups essentially intervenes in the 
{\it totally real case} of Greenberg's conjectures, which are, for us, the 
archetypal model illustrating the difficulties of proving such conjectures; 
indeed the theoretical approach gives equivalent non-effective conditions
for the minimality of Iwasawa's invariants: in general they must be satisfied 
for ``all $n$ large enough'', or ``for some $n$'' in the tower $\cup_n K_n$, 
and they depend on random phenomena.

\smallskip
The links between capitulations and the totally real Greenberg's criteria may 
appear diverse, except if one knows that these criteria are {\it equivalent} to the
capitulation of the logarithmic class group $\wt \CH_k$ of $k$ in the (cyclotomic) 
$\Z_p$-extension \cite{Jau2019a}.

\smallskip
Other studies about 
generalized Greenberg's conjecture and capitulation are that of 
Nguyen Quang Do \cite[Theorem C]{NQD2018} and Fujii \cite{Fu2025}
whose bibliography cites erroneous papers, but not that of Jaulent 
\cite[Theorems 5\,(iv),\,7]{Jau2019a}. 

\smallskip
Recall Greenberg's criteria restricted to the simplest case of real quadratic fields, 
for comparison with the imaginary one, and to clarify these questions of capitulation:

\begin{theorem}(\cite[Theorems 1, 2]{Gree1976}). Let $k$ be a real quadratic
field and let $p$ be a prime number. 
Then $\order \CH_{K_n}$ is bounded as $n \to \infty$ if and only if the 
following condition is satisfied, depending on the decomposition of $p$ in $k$:

\smallskip
(i) $p$ non-split.
Every ideal class of $\CH_k$ becomes principal in $K_n$ for some $n$.

\smallskip
(ii) $p$ split.
For $n$ sufficiently large, every class of $B_n$ {\rm $[$sub-module of $\CH_{K_n}$
fixed by $\Gal(K_n/k)$$]$} contains an ideal whose prime factors lie above $p$
{\rm $[$$B_n \!= \CH_{K_n}^\ram$ with our notations$]$}.
\end{theorem}

Recall that $\wt \CH_k \simeq \Gal(H_k^\lc/k^\cyc)$ where $H_k^\lc$ is 
the maximal abelian pro-$p$-extension of $k$, locally cyclotomic (thus 
such that $p$ totally splits in $H_k^\lc/k^\cyc$). In complete generality, 
the finiteness of this Galois group constitutes the Gross-Kuz'min conjecture. 
It is known that any tame place totally splits in $H_k^\pr/k^\cyc$, whatever 
$p$ and the base field~$k$ \cite[Remark\,III.4.8.2]{Gra2005}; so, the 
definition of $H_k^\lc$ only depends on the $p$-places. Of course, 
$H_k^\lc \subseteq H_k^\pr$, the maximal abelian $p$-ramified 
pro-$p$-extension of $k$.

\smallskip
(i) In the non-split case of $p$ in $k$ real, $H_k^\lc = k^\cyc H_k^\nr$
since, by $p$-principality, ${\mathfrak p} \mid p$ totally splits in $H_k^\nr/k$
and ramifies in $H_k^\pr/H_k^\lc$; then, capitulations of $\wt \CH_k$ and 
$\CH_k$ are equivalent \cite[Scolie 9 to Theorem 7]{Jau2019a}, giving 
Greenberg's first criterion.

\smallskip
(ii) In the split case of $p$ in $k$ real, $H_k^\lc$ may be a strict subfield of 
$k^\cyc H_k^\nr$; indeed, the structure of $\CT_k := \Gal(H_k^\pr/k^\cyc)$
is given by the following diagram (from \cite[\S\,2.3]{Jau2019a}), where the 
finite $p$-group $\CR_k$ is 
the normalized $p$-adic regulator (of order ``equivalent'' to $\frac{1}{p} 
\log_p(\varepsilon_k)$ for the fundamental unit of $k$ \cite[Section 5]
{Gra2018}), where $H_k^\pr/k^\cyc H_k^\nr$ is totally ramified, and 
where the image of $\Ccl(S_k)$ is generated by the decomposition 
groups of the two $p$-places of $k$:
\unitlength=1.25cm
\begin{equation*}
\vbox{\hbox{\hspace{-1.8cm}
\begin{picture}(10.0,1.9)
\put(0.6,0.9){\line(1,0){1.3}}
\put(2.9,0.9){\line(1,0){1.8}}
\put(2.1,0.8){$k^\cyc H^\spl_k$}
\put(5.4,0.9){\line(1,0){2.2}}
\put(8.7,0.9){\line(1,0){2.2}}
\put(-0.1,0.82){$\wt k \!=\! k^\cyc$}
\put(4.8,0.8){$H_k^\lc$}
\put(7.8,0.8){$k^\cyc H_k^\nr$}
\put(11.1,0.8){$H_k^\pr$}
\put(1.1,1.0){\tiny$\CH_k^{S_k}$}
\put(3.4,1.0){\tiny$\wt \Ccl(S_k)$}
\put(9.8,1.02){\ft$\CR_k$}
\put(5.2,0.0){\tiny$\CT_k$}
\bezier{400}(0.4,0.66)(3.3,-0.3)(11.2,0.66)
\bezier{300}(0.4,1.1)(4.0,1.7)(4.8,1.1)
\bezier{400}(0.4,1.2)(6.0,2.5)(7.8,1.2)
\bezier{400}(2.4,0.75)(3.0,0.4)(7.9,0.75)
\put(3.4,0.4){\tiny$\Ccl(S_k)$}
\put(5.0,1.92){\tiny$\CH_k$}
\put(3.0,1.45){\tiny$\wt \CH_k$}
\end{picture}}} 
\end{equation*}
\unitlength=1.0cm

By comparison with formula \eqref{firstclog} and the forthcoming 
Definition \ref{wtdeltapk} in the imaginary case, one sees that, 
in the real split case, the ``tricky invariant $\delta_p(k)$'' is not defined
because of the fundamental unit of $k$. If, for $k$ real in the split case, $\CT_k = 1$, 
then Iwasawa's invariants of $k^\cyc$ are trivial.\footnote{\label{prat} 
Under Leopoldt's conjecture for $p$, the 
property ``$\CT_k = 1$'' has been called ``$p$-rationality of $k$'' in Movahhedi's 
thesis \cite{Mov1988}; the long history of this notion can be found in a lot of papers, 
as \cite{Jau1986, Gra1986, NQD1986, MN1990, JN1993, Gra2019, MR2019a, 
MR2019b, Gra2019c, LiQi2023}. Most of number fields are $p$-rational, probably 
for almost all $p$, this being out of reach (see an analysis in \cite{Gra2019a}).
}

\begin{remark}
In the split case of $p$ in $k$ real, the capitulation of $\CH_k$ in the tower 
does no seem to be equivalent to that of $\wt \CH_k$ as shown by the 
following numerical examples with {\sc pari/gp}, where $\wt \CH_k = 1$, giving 
minimal Iwasawa's invariants (but we have seen that Greenberg's criterion is 
different in the split case, and is $\CH_{K_n}^{G_n} = \CH_{K_n}^\ram$ for $n \gg 0$):

\begin{itemize}
\item $k = \Q(\sqrt{79})$, giving $\CH_k \simeq \Z/3\Z$, $\CT_k \simeq \Z/9\Z$, 
$\wt \CH_k = 1$, 

\item $k = \Q(\sqrt{8761})$, giving $\CH_k \simeq \Z/27\Z$, $\CT_k \simeq \Z/81\Z$, 
$\wt \CH_k = 1$, 

\item $k = \Q(\sqrt{1109371})$, giving $\CH_k \simeq \Z/81\Z$, $\CT_k \simeq \Z/81\Z$, 
$\wt \CH_k = 1$,

\item $k = \Q(\sqrt{1141099})$, $k = \Q(\sqrt{1299286})$, giving $\CH_k \simeq \Z/81\Z$, 
$\CT_k \simeq \Z/243\Z$, $\wt \CH_k = 1$;
\end{itemize}

\noindent
capitulation of $\CH_k$ (if any) may be satisfied only in 
large layers of $K$. Then $\CT_k \ne 1$ is not an obstruction to minimal 
Iwasawa's invariants. Note that, under Leopoldt's conjecture,
$\CT_k$ {\it never capitulates} (\cite[Corollaire\,II.2.26, p.\,145]{Jau1986}, 
\cite[III,\,Proposition\,6]{Gra1986}, \cite[Proposition\,1.4]{NQD1986}, 
\cite[Theorem\,IV.2.1]{Gra2005} from \cite{Gra1982, Gra1983, Gra1985}). 
\end{remark}

In conclusion, we can say that the logarithmic class group, in relation with 
the first jump of the filtrations, is an appropriate object for these studies, as well 
as for numerical criteria, even if any phenomenon of capitulation in the tower is 
unverifiable in practice, except when the test is sufficient in the first layer,
which case is characterized in \cite[Remark\,2.2, Theorem\,3.4]{Gra2017b}, 
then in \cite[Th\'eor\`eme\,6, Proposition\,12]{Jau2024a}.

The philosophy of using $\wt \CH_k$ in the split case, is that it encodes 
information within the studied number field, that usual $S$-class groups 
in a tower do not provide. This is also clear in the imaginary case from the study, 
for instance, of the anti-cyclotomic $\Z_p$-extensions, where $\wt \CH_k$ plays 
a decisive role (e.g., Theorem \ref{stability}\,(iii)). Nevertheless, a capitulation of 
the logarithmic class group, in the imaginary split case, is not sufficient and needs 
further investigation.

\section{Main results and perspectives}

The present ``semi-simple'' study seems to depend on the notion of ``smooth 
complexity'' of $K$ that we introduce in Section \ref{complexity}; so we conjecture 
that for $k$ imaginary quadratic, $p \geq 3$ split in $k$, $K/k$ totaly ramified
at the two $p$-places,
$K \ne k^\cyc$, $\lambda_p(K/k) = 1$ and $\mu_p(K/k) = 0$, with similar results 
for imaginary abelian fields and odd isotypic components (Appendix~\ref{A}).
If $K/k$ is not totally ramified, Theorems \ref{H}, \ref{nongal}, show that
the notion of ``minimal invariants'' must be adapted.

\smallskip
By comparison, we consider, \S\,\ref{nonzero}, the famous Iwasawa counterexample, 
with {\it non semi-simple} action of $\Gal(k/\Q)$ and ingenious use of Chevalley--Herbrand 
formula, giving $\Z_p$-extensions $K/k$ with $\mu_p(K/k) > 0$, to examine the 
capitulation phenomenon in such an ad hoc construction (Section \ref{P6}).

\subsection{Logarithmic writings}\label{log}
Set $g = \Gal(k/\Q) =: \langle \tau \rangle$, where $\tau$ is the complex conjugation.
Let $k$ be an imaginary quadratic field and let $p$ be an odd prime split in $k$;
set $(p) = {\mathfrak p}{\ov {\mathfrak p}}$, where $z \to \ov {z}$ is the corresponding 
notation when replacing an object by its complex conjugate, if no confusion is 
possible (e.g., if $p^e$ is an inertia group index of ${\mathfrak p}$, $p^{\ov e}$ will 
be that of $\ov {\mathfrak p} = {\mathfrak p}^\tau$).

\smallskip                                                     
Let $\log_{\mathfrak p}$ and $\log_{\ov {\mathfrak p}}$ 
be the corresponding $\log$-functions in the two $p$-completions 
$k_{\mathfrak p}$ and $k_{\ov {\mathfrak p}}$ of $k$
(isomorphic to $\Q_p$), {\it with the convention 
$\log_{\mathfrak p}(p) = \log_{\ov {\mathfrak p}}(p) = 0$}.
Let $\hk$ be the class number of $k$, and let 
$x_{\mathfrak p} =: x \in k^\times$ be the fundamental 
${\mathfrak p}$-unit of $k$ (obtained, up to sign, from 
${\mathfrak p}^{\hp} =: (x)$, $\hp \mid \hk$ 
being the common order of the classes of ${\mathfrak p}$ and
$\ov {\mathfrak p}$ since $\Ccl(\ov {\mathfrak p}) = \Ccl({\mathfrak p})^{-1}$); 
so, $x_{\ov {\mathfrak p}}^{} = \ov x_{\mathfrak p}^{} = \ov x$. 
Since $x \ov x = p^{\hp}$, this defines, in the 
$\Z_p[g]$-algebra $k_{\mathfrak p} \oplus k_{\ov {\mathfrak p}}$:
\begin{equation*}
\left \{\begin{aligned}
& \log_{\mathfrak p}(\ov x) \ \hbox{(usual definition since 
$\ov x$ is a unit in $k_{\mathfrak p}^\times$)}, \\
& \log_{\ov {\mathfrak p}}(x) \ \hbox{(conjugate 
of $\log_{\mathfrak p}(\ov x)$)}, \\
& \log_{\mathfrak p}(x) = - \log_{\mathfrak p}(\ov x), \  \
\log_{\ov {\mathfrak p}}(\ov x) = - \log_{\ov {\mathfrak p}}(x).
\end{aligned}\right.
\end{equation*}

Thus, this defines, without any ambiguity, in terms of $p$-adic 
valuations on $\Q_p^\times$:
$$\delta_p(k) \!:= v_p \big (\fffrac{1}{p} \log_{\mathfrak p}(\ov x) \big)
\!= v_p \big (\fffrac{1}{p} \log_{\ov {\mathfrak p}}(x) \big) 
\!= v_p \big (\fffrac{1}{p}\log_{\mathfrak p}(x) \big)
\!= v_p \big (\fffrac{1}{p} \log_{\ov {\mathfrak p}}(\ov x) \big)
\!= v_{\mathfrak p}(\ov x^{\,p-1} \! -1) - 1 .$$

\begin{definition}\label{wtdeltapk}
Let $\hk$ be the class number of $k$ (i.e. the order of the class 
group $\BH_k$) and let $\hp \mid \hk$ be the 
order of the classes of ${\mathfrak p}$ and ${\ov {\mathfrak p}}$. Put: 
$$\hbox{${\mathfrak p}^{\hp} =: (x)$ and 
${\mathfrak p}^{\hk} =: (X)$, hence $X = x^{\hk 
\cdot \hp^{-1}}$, }$$ 

\noindent
and define the integers $\delta_p(k) = v_{\mathfrak p}(\ov x^{\,p-1}-1) - 1$
and $\wt \delta_p(k) := v_{\mathfrak p}
(\ov X^{\,p-1}-1) - 1$ satisfying the obvious 
relation $\wt \delta_p(k) = \delta_p(k) + \big [v_p (\hk) 
- v_p(\hp) \big] = v_{\mathfrak p}(\ov x^{\,p-1}-1) - 1
+ \big [v_p (\hk) - v_p(\hp) \big]$,
where $v_p (\hk) - v_p(\hp)$ is the valuation 
of the $S_k$-class group $\CH_k^{S_k} := \CH_k/ \Ccl(S_k)$, 
$S_k := \{{\mathfrak p}, \ov {\mathfrak p}\}$. 
For $k$ fixed, $\wt \delta_p(k) = \delta_p(k)$ for almost all primes $p$.
\end{definition}

\subsection{Summary of the results}

Let $k$ be an imaginary quadratic field, let $p$ be an 
odd prime, split in $k$, and set $(p) = {\mathfrak p}\,{\ov {\mathfrak p}}$. 
Denote by $\BH_k$ (resp. $\CH_k \simeq \BH_k \otimes \Z_p$) the class 
group (resp. the $p$-class group) of $k$. For any $\Z_p$-extension $K/k$, 
let $K_n \subset K$ be the layer of degree $p^n$ over $k$, and let 
$\CH_{K_n}$ be its $p$-class group. In general $K/k$ is 
non-Galois and its conjugate $\ov K := K^\tau$ is such that 
$\Gal(K \ov K/k) \simeq \Z_p^2$. 
We assume that ${\mathfrak p}$ (resp. ${\ov {\mathfrak p}}$) is ramified 
from some layer $K_e$ (resp. $K_{\ov e}$) of $K$ ($e \geq \ov e \geq 0$),
which excludes two canonical $\Z_p$-extensions defined in 
\S\,\ref{L} as inertia fields of ${\mathfrak p}$ and ${\ov {\mathfrak p}}$ in $\wt k/k$.

\smallskip
See Theorem \ref{6cases} for the parametrization of these 
$K/k$'s and for information on the resulting integers $e$, $\ov e$, 
their equality being false in infinitely many non-Galois cases.

\smallskip
Note that in the split case for $p$ in $k$ and assuming that 
conjecturally $\mu_p(K/k) = 0$, the existence of the layers
$K_e$, $K_{\ov e}$, implies $\lambda_p(K/k) \geq 1$
from Chevalley--Herbrand formula.

\smallskip
Before giving an excerpt of the results, we just 
report the \S\S\,\ref{nonkummercase}, 
\ref{verifcap}, giving an easy way to compute the first layer of any $\Z_p$-extension 
of $k$ (e.g., Program \ref{P7}, for $p = 3$ and $K = k^\acyc$), and the observation of 
capitulations. In Appendix \ref{C}, explicit tables and computations, with {\sc pari/gp} 
programs, are given; the programs may be copied and pasted easily by anyone;
they run with any {\sc pari/gp} package, or, if necessary, using:

\smallskip
\centerline{\cite[\url{https://pari.math.u-bordeaux.fr/gpwasm.html}]{Pari2019}.}

\smallskip
Refer to the above Definition \ref{wtdeltapk} for definitions of
$\delta_p(k)$, $S_k$, $x$:

\medskip
({\bf a}) (Theorem \ref{fq}, Appendix \ref{A}).
Let $\wt \CH_k$ be the logarithmic class group of $k$ 
(isomorphic to $\Gal(H_k^\lc/k^\cyc)$, where $k^\cyc$ is the cyclotomic 
$\Z_p$-extension and $H_k^\lc$ the maximal abelian 
pro-$p$-extension of $k$, locally cyclotomic. 
Then $\order \wt \CH_k = p^{\delta_p(k)} \times \order \CH_k^{S_k}$.

\smallskip
({\bf b}) (Theorem \ref{mainorder}).
The Hasse norm residue symbols $\ds \big(\fffrac{\ov x\,,\,K_n/k}{\mathfrak p}\big)$,
$\big(\fffrac{x\,,\,K_n/k}{{\ov {\mathfrak p}}}\big)$ are of orders 
$p^{\max(0, n - e - \delta_p(k))}$, $p^{\max(0, n - \ov e - \delta_p(k))}$, 
respectively.

\smallskip
({\bf c}) (Theorem \ref{stability}). 
The order of $\wt \CH_k$ is also given by 
$\order (\CH_{K_n}^2 / \CH_{K_n}^1)$, for $n$ large enough, where 
$\big\{\CH_{K_n}^i \big\}_{i \geq 0}$ is the filtration of $\CH_{K_n}$ 
(see Footnote \ref{filtration0}). 
Thus, $\wt \CH_k = 1$ leads to $\order \CH_{K_n} = 
\order \CH_{K_n}^1 = p^{n + v_p(\order \CH_k) - e - \ov e}$, whence 
$\lambda_p(K/k) = 1$, $\mu_p(K/k) = 0$, $\nu_p(K/k) = 
v_p(\order \CH_k) - e - \ov e$.

\smallskip
({\bf d}) (Theorem \ref{goldplus}).
Assume that $\CH_k$ is generated by $\Ccl({\mathfrak p})$, ${\mathfrak p} \mid p$, 
and that ${\mathfrak p}$ and ${\ov {\mathfrak p}}$ totally ramify in $K/k$ 
($e = \ov e = 0$). Then, the relation: ``\,$\order \CH_{K_n} = p^{n + v_p(\order \CH_k)}$ 
for all $n \geq 0$\,'', is equivalent to the condition ``\,$\delta_p(k) = 0$\,''
(whence to the condition $\wt \CH_k = 1$).

\smallskip
We obtain again the following Gold--Sands criterion for $k^\cyc$, 
which has some importance:

\smallskip
({\bf e}) (Theorem \ref{sands}).
If $\Ccl({\mathfrak p}) = 1$, the cyclotomic $\Z_p$-extension $k^\cyc/k$ 
is such that $\lambda_p(k^\cyc/k) \geq 2$ if and only if $\wt \CH_k \ne 1$, 
whence, if and only if ``\,$v_p(\order \CH_k) \geq 1$ or $\delta_p(k) \geq 1$\,''.

\smallskip
({\bf f}) (Theorems \ref{H}, \ref{nongal}).
If the $p$-Hilbert class field $H_k^\pr$ is contained in $K$, if 
$\Ccl({\mathfrak p}) = 1$ and if $\delta_p(k) = 0$, then 
``\,$\mu_p(k^\acyc/k) = 0\ \ \&\ \ p^e \leq \lambda_p(k^\acyc/k) < 2p^e$\,'', 
where $p^e = \order \CH_k$.

\medskip
An essential subject is to describe $H_k^\pr$, the maximal $p$-ramified 
abelian pro-$p$-extension of $k$ and the torsion group $\CT_k := 
\Gal(H_k^\pr/\wt k)$ (see \cite{NQD1986} for its cohomological definition 
and most of our papers for its arithmetic description, as
\cite[III.2\,(c)]{Gra2005}). We use, numerically,
suitable ray class fields,  we characterize $\CT_k$ as the kernel of the 
$\Log_p$-function and obtain an effective method 
of computation of the first layer $K_1$ of any $\Z_p$-extension $K/k$.

\medskip
Many pioneering works dealing with analogues of Gold's criterion consider $k^\cyc$
assuming $\CH_k = 1$. More recent literature succeeds in obtaining similar results
in more general cases, but, to our knowledge, the consideration of non-trivial inertia 
fields $K_e$ and $K_{\ov e}$, does not exist. 
Some results are based on assuming that $k$ has a smooth arithmetic complexity, 
e.g., $k$ is $p$-principal ($\CH_k = 1$), or logarithmically principal ($\wt \CH_k = 1$), 
or $p$-rational ($\CT_k = 1$, see Footnote \ref{prat}), or the filtrations 
$\{\CH_{K_n}^i\}_{i \geq 0}$ are limited to the step $i = 1$ giving $\CH_{K_n}
= \CH_{K_n}^1$ and Iwasawa's formula via the Chevalley--Herbrand formula. 
In Section \ref{complexity} we propose a general definition of the ``smooth 
complexity'' of a cyclic (or pro-cyclic) $p$-extension and show the strong link with 
capitulation of $\CH_k$ in the tower.

\smallskip
The general case ($K \ne k^\cyc$, $\CH_k$ arbitrary, and $\CH_{K_n} = 
\CH_{K_n}^{b_n}$ for the $b_n$'s defining the lengths of the filtrations) 
depends on random $p$-adic numerical circumstances in the 
computation of $\CH_{K_n}^i$ for $i > 2$, but subject to accessible 
probability laws, as done in \cite{KP2022} in neighboring areas. We 
discuss this aspect in \S\,\ref{obstructions} to try to understand the 
nature of the obstructions for large $\lambda_p(K/k)$'s, which clearly 
depends on the groups of units in the $\Z_p$-extensions, and on 
capitulation aspects. 

\smallskip
A precise theoretical result may be found in Grandet--Jaulent \cite[Theorem\,(i-iv), 
p.\,214]{GrJa1985}, only assuming the triviality of $\mu_p(K/k)$, but valid for any 
$\Z_p$-extension of any base-field, and showing that for $n$ large enough:
$$\hbox{$\CH_{K_n} \simeq \big(\bigoplus_{i = 1}^\lambda 
\Z/p^{n+\alpha_i}\Z \big) \oplus \big(\bigoplus_{i = \lambda + 1}^\kappa 
\Z/p^{\alpha_i} \Z\big)$, $\alpha_1, \ldots, \alpha_\lambda \in \Z$, 
$\alpha_{\lambda+1}, \ldots, \alpha_\kappa \in \Z_{> 0}$,} $$ 

\noindent
where $\bigoplus_{i = \lambda + 1}^\kappa 
\Z/p^{\alpha_i} \Z \simeq {\rm Cap}_{K_n}$, the subgroup of $p$-classes 
of $K_n$ which capitulate in $K$, these $\alpha_i$'s being 
independent of $n$. Moreover, let ${\mathfrak T}(K/k) := \ds \tor_{\Z_p}^{} 
\big(\limproj \CH_{K_n} \big)$; then for all $n$ large enough, the norm map 
induces an isomorphism between ${\mathfrak T}(K/k)$ and ${\rm Cap_{K_n}}$.

\subsection{Prerequisites on \texorpdfstring{$p$}{Lg}-ramification theory}

Let $k$ be an imaginary quadratic field and set $g = \Gal(k/\Q) =: 
\langle \tau \rangle$.

\subsubsection{The cyclotomic and anti-cyclotomic 
\texorpdfstring{$\Z_p$}{Lg}-extensions}

Let $p$ be an odd prime number. Let $\wt k$ be the 
compositum of the $\Z_p$-extensions of $k$. 
We denote by $\CH_k$ (resp. $H_k^\nr$) the $p$-class group (resp. the $p$-Hilbert
class field) of $k$, and put $\knr := \wt k \cap H_k^\nr$. Let $\Bcl({\mathfrak a})$ 
(resp. $\Ccl({\mathfrak a})$), be the image of an ideal ${\mathfrak a}$ in the  
class group $\BH_k$ (resp. in the $p$-class group $\CH_k$).

\smallskip
(i) The group $g$ operates on $\Gamma := \Gal(\wt k/k) \simeq \Z_p^2$ 
and we define $\Gamma^+ = \Gamma^{\frac{1+\tau}{2}}$ and
$\Gamma^- = \Gamma^{\frac{1-\tau}{2}}$. 
Thus, $\Gamma = \Gamma^+ \oplus \Gamma^-$ and 
$\Gamma^+ \cap \Gamma^- = 1$, leading to define the two  
{\it fundamental} $\Z_p$-extensions yielding $\wt k$ as direct compositum  
over $k$; they are the only ones to be Galois over $\Q$:

\smallskip
$\bullet$ The cyclotomic $\Z_p$-extension $k^\cyc$ (fixed by $\Gamma^-$), equal to 
the compositum $k \,\Q^\cyc$, where the cyclotomic $\Z_p$-extension $\Q^\cyc$ 
is linearly disjoint from $k$; so $k^\cyc/k$ is totally ramified at $p$.

\smallskip
$\bullet$ The anti-cyclotomic $\Z_p$-extension $k^{\acyc}$ (fixed by $\Gamma^+$) 
is a pro-diedral $\Z_p$-exten\-sion of $\Q$. 

\smallskip
(ii) Since $p \ne 2$, $\tau$ operates on a 
$p$-class, $\Ccl({\mathfrak a})$, of $k$, by $\Ccl({\mathfrak a})^{\,\tau} = 
\Ccl({\mathfrak a})^{\,-1}$, and operates similarly on $\Gal(H_k^\nr/k)$; thus 
$\knr = k^{\acyc} \cap H_k^\nr$, the extension $\knr/k$ is $p$-cyclic, and 
$k^{\acyc}/\knr$ is totally ramified at $p$. 
In the same way, $\tau$ operates by inversion on $\wt \CH_k$, $\CT_k$ 
and $\Gal(H_k^\pr/k^\cyc)$.

\subsubsection{Abelian \texorpdfstring{$p$}{Lg}-ramification over
\texorpdfstring{$k$}{Lg} (split case of \texorpdfstring{$p$}{Lg} in $k$)}

The inertia groups, $I_{\mathfrak p}(H_k^\pr/k)$ and 
$I_{\ov {\mathfrak p}}(H_k^\pr/k)$, of ${\mathfrak p}$ and 
${\ov {\mathfrak p}}$ in $H_k^\pr/k$, are images of the groups of 
principal local units $\CU_{\mathfrak p}$ and $\CU_{\ov {\mathfrak p}}$ 
(isomorphic to $1+p\Z_p$) of the completions $k_{\mathfrak p}$ and 
$k_{\ov {\mathfrak p}}$, respectively, by the local reciprocity maps, 
according to the following first diagram of the maximal $p$-sub-extension 
$k^\ab$ of an abelian closure of $k$, where $H_k^\ta$ is the maximal 
tamely ramified abelian pro-$p$-extension of $k$.

\smallskip
In \cite[III.4.4.1]{Gra2005} is proved that, in this diagram, 
$\Gal(k^\ab/H_k^\pr H_k^\ta)$ is isomorphic to $E_k \otimes \Z_p$, where 
$E_k$ is the group of units of~$k$. In the present case, it is reduced to the group
$\mu_k^{} \otimes \Z_p$ of $p$-roots of unity of $k$; for $p \ne 2$, 
$\mu_k^{} \otimes \Z_p = 1$,
except for $p = 3$ and $k = \Q(\sqrt{-3})$, but the assumption $3$ split in $k$ 
does not hold, and $H_k^\pr H_k^\ta = k^\ab$ with $\Gal(H_k^\pr/H_k^\nr) 
\simeq \Z_p \times \Z_p$, since $\tor_{\Z_p}^{}(\CU_k) = 1$:
\unitlength=0.85cm
\begin{equation}\label{Tpsplit1}
\begin{aligned}
\vbox{\hbox{\hspace{2.6cm}
\begin{picture}(10.0,2.1)
\put(1.6,2.0){\line(1,0){3.5}}
\put(1.6,0.18){\line(1,0){3.9}}
\put(-1.2,0.18){\line(1,0){1.9}}
\put(1.0,0.5){\line(0,1){1.20}}
\put(6.00,0.5){\line(0,1){1.20}}
\put(6.2,1.1){\ft$\CU_k = \CU_{\mathfrak p}
\oplus \CU_{\ov {\mathfrak p}} \simeq \Z_p^2$}
\put(-0.3,1.1){\ft$I_{\mathfrak p}
\oplus I_{\ov {\mathfrak p}}$}
\put(5.2,1.9){$H_k^\pr H_k^\ta=k^\ab$}
\put(0.8,1.9){$H_k^\pr$} 
\put(5.8,0.1){$H_k^\ta$}
\put(0.8,0.1){$H_k^\nr$}
\put(-1.6,0.1){$k$}
\end{picture}}} 
\end{aligned}
\end{equation}
\unitlength=1.0cm

Since $I_p(k^\cyc/k)$ and $I_p(k^\acyc/k)$ 
are isomorphic to $\Z_p$, $H_k^\pr/\wt k$ is unramified, $\wt k/k^\cyc$ and 
$\wt k/k^\acyc$ are unramified (or apply \cite[Lemme 4]{Jau2024a}).                

\smallskip
The second diagram (\cite[III.2.6.1\,(Fig. 2)]{Gra2005}), corresponding 
to $H_k^\pr/k$, admits also some 
simplifications from $\tor_{\Z_p}^{}(\CU_k) = 1$, which implies 
$\CW_k^\bp := \tor_{\Z_p}^{}(\CU_k)/\mu_p(k) = 1$, and $H_k^\pr$ 
coincides with the Bertrandias--Payan field $H_k^\bp$ (maximal abelian
pro-$p$-extension of $k$ in which any cyclic extension is embeddable
in cyclic $p$-extensions of $k$ of arbitrary degree \cite{BP1972}). 
Take care that in more general setting (as in \cite{Jau2019a}), the good object 
may be $\CT_k^\bp$ instead of $\CT_k$. In our context, $\CW_k^\bp \ne 1$
is equivalent to $p = 3$ and $-m \equiv -3 \pmod 9$, which is excluded
(but not in \cite{Gra2026} where all cases are considered):
\unitlength=0.85cm 
\begin{equation}\label{Tpsplit2}
\begin{aligned}
\vbox{\hbox{\hspace{-4.0cm} 
\begin{picture}(10.0,3.2)
\put(6.6,2.46){\line(1,0){1.45}}
\put(6.6,2.52){\line(1,0){1.45}}
\put(9.1,2.46){\line(1,0){1.55}}
\put(9.1,2.52){\line(1,0){1.55}}
\put(3.8,2.50){\line(1,0){1.6}}
\put(3.8,2.2){\tiny${\rm totally\ split}$}
\put(9.25,2.1){\ft$\CW_k^\bp\! = \!1$}
\put(3.85,0.50){\line(1,0){1.45}}
\put(1.7,0.50){\line(1,0){1.45}}
\put(1.3,0.40){$k$}
\bezier{300}(1.38,0.35)(3.5,0.0)(5.6,0.35)
\put(3.2,-0.1){\ft$\CH_k$}
\put(4.35,0.64){\ft$\CH_k{'}$}
\bezier{400}(3.7,2.7)(7.25,3.3)(10.8,2.7)
\put(7.0,3.14){\ft$\CT_k$}
\put(3.50,0.9){\line(0,1){1.25}}
\put(5.7,0.9){\line(0,1){1.25}}
\bezier{350}(6.2,0.5)(8.5,0.6)(10.9,2.3)
\put(8.6,0.75){\ft$\CU_{k}$}
\put(10.85,2.4){$H_k^\pr$}
\put(5.6,2.4){$\wt k H_k^\nr$}
\put(8.15,2.4){${H_k^\bp}$}
\put(6.8,2.15){\ft$\CR_k\! = \!1$}
\put(3.35,2.4){$\wt k$}
\put(5.5,0.4){$H_k^\nr$}
\put(3.35,0.4){$\knr$}
\put(2.25,1.5){\ft $\Z_p \!\times\! \Z_p$}
\end{picture}}} 
\end{aligned}
\end{equation}
\unitlength=1.0cm

The normalized regulator $\CR_k$ is trivial 
since $E_k \otimes \Z_p = 1$, hence $H_k^\pr = \wt k H_k^\nr$.
Thus, $\CT_k$ is isomorphic to a subgroup $\CH_k{'}$ of $\CH_k$
and we have the fundamental formula:
$$\big[\knr : k \big] := \big[\wt k \cap H_k^\nr : k \big] = 
[k^\acyc \cap H_k^\nr : k] = 
\ffrac{\order \CH_k}{\order \CT_k}, $$

\noindent
easily computable with {\sc pari/gp} instructions.
We find again that $H_k^\pr/\wt k$ is unramified; then it is totally split 
at $p$; indeed, since the Gross--Kuz'min conjecture holds for $k$
(finiteness of splitting of the $p$-places in $H_k^\pr/k$), the 
(pro-cyclic) decomposition groups of ${\mathfrak p}$ and ${\ov {\mathfrak p}}$ 
in $H_k^\pr/k^\cyc$ are isomorphic to $\Z_p$ and can not intersect 
$\CT_k$ non-trivially. This case differs from the case $p$ non-split in $k$
for which $p$ is totally ramified in $\wt k/k^\acyc \cap H_k^\nr$ and where
$\Gal(H_k^\lc/k^\cyc) \simeq \CH_k$. 

\smallskip
The extension $\wt k/k^\cyc$ being unramified, the decomposition groups 
of ${\mathfrak p}$ and ${\ov {\mathfrak p}}$ in this extension are of finite index
(this is the purpose of \S\,\ref{split}); similarly, the decomposition groups are 
of finite index in $\Gal(\wt k/k^\acyc)$. These indices depend canonically on 
$p^{\delta_p(k)}$ and of the splitting of $p$ in $k^\acyc$. In \S\,\ref{L}
we will precise where inertia and decomposition groups take place.

\begin{remark}
The general study of $\CA_k := \Gal(H_k^\pr/k)$ and its torsion group
$\CT_k$ is done in our book, and we have the other formula, whatever 
the base-field $k$ and the prime $p$:
$$[\knr : k] = \big(\Log_p(I_k \otimes \Z_p) : \Log_p(P_k \otimes \Z_p) \big), $$

\noindent
where $I_k$ is the group of prime-to-$p$ ideals of $k$ and $P_k$ the sub-group 
of principal ideals; then $\Log_p({\mathfrak a}) := \frac{1}{h} \log_p(\alpha)
\pmod {\Q_p\log_p(E_k)}$ if ${\mathfrak a}^{h} = (\alpha)$, 
$(\alpha) \in P_k \otimes \Z_p$. 
Via the Artin Symbol, $\Log_p$ defines an isomorphism $\Log_p(I_k \otimes \Z_p) 
\simeq \Gal(\wt k/k)$. This will be involved in Section \ref{nonkummercase}.

Set $\CH_k = \langle \Ccl({\mathfrak a}_1) \rangle \oplus \cdots \oplus
\langle \Ccl({\mathfrak a}_r)\rangle$; the degree $[\knr : k]$ only depends 
on the Fermat quotients of the $\alpha_i$'s deduced from the ${\mathfrak a}_i$'s,
and allows statistics on the values of this degree when $k$
is imaginary quadratic, as done in Kundu--Washington \cite{KW2023, KW2024}.
\end{remark}

\subsubsection{General diagram of ramification in 
\texorpdfstring{$H_k^\pr/k$}{Lg}}\label{L}

We introduce the inertia fields $L$ of ${\mathfrak p}$ and $\ov L$ of 
${\ov {\mathfrak p}}$, in $\wt k/k$; so, $L/k$ (resp. $\ov L/k$) is unramified 
at ${\mathfrak p}$ (resp. ${\ov {\mathfrak p}}$). Then, $\ov L \ne L$, 
$L \,\ov L = \wt k$, and $L \cap \ov L = \knr$ (indeed, we have $L \cap \ov L \subset 
\knr$; so, $L\, \knr/L$, being unramified at ${\mathfrak p}$, is trivial, hence 
$\knr \subset L$, $\knr \subset \ov L$).\,\footnote{\,If $e$, $\ov e$, $L$, $\ov L$, 
are non-ambiguous notations associated to ${\mathfrak p}$, ${\ov {\mathfrak p}}$,
it is not the case in a layer $K_n$ of a non-Galois $\Z_p$-extension $K$
since the complex conjugate of an object $X_{K_n}({\mathfrak p})$ of $K_n$ is not 
$X_{K_n}({\ov {\mathfrak p}})$ of $K_n$, but $X_{\ov K_n}({\ov {\mathfrak p}})$.
The simplest example is that of ${\mathfrak P}_n \mid
{\mathfrak p}$ in $K_n$; its image by $\tau$ is an ideal $\ov {\mathfrak P}_n 
\mid {\ov {\mathfrak p}}$ of $\ov K_n$ (cf. Footnote \ref{Kprime}).} 
Since ${\mathfrak p}$ does not ramify in $L/k$,
$L \cap k^\cyc = k$, thus $\wt k$ is the direct compositum $L\,k^\cyc$ over $k$,
$\Gal(L/k) \simeq \Gal(\wt k/k^\cyc) \simeq \Z_p$; thus $L/k$, $\ov L/k$ are 
$\Z_p$-extensions.

\smallskip
The field $H'^\pr_k = k^\acyc H_k^\nr$ is canonical, and 
$\Gal(H'^\pr_k/k^\acyc) \simeq \CT_k$, but $p$ does not 
necessarily totally split in $H'^\pr_k/k^\acyc$. See Remarks
\ref{lognonsplit} and \ref{logsplit} for the definition and properties
of $H_k^\lac$, the anti-cyclotomic analogue of $H_k^\lc$. Of course,
if $p$ does not split in $k$, $H_k^\lac = H'^\pr_k$. 

\smallskip
In the following diagrams, ramification (resp. non-ramification) of 
${\mathfrak p}$, ${\ov {\mathfrak p}}$ is made explicit with indications
${\mathfrak p} \Ram$ (resp. ${\mathfrak p} \Nram$):
\unitlength=1.5cm 
\begin{equation}\label{ramification}
\begin{aligned}
\vbox{\hbox{\hspace{-2.0cm}
\begin{picture}(8.5,5.0)
\put(3.44,4.35){$\wt k$}
\put(3.7,4.45){\line(1,0){3.15}}
\bezier{300}(3.5,4.65)(5.25,5.0)(7.0,4.65)
\put(5.15,4.9){$\CT_k$}
\put(4.0,1.1){\line(1,0){2.8}}
\put(4.7,1.15){$\simeq\!\CT_k$}
\put(6.9,1.05){$H_k^\nr$}
\put(6.95,4.35){$\wt k H_k^\nr \!=\! H_k^\pr$}
\put(1.2,2.26){$k^\cyc$}
\put(5.3,2.42){$k^\acyc$}
\put(1.55,2.5){\line(1,1){1.78}}
\put(1.6,2.3){\line(1,-1){2.3}}
\put(4.0,1.2){\line(1,1){1.2}}
\put(7.35,1.3){\line(1,1){1.1}}
\put(8.45,2.45){$H'^\pr_k$}
\put(5.68,2.5){\line(1,0){2.72}}
\put(6.4,2.3){\tiny${\mathfrak p}, {\ov {\mathfrak p}} \,\Nram$}
\put(6.4,2.3){\tiny${\mathfrak p}, {\ov {\mathfrak p}} \,\Nram$}
\put(7.8,3.65){\tiny${\ov {\mathfrak p}}\,\Nram$}
\put(8.0,3.45){\tiny${\mathfrak p}\,\Nram$}
\put(8.0,1.8){\tiny${\mathfrak p}\,\Ram$}
\put(7.85,1.6){\tiny${\ov {\mathfrak p}}\,\Ram$}
\put(3.6,4.2){\line(1,-3){0.5}}
\put(3.4,4.2){\line(-1,-5){0.32}}
\put(3.0,2.35){$L$}
\put(4.05,2.46){$\ov L$}
\put(3.1,2.3){\line(1,-2){0.47}}
\put(4.1,2.4){\line(-1,-3){0.34}}
\put(3.62,1.05){$\knr$}
\put(3.7,1.0){\line(1,-2){0.44}}
\put(4.05,-0.1){$k$}
\put(3.7,4.3){\line(1,-1){1.62}}
\put(7.15,4.2){\line(1,-1){1.4}}
\put(4.8,4.24){\tiny${\mathfrak p}, {\ov {\mathfrak p}}\ {\rm totally\ split}$}
\put(4.7,0.9){\tiny${\mathfrak p}, {\ov {\mathfrak p}} \,\Nram$}
\put(3.75,0.55){\tiny${\mathfrak p}, {\ov {\mathfrak p}} \,\Nram$}
\put(2.0,3.7){\tiny${\mathfrak p}\,\Nram$}
\put(1.85,3.5){\tiny${\ov {\mathfrak p}}\,\Nram$}
\put(3.96,3.2){\tiny${\mathfrak p}\,\Nram$}
\put(4.06,3.0){\tiny${\ov {\mathfrak p}}\,\Ram$}
\put(2.58,3.2){\tiny${\mathfrak p}\,\Ram$}
\put(2.58,3.0){\tiny${\ov {\mathfrak p}}\,\Nram$}
\put(1.8,1.4){\tiny${\mathfrak p}\,\Ram$}
\put(2.05,1.2){\tiny${\ov {\mathfrak p}}\,\Ram$}
\put(4.35,3.7){\tiny${\ov {\mathfrak p}}\,\Nram$}
\put(4.59,3.5){\tiny${\mathfrak p}\,\Nram$}
\put(4.65,1.6){\tiny${\ov {\mathfrak p}}\,\Ram$}
\put(4.85,1.8){\tiny${\mathfrak p}\,\Ram$}
\put(3.75,1.8){\tiny${\ov {\mathfrak p}}\,\Nram$}
\put(3.85,2.0){\tiny${\mathfrak p}\,\Ram$}
\put(2.9,1.8){\tiny${\ov {\mathfrak p}}\,\Ram$}
\put(2.85,2.0){\tiny${\mathfrak p}\,\Nram$}
\bezier{300}(4.35,-0.1)(6.0,0.0)(6.9,0.95)
\put(6.1,0.2){$\CH_k$}
\bezier{300}(3.9,-0.1)(1.8,0.2)(1.35,2.2)
\put(1.85,0.45){$\Z_p$}
\bezier{300}(4.05,0.0)(2.9,1.1)(3.0,2.3)
\put(2.82,1.3){\ft$\Z_p$}
\bezier{300}(1.35,2.5)(1.5,4.2)(3.35,4.45)
\put(1.5,3.8){$\Z_p$}
\bezier{300}(3.75,4.4)(5.7,4.1)(5.45,2.7)
\put(5.25,3.8){$\Z_p$}
\bezier{300}(4.3,0.0)(6.2,0.7)(5.45,2.3)
\put(5.4,0.45){$\Z_p$}
\end{picture}}} 
\end{aligned}
\end{equation}

\subsubsection{General diagram including a \texorpdfstring{$\Z_p$}{Lg}-extension 
\texorpdfstring{$K/k$, $K := \bigcup_{n \geq 0} K_n$}{Lg}}

In this paper we always assume $K \ne L, \ov L$ since in these $\Z_p$-extensions, 
only a single $p$-place ramifies, thus giving rather simpler properties (see \cite{FK2002} 
for numerical experiments); this is equivalent to the finiteness of $K \cap L/k$ and  
$K \cap \ov L/k$, characterizing the ramification of the two $p$-places
(see Theorem \ref{6cases} for details; the ramification indices strongly depend
on the intersections of $K$ with $k^\cyc$ and $k^\acyc$, and of course with $L$
and $\ov L$):
\unitlength=1.4cm 
\begin{equation}\label{maindiagram}
\begin{aligned}
\vbox{\hbox{\hspace{-1.8cm}
\begin{picture}(8.0,7.2)
\put(4.4,1.8){\line(1,0){2.5}}
\put(5.45,1.85){$\simeq\!\CT_k$}
\put(1.6,0.46){\line(1,0){1.46}}
\put(3.45,6.40){$\wt k$}
\put(3.8,6.45){\line(1,0){3.1}}
\bezier{300}(3.6,6.7)(5.3,7.0)(7.0,6.7)
\put(5.15,6.95){$\CT_k$}
\put(3.50,4.15){\line(0,1){2.18}}
\put(3.50,2.0){\line(0,1){1.75}}
\put(3.50,0.65){\line(0,1){1.0}}
\put(3.35,3.86){$K$}
\put(3.15,1.73){$K\!\!\cap\!\!L$}
\put(3.15,0.40){$K\!\!\cap \!\!\ov L\!=\! K^\nr$}
\put(1.4,0.42){$k$}
\bezier{300}(1.55,0.6)(1.9,1.8)(3.15,1.8)
\bezier{300}(1.6,0.52)(2.3,1.3)(4.15,1.75)
\put(2.48,1.3){\ft $p^{\wt e}$}
\put(1.74,1.4){\ft $p^e$}
\put(7.0,1.7){$H_k^\nr$}
\put(6.95,6.35){$\wt k H_k^\nr \!=\! H_k^\pr$}
\put(4.8,6.24){\tiny${\mathfrak p}, {\ov {\mathfrak p}}\ {\rm totally\ split}$}
\put(1.3,4.35){$L$}
\put(5.5,4.4){$\ov L$}
\put(1.55,4.5){\line(1,1){1.78}}
\put(1.6,4.3){\line(1,-1){1.72}}
\put(3.65,2.3){\line(1,-1){0.4}}
\put(3.7,0.78){\line(1,2){0.4}}
\put(4.33,2.1){\line(1,2){0.6}}
\put(5.1,3.55){\line(1,2){0.4}}
\put(4.15,1.7){$\knr$}
\put(3.7,6.3){\line(1,-1){1.65}}
\put(1.54,4.25){\line(3,-4){1.7}}
\put(2.1,5.7){\tiny${\mathfrak p}\,\Ram$}
\put(1.85,5.5){\tiny${\ov {\mathfrak p}}\,\Nram$}
\put(4.0,4.5){\tiny${\mathfrak p}, {\ov {\mathfrak p}}$}
\put(4.0,4.3){\tiny$\Nram$}
\put(3.0,5.2){\tiny${\mathfrak p}, {\ov {\mathfrak p}}$}
\put(3.0,5.0){\tiny$\Nram$}
\put(2.25,3.8){\tiny${\mathfrak p}\,\Nram$}
\put(2.5,3.55){\tiny${\ov {\mathfrak p}}\,\Ram$}
\put(4.4,5.7){\tiny${\ov {\mathfrak p}}\,\Ram$}
\put(4.65,5.5){\tiny${\mathfrak p}\,\Nram$}
\put(5.3,3.6){\tiny${\ov {\mathfrak p}}\,\Nram$}
\put(5.4,3.8){\tiny${\mathfrak p}\,\Ram$}
\put(2.3,2.2){\tiny${\mathfrak p}\,\Nram$}
\put(2.1,2.4){\tiny${\ov {\mathfrak p}}\,\Ram$}
\put(3.54,3.0){\tiny${\mathfrak p}, {\ov {\mathfrak p}}$}
\put(3.54,2.8){\tiny$\Ram$}
\put(2.0,0.54){\tiny${\mathfrak p}, {\ov {\mathfrak p}}\,\Nram$}
\put(3.96,1.15){\tiny${\mathfrak p}, {\ov {\mathfrak p}}$}
\put(3.96,0.95){\tiny$\Nram$}
\put(2.8,1.1){\tiny${\mathfrak p}\,\Nram$}
\put(2.8,0.9){\tiny${\ov {\mathfrak p}}\,\Ram$}
\put(5.3,1.6){\tiny${\mathfrak p}, {\ov {\mathfrak p}}\,\Nram$}
\put(2.3,0.24){\tiny$p^{\ov e}$}
\bezier{400}(1.5,0.35)(4.0,-0.5)(7.1,1.56)
\put(5.2,0.35){$\CH_k$}
\put(3.6,6.3){\line(1,-2){1.8}}
\put(7.2,6.3){\line(1,-2){1.7}}
\put(7.5,1.76){\line(3,2){1.05}}
\put(8.6,2.54){$H'^\pr_k$}
\put(5.8,2.6){\line(1,0){2.7}}
\put(6.65,2.7){\tiny${\mathfrak p}, {\ov {\mathfrak p}}\,\Nram$}
\put(7.3,4.4){\tiny${\ov {\mathfrak p}}\,\Nram$}
\put(7.5,4.2){\tiny${\mathfrak p}\,\Nram$}
\put(8.05,2.0){\tiny${\mathfrak p}\,{\ov {\mathfrak p}} \,\Ram$}
\put(5.4,2.52){$k^\acyc$}
\put(4.6,2.0){\line(3,2){0.75}}
\put(5.,2.15){\tiny${\mathfrak p},{\ov {\mathfrak p}}\,\Ram$}
\end{picture}}} 
\end{aligned}
\end{equation}
\unitlength=1.0cm

Suppose that $K \cap \ov L \subseteq K \cap L$.  
Since $L \cap \ov L = \knr$, $K \cap L \cap \ov L = K \cap \ov L = 
K \cap \knr =: K^\nr$ is a subfield of $k^\acyc$. 
Put $[K \cap \ov L : k] = p^{\ov e}$, $[K \cap L : k] = p^e$, $e \geq \ov e$
(whence $K \cap \ov L = K_{\ov e}$ and $K \cap L = K_e$, $\ov e \leq e$). 
Since $H_k^\pr$ is the direct compositum of $H_k^\nr$ and $\wt k$ 
over $\knr$, we have:
$$\order \CH_k = [\knr  : k] \times \order \CT_k
= p^{\ov e} \times [\knr : K^\nr] \times \order \CT_k = p^{\wt e}
\times \order \CT_k. $$
We refer to Appendix \ref{B} to justify this diagram, about
the description of the ramification of ${\mathfrak p}$ and 
${\ov {\mathfrak p}}$ (i.e. the computation of the integers $e$, 
$\ov e$ from the parametrization of $\Gal(\wt k/K)$).

\section{The logarithmic class group \texorpdfstring{$\wt \CH_k$}{Lg}}
\label{clg}

Recall, from Section \ref{why}, that a suitable interpretation is
$\wt \CH_k \simeq \Gal(H_k^\lc/k^\cyc)$, where $H_k^\lc$ is the maximal 
abelian pro-$p$-extension of $k$, locally cyclotomic. For definitions and 
properties of $\wt \CH_k$, as group of classes of divisors,
see Appendix \ref{A} and Jaulent's papers of the 
bibliography.\,\footnote{\,\label{wtH=H}For numerical computations; the {\sc pari/gp} 
instruction ${\sf bnflog(k,p)}$ gives the vector 
${\sf \big[\wt \CH_k,\, \langle \wt \Ccl(S_k) \rangle,\, \CH_k^{S_k}}\big ]$, 
where $\langle \wt \Ccl(S_k) \rangle$ is the subgroup of $\wt \CH_k$ generated 
by the logarithmic classes of $S_k$, and $\CH_k^{S_k} = 
\CH_k/\langle \Ccl(S_k)\rangle$. These invariants are linked by the canonical 
exact sequence $0 \to \wt \Ccl(S_k) \too \wt \CH_k \too \CH_k^{S_k} \to 1$
allowing capitulations comparison between $\wt \CH_k$ and $\CH_k$; if
$\wt \Ccl(S_k) = \Ccl(S_k) = 1$, then $\wt \CH_k = \CH_k$, which is 
satisfied for instance when $p$ does not split in $k$ (real or imaginary).
See also \cite{DJPS2005} and \cite{BJ2016}.
In the imaginary split case, $\wt \Ccl(S_k)$, $\Ccl(S_k)$ are independent; thus 
$\order \langle \wt \Ccl(S_k) \rangle = p^{\delta_p(k)}$ from Relation \eqref{firstclog}.}

\subsection{A characteristic property of \texorpdfstring{$\wt \CH_k$}{Lg} in abelian 
\texorpdfstring{$p$}{Lg}-ramification}

A criterion using $\wt \CH_k$ in the case of totally $p$-adic imaginary quadratic 
fields is the following. Denote by $\CC_{\,\wt k}$ the $p$-class group of 
$\wt k$, defined as the projective limit, for the norms, of the $p$-class groups 
$\CH_F$, for families of finite extensions $F/k$ such that $\wt k = \cup F$. In 
\cite[Proposition 12]{Jau2024a}, a general result claims that 
$\CC_{\,\wt k} = 1$ if and only if $\wt \CH_k = 1$. Assume that 
this property is satisfied; this implies in particular that for $k^\cyc$, 
the compositum of $\wt k$ with the Hilbert class field of the $n^{\rm th}$ 
layer $k_n^\cyc$ of $k^\cyc$ is included in $\wt k$, for all $n \geq 0$.
Since the extension $\wt k/k^\cyc$ is unramified and pro-cyclic
of Galois group isomorphic to $\Z_p$, this proves that 
$\lambda_p(k^\cyc/k) = 1$ and $\mu_p(k^\cyc/k) = 0$. 

\subsection{Decomposition of the \texorpdfstring{$p$}{Lg}-places
in \texorpdfstring{$H_k^\pr/k$}{Lg}} \label{delta239}

The following diagram introduces $H_k^\lc$ for totally $p$-adic imaginary 
quadratic fields $k = \Q(\sqrt{-m})$.\label{split}
In the unramified abelian extension $H_k^\pr/k^\cyc$ the decomposition groups 
$D_{\mathfrak p}$, $D_{\ov {\mathfrak p}}$ are pro-cyclic groups fixing $M$, 
$\ov M$, respectively, such that $M \cap \ov M = H_k^\lc$, $M \ov M = F$. 
Necessarily, we have $\Gal(M/\wt k \cap H_k^\lc) \simeq \Gal(\ov M/\wt k 
\cap H_k^\lc) \simeq \CT_k$, to get total splitting in $H_k^\pr/ \wt k$:
\unitlength=0.85cm 
\begin{equation}\label{hlc}
\begin{aligned}
\vbox{\hbox{\hspace{-1.8cm} 
\begin{picture}(10.0,6.55)
\put(3.8,5.9){\line(1,0){4.3}}
\put(3.95,3.6){\line(1,0){4.3}}
\put(4.2,1.0){\line(1,0){1.25}}
\put(3.15,3.5){$\wt k \cap\! F$}
\put(8.35,3.5){$F$}
\put(5.65,6.46){\ft$\CT_k$}
\bezier{350}(3.8,6.2)(6.0,6.45)(8.2,6.2)
\bezier{350}(5.9,2.6)(6.2,4.8)(8.3,5.8)
\bezier{350}(8.35,2.2)(10.7,3.2)(8.6,5.75) 
\put(5.6,4.6){\ft$D_{\ov {\mathfrak p}}\!\simeq \! \Z_p$}
\put(9.65,3.7){\ft$D_{\mathfrak p}\!\simeq \! \Z_p$}
\put(3.50,1.2){\line(0,1){2.25}}
\put(3.50,3.8){\line(0,1){1.95}}
\put(8.45,3.8){\line(0,1){1.95}}
\put(7.6,4.78){\tiny$p$-{\rm inert}}
\put(6.1,1.05){\line(2,1){2.0}}
\put(8.0,2.1){\ft$M$}
\put(5.75,2.25){\ft$\ov M$}
\put(6.1,2.4){\line(2,1){2.1}}
\put(5.65,1.2){\line(1,5){0.2}}
\put(8.2,2.4){\line(1,5){0.2}}
\put(8.3,5.9){$H_k^\pr\!=\!\wt k F$}
\put(3.35,5.9){$\wt k$}
\put(1.4,0.9){$k^\cyc$}
\put(1.8,1.0){\line(1,0){1.25}}
\put(2.1,0.8){\tiny$p$-{\rm split}}
\bezier{200}(1.5,0.8)(3.57,0.2)(5.65,0.8)
\put(3.2,0.10){\ft$\simeq \! \wt \CH_k$} 
\put(5.5,0.9){$H_k^\lc$}
\put(3.15,0.9){$\wt k \!\cap\! H_k^\lc$}
\put(2.8,4.78){\ft$\Z_p$}
\bezier{200}(3.35,5.8)(3.0,4.9)(3.35,4.0)
\bezier{400}(3.2,5.8)(1.75,3.5)(3.3,1.3)
\bezier{400}(1.34,1.1)(0.75,4.5)(3.3,6.0)
\put(0.8,2.8){\ft$\Z_p$}
\put(3.6,4.78){\tiny$p$-{\rm inert}}
\put(5.55,5.65){\tiny$p$-{\rm split}}
\put(4.45,0.8){\tiny$p$-{\rm split}}
\put(4.6,3.75){\tiny$p$-{\rm split}}
\put(6.2,2.85){\tiny${\ov {\mathfrak p}}$-{\rm inert}}
\put(6.2,3.07){\tiny${{\mathfrak p}}$-{\rm split}}
\put(8.3,2.65){\tiny${\mathfrak p}$-{\rm inert}}
\put(8.3,2.87){\tiny${\ov {\mathfrak p}}$-{\rm split}}
\put(5.0,1.77){\tiny${\ov {\mathfrak p}}$-{\rm split}}
\put(5.0,1.55){\tiny${{\mathfrak p}}$-{\rm inert}}
\put(7.0,1.32){\tiny${\mathfrak p}$-{\rm split}}
\put(7.0,1.1){\tiny${\ov {\mathfrak p}}$-{\rm inert}}
\put(2.9,2.8){\tiny{\rm finite}}
\put(2.9,2.6){\tiny{\rm extension}}
\put(3.6,2.2){\tiny$p$-{\rm inert}}
\put(2.1,3.55){\ft$\Z_p$}
\end{picture}}} 
\end{aligned}
\end{equation}
\unitlength=1.0cm

Numerical examples, given by the {\sc pari/gp}  
Program \ref{P1}, may illustrate the above diagram; they
show that $\CT_k$ (in ${\sf T_k}$) may be trivial, contrary to $\CH_k$
(in ${\sf H_k}$) and to the logarithmic class group  $\wt \CH_k$ 
(in ${\sf Clogk[1]}$).
When $\CT_k = 1$ (i.e. $k$ is $p$-rational), $H_k^\lc \subset \wt k$,
$\wt \CH_k$ is cyclic of order $p^{\wt \delta_p(k)} = 
p_{}^{\delta_p(k)+v_p(\hk) - v_p(\hp)}$ 
(Definition \ref{wtdeltapk}); 
moreover, the $p$-Hilbert class field $H_k^\nr$ is contained in 
$k^\acyc$, and $\CH_k$ is cyclic. Illustrations of $p$-rationality, 
with non-trivial and non-necessarily isomorphic $\wt \CH_k$ and $\CH_k$, 
are (with $\delta_p(k)$, $\wt \delta_p(k)$ in ${\sf d(x)}$, ${\sf d(X)}$, 
respectively):

\ft\begin{verbatim}
m=239   p=3  Clogk=[[3],[],[3]]    Tk=[]  Hk=[3]   d(x)=0 d(X)=1 
m=7727  p=3  Clogk=[[3],[],[3]]    Tk=[]  Hk=[81]  d(x)=0 d(X)=1
m=14207 p=3  Clogk=[[9],[],[9]]    Tk=[]  Hk=[81]  d(x)=0 d(X)=2
m=1238  p=3  Clogk=[[9],[3],[3]]   Tk=[]  Hk=[3]   d(x)=1 d(X)=2 
m=2239  p=5  Clogk=[[25],[5],[5]]  Tk=[]  Hk=[5]   d(x)=1 d(X)=2 
\end{verbatim}\ns
 
\begin{remark}
The structures of the invariants $\wt \CH_k$, $\CT_k$ and $\CH_k$ are most often
obvious from the data; but some examples are interesting as the following one:

\ft\begin{verbatim}
m=78731 p=3  Clogk=[[6561,3],[6561],[3]] Tk=[27]  Hk=[27]  d(x)=8 d(X)=9
\end{verbatim}\ns

\noindent
for which $\Bcl({\mathfrak p})$ is of order $9$; surprisingly, ${\mathfrak p}^9 = 
\frac{1}{2}(1 + \sqrt{-78731})$, which is exceptional since, in general, coefficients 
may have thousands of digits. Since $\wt \CH_k \simeq \Z/3^8\Z \times \Z/3\Z$, 
$H_k^\lc$ can not be contained in $\wt k$; so $[H_k^\lc : \wt k \cap H_k^\lc] =: 
3^r$, $r \geq 1$, and $\wt \CH_k/\Gal(H_k^\lc / \wt k \cap H_k^\lc)$ must 
be cyclic of order $3^{9-r}$. Two cases may occur:

\smallskip
$\bullet$ $H_k^\pr = \wt k H_k^\lc$. In this case, $r = 3$, $M = \ov M = H_k^\lc = F$,
and $\Gal(\wt k \cap H_k^\lc / k^\cyc)$ must be a cyclic quotient of order 
$3^6$ of $\Z/3^8\Z \times \Z/3\Z =: \langle a, b \rangle$
($a$ of order $3^8$, $b$ of order~$3$); the unique solution is the 
quotient by $\langle a^{3^5}\! \cdot b \rangle$ (of order~$3^3$), 
giving a cyclic group of order $3^6$.

\smallskip
$\bullet$ $[H_k^\pr : \wt k H_k^\lc] \geq 3$. In this case, $r \leq 2$, and $3$ 
can not be totally inert in $H_k^\pr/H_k^\lc$; so, this suggests that the 
two  decomposition groups of ${\mathfrak p}$ and ${\ov {\mathfrak p}}$ do not 
coincide in $\Gal(H_k^\pr/H_k^\lc)$, but are of index a suitable power of $3$ with
cyclicity of $M$ and $\ov M$ over $\wt k \cap H_k^\lc$ since $\CT_k \simeq \Z/27\Z$
(see the diagram where all possibilities are represented). From the forthcoming 
Theorem \ref{fq} and numerical results (e.g., $\Ccl(S_k) \simeq \Z/9\Z$ in 
$\Gal(H_k^\nr/k)$ from projections of $\Gal(\ov M/H_k^\lc)$ and $\Gal(M/H_k^\lc)$), 
one deduces $r = 1$, $[\wt k \cap H_k^\lc : k^\cyc] = 3^8$.                                                     

\smallskip
This situation does exist, a priori, for all $p$; for instance with $p = 5$ one gets the case:

\ft\begin{verbatim}
m=30911   p=5   Clogk=[[125,5],[125],[5]]   Tk=[25]   Hk=[25]
\end{verbatim}\ns

Even if most examples give either cyclic $\wt \CH_k$ or else bi-cyclic ones, we have 
found some examples of rank $3$, restricting the search with $p = 3$; of course, 
this implies that $\CT_k$ is at least of rank $2$. Numerical results are given by 
Program \ref{P2}.
\end{remark}

\subsection{Order of \texorpdfstring{$\wt \CH_k$}{Lg} -- Fermat quotients of the
\texorpdfstring{$S_k$}{Lg}-units \texorpdfstring{$\ov x$, $\ov X$}{Lg}}

Let's assume $p \geq 3$, split in an {\it arbitrary imaginary quadratic field} 
$k = \Q(\sqrt{-m})$.
Our purpose was to give a class field theory computation (from 
a method we have used in \cite{Gra2017b, Gra2019a, Gra2019b}) of the order 
of the Hasse norm residue symbols $\big(\fffrac{x\,,\,K_n/k}{\mathfrak p} \big)$,
$\big(\fffrac{x\,,\,K_n/k}{\ov {\mathfrak p}} \big)$ for any $\Z_p$-extension 
$K = \bigcup_n K_n$ of $k$, where the $S_k$-unit $x$ is the generator of 
${\mathfrak p}^{\hp}$, where $\hp$ is the 
common order of $\Bcl({\mathfrak p})$ and $\Bcl({\ov {\mathfrak p}})$ in 
$\BH_k$.\footnote{\label{note} The global Hasse norm residue symbols are
defined from the local ones, given by the local reciprocity map of local class field
theory; the link between Artin reciprocity law was made by
Hasse as explained by Peter Roquette (2006) in the summary of a lecture:

{\it In 1927 Artin proved his general reciprocity law which admitted a completely new perspective on class field theory. Five years later, in 1932, Hasse succeeded to give a proof of Artin's law based on a local-global principle; this paved the way to various generalizations which are investigated today. We shall report on the development in the years between 1927 and 1932, from Artin's law to Hasse's result. Our main source are the letters which had been exchanged between Hasse and Artin in that time period, mirroring the excitement which prevailed among our protagonists. Emmy Noether and Richard Brauer too were involved in this development.}

The interest of Artin symbols is that they are products of familiar Frobenius automorphisms.}

\smallskip
We show, Section \ref{hasse}, that the order of these 
symbols especially depends on the integer $\delta_p(k)$ 
which intervenes in the higher-rank Chevalley--Herbrand formulas, in direction 
of a Gold criterion. We have observed a remarkable numerical coincidence:
more precisely, consider, instead of $x$, the $S_k$-unit $X$ defined by 
${\mathfrak p}^{\hk} =: (X)$; then, from $X = x^{\hk/\hp}$:
$$\wt \delta_p(k) := v_{\mathfrak p}
\big (\fffrac{\ov X^{\,p-1} \!-1}{p}\big) = v_{\mathfrak p}
\big (\fffrac{\ov x^{\,p-1} \!-1}{p}\big)
\! +\! \big[v_p(\hk) - v_p(\hp) \big] = 
\delta_p(k)\! +\! \big[v_p(\hk) - v_p(\hp) \big]$$

\noindent
is the valuation of the Fermat quotient of $\ov X$ (Definition \ref{wtdeltapk}). 
Then $p^{\wt \delta_p(k)}$ is quite simply the order of the logarithmic class 
group $\wt \CH_k$ for totally $p$-adic imaginary quadratic fields.
Using the techniques of logarithmic class field theory, Jaulent proves a 
more general framework where this property holds for semi-simple totally 
$p$-adic imaginary abelian number fields (see Appendix \ref{A}, 
Theorem \ref{JFJ}, Corollary \ref{qc}), developed in the direction of a 
Gold version about $\wt \CH_k$ in \cite[Th\'eor\`emes 2, 4]{Jau2024b}:

\begin{theorem}\label{fq}
Let $k$ be an imaginary quadratic field of class number $\hk$, 
and let $p \geq 3$ be a prime number split in $k$ into $(p) = {\mathfrak p}
{\ov {\mathfrak p}}$. 
Let's define $\wt \delta_p(k) := \delta_p(k) + \big[v_p(\hk) 
- v_p(\hp) \big]$, where $\hp \mid \hk$ is the 
common order of $\Bcl({\mathfrak p})$ and $\Bcl({\ov {\mathfrak p}})$ 
in $\BH_k$. Then:
$$\order \wt \CH_k = p^{\wt \delta_p(k)} = p^{\delta_p(k)} \times 
\ffrac{\order \CH_k}{\order \langle\,\Ccl({\mathfrak p})\,\rangle} =
p^{\delta_p(k)} \times \order \CH_k^{S_k}. $$

So, $\wt \CH_k = 1$ if and only if  $\delta_p(k) = 0$ and
$\CH_k = \langle\,\Ccl({\mathfrak p})\,\rangle$. 
\end{theorem}

\begin{corollary}
(i) For non-cyclic $\CH_k$'s, necessarily $\CH_k^{S_k} \ne 1$; then the 
logarithmic class group of $k$ is non-trivial.

\smallskip
(ii) If $K/k$ is a $\Z_p$-extension (necessarily distinct from $k^\cyc/k$) 
and if there exists a layer $K_s$, $s \geq 1$, in which $p$ totally split, 
then $\wt \CH_k \ne 1$ since $[H_k^\lc : k^\cyc ] \geq p^s$.
\end{corollary}

\begin{remark}\label{borel-cantelli}
For $p$ split in $k$, $\order \wt \CH_k$ is given by means of the valuation 
of the Fermat quotient of an algebraic number,
but depending on $p$; indeed, $\fffrac{1}{p} (\ov x^{\,p-1} - 1)$ is not 
a ``standard'' Fermat's quotient, compared with $\fffrac{1}{p} (a^{p-1} - 1)$, 
for the fixed prime-to-$p$ algebraic number $a$. 

\smallskip
Thus, at first glance, it seems difficult to apply the $p$-adic 
heuristics worked out in \cite[Section~8]{Gra2016} and to extend, to totally real 
number fields $k$, the conjecture \cite[Conjecture 11]{Jau2017} stated for the 
real quadratic case. In a numerical viewpoint, we see that for $k = \Q(\sqrt{-3})$, 
one finds the following large list of primes $p$, such that $\delta_p(k) \ne 0$:
$p \in$ $\{$$13$, $181$, $2521$, $76543$, $489061$, $6811741$, $1321442641$, 
$18405321661$$\}$, but nothing else after three days of calculations. 
For $k = \Q(\sqrt{-5})$, one obtains the unique solution $p = 5881$, up 
to $10^9$. The case of $ \Q(\sqrt{-47})$, also giving many small solutions 
$p \in$ $\{$$3$, $17$, $157$, $1193$, $1493$, $1511$, $19564087$$\}$, 
doesn't give more up to $10^9$ (see Program \ref {P5}).

\smallskip
So this raises the question of the finiteness or not of the set of solutions $p$.
If $x = u+v \sqrt{-m}$ in $k = \Q(\sqrt{-m})$, then $\ov x = u-v\sqrt{-m} 
\equiv 2u \pmod {{\mathfrak p}^{\hp}}$, and $\delta_p(k)$ 
depends on the $p$-valuation of the Fermat quotient $q = \ffrac{a^{p-1} - 1}{p}$ 
of $a = 2u$, when $\hp \geq 2$ ($\hp = 1$ reduces, after some 
elementary computations, to the $p$-valuation of $(2u)^2 q +1 = a^2 q+1$).

\smallskip
Such a non-zero valuation is in general related to probability $\frac{1}{p}$, the 
Borel--Cantelli heuristic giving, a priori, infiniteness of solutions $p$; but the trace 
$2u$ is subject to the strong relationship $(u+v \sqrt{-m}) = 
{\mathfrak p}^{\hp}$, 
which does not define uniformly random coefficients contrary to the probability 
draw of a solution $(U_0, V_0)$ of the congruence $U+V \sqrt{-m} \equiv 0 \pmod 
{{\mathfrak p}^{\delta_p(k)+1}}$ giving the same Fermat quotient (among 
$p^{\delta_p(k)+1}$ ones, characterized by $U_0$ since $\sqrt{-m} \in \Z_p^\times$). 
So we may think that the corresponding probability of $\delta_p(k) \ne 0$ is 
much less than $\frac{1}{p}$, giving infiniteness; in other
words, $(u+v \sqrt{-m}) = {\mathfrak p}^{\hp}$ is a global fact,
while $U+V \sqrt{-m} \equiv 0 \pmod {{\mathfrak p}^{\delta_p(k)+1}}$ is 
a local one without any mystery.

\smallskip
Nevertheless, we suggest the following conjecture, for any number field $k$;
this is coherent with our similar conjectures about the $\CT_k(p)$'s 
(\cite[Conjecture 8.11]{Gra2016},\cite{Gra2019a}):
\end{remark}

\begin{conjecture}\label{pfinitude}
For any fixed number field $k$, satisfying for all $p$ the Leopoldt and Gross--Kuz'min 
conjectures, the set of primes $p$ such that $\wt \CH_k(p) \ne 1$ is finite.
In other words, the invariants $\CH_k(p)$, $\wt \CH_k(p)$, $\CT_k(p)$, would define 
finite global invariants of $k$.
\end{conjecture}

If Leopoldt and Gross--Kuz'min conjectures are of transcendental nature, the 
existence of the previous global invariants (only known for class groups, from a
non-algebraic proof) are much less accessible and perhaps are ``folk conjectures''; 
but computational experiments suggest this, even if it is non-obvious to characterize 
many $(k, p)$ with trivial invariants (e.g., for the $p$-rationality, regarding the
$abc$-conjecture, cf. \cite{MR2019a, MR2019b}).

\section{Hasse's norm residue symbols in 
\texorpdfstring{$\wt k/k$}{Lg}} \label{hasse}

Let $K_n$ be the $n$th layer of a $\Z_p$-extension $K/k$. For $a \in k^\times$ and 
for any unramified prime ideal ${\mathfrak q}$, the Hasse norm residue symbol 
$\big(\fffrac{a\,,\,K_n/k}{{\mathfrak q}} \big)$ is given by the Artin symbol 
$\big(\fffrac{K_n/k}{{\mathfrak q}} \big)^{\!v_{\mathfrak q}(a)}$, power of a
Frobenius automorphism; but we intend to compute the order of the less 
obvious symbol $\big(\fffrac{x\,,\,K_n/k}{\mathfrak p} \big)$, for 
the $S_k$-unit $x$ given in Definition \ref{wtdeltapk}, when $k = \Q(\sqrt{-m})$
with $p$ split in $k$ (see Footnote \ref{note} for some history).

\subsection{Expression of \texorpdfstring{$\big(\fffrac{x\,,\,{\wt k}/k}
{\mathfrak p} \big)$ via an Artin symbol}{Lg}}

Recall a method to compute $\big(\fffrac{x\,,\,{\wt k}/k}
{\mathfrak p} \big)$, for $p$ split into $(p) = {\mathfrak p} {\ov {\mathfrak p}}$,
when $x$ is the fundamental ${\mathfrak p}$-unit of $k$.
Put ${\mathfrak p}^{\hp} =: (x)$ for the order 
$\hp$ of $\Bcl({\mathfrak p})$. 
Of course, $\big(\fffrac{\bullet\,,\,{\wt k}/k}{\bullet} \big)$ will be considered
as the pro-limit for $N \to \infty$ of $\big(\fffrac{\bullet\,,\,{F_N}/k}{\bullet} \big)$, 
where $F_N$ is for instance the intersection of $\wt k$ with the ray class field
$k(p^N)$ of modulus $(p^N)$.

\smallskip
Let $y^{}_{N,\mathfrak p} \in k^\times$ (called a ${\mathfrak p}$-associate of $x$) 
satisfying the following conditions:

\smallskip
\quad (i) $y^{}_{N,\mathfrak p} x^{-1} \equiv 1 \pmod {{\mathfrak p}^N}$,

\smallskip
\quad (ii) $y^{}_{N,\mathfrak p} \equiv 1 \pmod {{\ov {\mathfrak p}}^N}$.

\smallskip
For $N$ large enough, 
$v_{\mathfrak p}(y^{}_{N,\mathfrak p}) = v_{\mathfrak p}^{}(x) = 
\hp$. Product Formula for $y^{}_{N,\mathfrak p}$ yields:
$$\big(\fffrac{y^{}_{N,\mathfrak p}\,,\,F_N/k}{\mathfrak p} \big) = 
\hbox{$\prd_{{\mathfrak q} \,,\, {\mathfrak q} \,\ne\, {\mathfrak p}}$}
\big (\fffrac{y^{}_{N,\mathfrak p}\,,\, F_N/k}{{\mathfrak q}} \big)^{-1}, $$

\noindent  
and since $\big( \fffrac {x\,,\, F_N/k}{\mathfrak p} \big) = 
\big(\fffrac {y^{}_{N,\mathfrak p}\,,\, F_N/k}{\mathfrak p} \big)$
from (i), by definition of the local norm ${\mathfrak p}$-conductor 
of $F_N/k$, which divides $p^N$, then
$\big(\fffrac{x\,,\, F_N/k}{\mathfrak p} \big) = \prd_{{\mathfrak q} \,,\, 
{\mathfrak q} \,\ne \, {\mathfrak p}}
\big (\fffrac{y^{}_{N,\mathfrak p}\,,\, F_N/k}{{\mathfrak q}} \big)^{-1}$.
Let's compute these symbols:

$\quad\bullet\ $ 
If ${\mathfrak q} = {\ov {\mathfrak p}}$, since $y^{}_{N,\mathfrak p} 
\equiv 1 \pmod {{\ov {\mathfrak p}}^N}$, one has 
$\big(\fffrac {y^{}_{N,\mathfrak p}\,,\, F_N/k}{{\ov {\mathfrak p}}} \big) = 1$,

$\quad\bullet\ $ 
If ${\mathfrak q} \nmid p$ being unramified, 
$\big (\fffrac{y^{}_{N,\mathfrak p}\,,\, F_N/k}{{\mathfrak q}} \big) =
\big (\fffrac{F_N /k}{{\mathfrak q}} \big)^{v_{\mathfrak q}(y^{}_{N,\mathfrak p})}$. 

Finaly, since $v_{\ov{\mathfrak p}}(y^{}_{N,\mathfrak p}) = 0$,
$\big (\fffrac{x\,,\, F_N/k}{\mathfrak p} \big) = \prd_{{\mathfrak q}\, \nmid\, p}
\big (\fffrac{F_N /k}{{\mathfrak q}} \big)^{-v_{\mathfrak q}(y^{}_{N,\mathfrak p})}
=: \big (\fffrac{F_N /k}{{\mathfrak a}_{N,\mathfrak p}} \big)^{-1}$,  
inverse of the Artin symbol of the principal ideal:
\begin{equation}\label{ideal}
{\mathfrak a}_{N,\mathfrak p} :=
\prd_{{\mathfrak q} \,\nmid\, p} {\mathfrak q}^{v_{\mathfrak q}(y^{}_{N,\mathfrak p})}
= (y^{}_{N,\mathfrak p})\,{\mathfrak p}^{-\hp} = 
\big (y^{}_{N,\mathfrak p} \cdot x^{-1} \big). 
\end{equation}

\subsection{Computation of the orders of \texorpdfstring{$\big(\frac{x\,,\,K_n/k}
{\mathfrak p} \big)$}{Lg} and 
\texorpdfstring{$\big(\frac{x\,,\,{K_n}/k}{\ov {\mathfrak p}} \big)$}{Lg}}

Let $K = \cup_n K_n$ be a $\Z_p$-extension of $k = \Q(\sqrt{-m})$ 
and let $L$, $\ov L$, be the inertia fields defined in Diagrams \ref{ramification}, 
\ref{maindiagram}. We only assume that $K$ is distinct from $L$ and $\ov L$
(otherwise, a single $p$-place does ramify in $K/k$ and many simplifications
arize in this case). The principle is to show that these orders can be deduced 
from that of $\big(\fffrac{x\,,\,{k^\cyc_n}/k}{\mathfrak p} \big)$. The following 
lemma simplifies the relations between all possible choices for $x$, $\ov x$,
${\mathfrak p}$, ${\ov {\mathfrak p}}$ and the corresponding symbols; it uses
complex conjugation and the general formula 
$\big(\fffrac{z\,,\,M/F}{\mathfrak q} \big)^\tau := 
\tau \circ \big(\fffrac{z\,,\,M/F}{\mathfrak q} \big) \circ \tau^{-1} =
\big(\fffrac{z^\tau\,,\,M^\tau/F^\tau}{{\mathfrak q}^\tau} \big)$, which reduces
to $\big(\fffrac{z^\tau\,,\,M/F}{{\mathfrak q}^\tau} \big)$ in the absolute Galois 
cases:

\begin{lemma}\label{ordercyc} 
$\big(\fffrac{p\,,\,k^\cyc/k}{\mathfrak p} \big)  \!=\!
\big(\fffrac{p\,,\,k^\cyc/k}{\ov {\mathfrak p}} \big)  \!=\! 1$, 
$\big(\fffrac{\ov x\,,\,k^\cyc/k}{\mathfrak p} \big)  \!=\!
\big(\fffrac{x\,,\,k^\cyc/k}{\mathfrak p} \big)^{-1}  \!=\! 
\big(\fffrac{\ov x\,,\,k^\cyc/k}{\ov {\mathfrak p}} \big)^{-1}  \!=\!
\big(\fffrac{x\,,\,k^\cyc/k}{\ov {\mathfrak p}} \big)$. 
\end{lemma}

\begin{proof}
In the cyclotomic extension $\Q(\mu_{p^{n+1}}^{})/\Q$, $p = 
\BN_{\Q(\mu_{p^{n+1}}^{})/\Q}(1 - \zeta_{p^{n+1}}^{})$, for all 
$n \geq 0$ (hence norm in $\Q_n^\cyc/\Q$), whence $p$ local 
norm at ${\mathfrak p}$ and 
${\ov {\mathfrak p}}$ in $k^\cyc/k$, and, since $x\, \ov x = 
p^{\hp}$, the relations follow, showing equality 
of the four symbols or their inverses. 
\end{proof}

\begin{lemma}\label{ordercyc2}
Set $K \cap L =: K_e$, $K \cap \ov L =: K_{\ov e}$, with $e \geq \ov e$.
For $n \geq e$, the order of $\big(\fffrac{x\,,\,{K_n}/k}{\mathfrak p} \big)$
is that of $\big(\fffrac{x\,,\,{k_{n - e}^\cyc}/k}{\mathfrak p} \big)$ and
for $n \geq \ov e$, the order of $\big(\fffrac{x\,,\,{K_n}/k}{\ov {\mathfrak p}} \big)$
is that of $\big(\fffrac{x\,,\,{k_{n - \ov e}^\cyc}/k}{\ov {\mathfrak p}} \big)$;
otherwise, if $n < e$ (resp. $n < \ov e$), the orders are $1$.
\end{lemma}

\begin{proof}
We consider the following diagrams, where $\wt k$ is the direct compositum
over $k$ (resp. over $K_e$) of $L$ and $k^\cyc$ (resp. of $L$ and $K$),
and similarly, $\wt k$ is the direct compositum over $k$ (resp. over 
$K_{\ov e}$) of $\ov L$ and $k^\cyc$ (resp. of $\ov L$ and $K$); 
we denote by $L^m$, $m \geq 0$, the extension of degree $p^m$ 
of $L$ in $\wt k$, and similarly for $\ov L$:
\unitlength=0.6cm
\[\vbox{\hbox{\hspace{3.0cm}
\begin{picture}(10.0,3.65)
\put(4.5,3.2){\line(1,0){2.8}}
\put(1.8,3.2){\line(1,0){1.9}}
\put(4.6,0.2){\line(1,0){2.65}}
\put(1.85,0.2){\line(1,0){1.85}}
\put(1.5,0.5){\line(0,1){2.4}}
\put(4.00,0.5){\line(0,1){2.4}}
\put(7.5,0.5){\line(0,1){2.4}}
\put(-5.0,3.0){${\rm Symbol}\ 
\big(\fffrac{\ov x\,,\,{\wt k}/k}{\mathfrak p} \big)$:}
\put(7.6,1.6){\ft$\Z_p$}
\put(2.6,3.4){\ft$p^m$}
\put(7.4,3.1){\ft$\wt k$}
\put(3.85,3.1){\ft$L^m$}
\put(1.4,3.1){\ft$L$}
\put(3.85,0.1){\ft$k_m^\cyc$}
\put(7.3,0.1){\ft$k^\cyc$}
\put(1.34,0.1){\ft$k$}
\put(2.2,2.8){\tiny${\mathfrak p}\,\Ram$}
\put(5.2,2.8){\tiny${\mathfrak p}\,\Ram$}
\put(1.6,1.7){\tiny${\mathfrak p}\,\Nram$}
\end{picture}}} \]

\unitlength=0.6cm
\[\vbox{\hbox{\hspace{3.0cm}
\begin{picture}(10.0,3.55)
\put(4.85,3.2){\line(1,0){2.45}}
\put(1.75,3.2){\line(1,0){1.65}}
\put(4.5,0.2){\line(1,0){2.75}}
\put(1.85,0.2){\line(1,0){1.85}}
\put(-0.94,0.2){\line(1,0){2.0}}
\put(1.5,0.5){\line(0,1){2.4}}
\put(4.00,0.5){\line(0,1){2.4}}
\put(7.5,0.5){\line(0,1){2.4}}
\put(-5.0,3.0){${\rm Symbol}\ 
\big(\fffrac{\ov x\,,\,{\wt k}/k}{\mathfrak p} \big)$:}
\put(7.6,1.6){\ft$\Z_p$}
\put(2.2,3.4){\ft$p^{n\!-e}$}
\put(7.4,3.1){\ft$\wt k$}
\put(3.5,3.1){\ft$L^{n\!-e}$}
\put(1.4,3.1){\ft$L$}
\put(3.85,0.1){\ft$K_n$}
\put(7.3,0.1){\ft$K$}
\put(1.2,0.1){\ft$K_e$}
\put(-1.3,0.1){\ft$k$}
\put(-0.7,0.4){\tiny${\mathfrak p}\,\Nram$}
\put(-0.1,1.0){\ft$p^e$}
\put(1.6,1.7){\tiny${\mathfrak p}\,\Nram$}
\put(2.2,2.8){\tiny${\mathfrak p}\,\Ram$}
\put(5.5,2.8){\tiny${\mathfrak p}\,\Ram$}
\end{picture}}} \]

\unitlength=0.6cm
\[\vbox{\hbox{\hspace{3.0cm}
\begin{picture}(10.0,3.55)
\put(4.4,3.2){\line(1,0){2.85}}
\put(1.8,3.2){\line(1,0){1.9}}
\put(4.6,0.2){\line(1,0){2.65}}
\put(1.85,0.2){\line(1,0){1.85}}
\put(1.5,0.5){\line(0,1){2.4}}
\put(4.00,0.5){\line(0,1){2.4}}
\put(7.5,0.5){\line(0,1){2.4}}
\put(-5.0,3.0){${\rm Symbol}\ 
\big(\fffrac{x\,,\,{\wt k}/k}{\ov {\mathfrak p}} \big)$:}
\put(7.6,1.6){\ft$\Z_p$}
\put(2.6,3.4){\ft$p^m$}
\put(7.4,3.1){\ft$\wt k$}
\put(3.85,3.1){\ft$\ov L^m$}
\put(1.4,3.1){\ft$\ov L$}
\put(3.85,0.1){\ft$k_m^\cyc$}
\put(7.3,0.1){\ft$k^\cyc$}
\put(1.34,0.1){\ft$k$}
\put(2.2,2.75){\tiny${\ov {\mathfrak p}}\,\Ram$}
\put(5.2,2.75){\tiny${\ov {\mathfrak p}}\,\Ram$}
\put(1.6,1.7){\tiny${\ov {\mathfrak p}}\,\Nram$}
\end{picture}}} \]

\unitlength=0.6cm
\[\vbox{\hbox{\hspace{3.0cm}
\begin{picture}(10.0,3.55)
\put(4.5,3.2){\line(1,0){2.8}}
\put(1.75,3.2){\line(1,0){1.65}}
\put(4.5,0.2){\line(1,0){2.75}}
\put(1.85,0.2){\line(1,0){1.85}}
\put(-0.94,0.2){\line(1,0){2.0}}
\put(1.5,0.5){\line(0,1){2.4}}
\put(4.00,0.5){\line(0,1){2.4}}
\put(7.5,0.5){\line(0,1){2.4}}
\put(-5.0,3.0){${\rm Symbol}\ 
\big(\fffrac{x\,,\,{\wt k}/k}{\ov {\mathfrak p}} \big)$:}
\put(7.6,1.6){\ft$\Z_p$}
\put(2.2,3.4){\ft$p^{n\!-\!\ov e}$}
\put(7.4,3.1){\ft$\wt k$}
\put(3.5,3.1){\ft$\ov L^{n\!-\!\ov e}$}
\put(1.4,3.1){\ft$\ov L$}
\put(3.85,0.1){\ft$K_n$}
\put(7.3,0.1){\ft$K$}
\put(1.1,0.1){\ft$K_{\ov e}$}
\put(-1.3,0.1){\ft$k$}
\put(-0.7,0.4){\tiny${\ov {\mathfrak p}}\,\Nram$}
\put(-0.1,1.0){\ft$p^{\ov e}$}
\put(1.6,1.7){\tiny${\ov {\mathfrak p}}\,\Nram$}
\put(2.2,2.75){\tiny${\ov {\mathfrak p}}\,\Ram$}
\put(5.5,2.75){\tiny${\ov {\mathfrak p}}\,\Ram$}
\end{picture}}} \]
\unitlength=1.0cm

\smallskip
In the first and third diagram, $k_m^\cyc = k \,\Q_m^\cyc$, 
where $\Q_m^\cyc$ is of conductor $p^{m+1}$. 

\smallskip
In the second and fourth diagram, since $[K_n : k] = p^n$ and $K\cap L = K_e$
(resp. $K\cap \ov L = K_{\ov e}$), it follows that $[L K_n : L] = p^{n - e}$ 
(resp. $[\ov L K_n : \ov L] = p^{n - \ov e}$). 

\smallskip
Finally, $L^m = L^{n - e}$ for $n = m + e$  
(resp. $\ov L^m = \ov L^{n - \ov e}$ for $n = m + \ov e$).

\smallskip
Recall that ${\mathfrak p}$ (resp. ${\ov {\mathfrak p}}$) is totally ramified in 
$K/K_e$ (resp. $K/K_{\ov e}$) and unramified in $L/K_e$
(resp. $\ov L/K_{\ov e}$). So:

\smallskip
$\big(\fffrac{\ov x\,,\,{\wt k}/k}{\mathfrak p} \big)$ fixes $L$ since $\ov x$ is a local 
unit in an extension $L/k$ unramified at ${\mathfrak p}$; \par \vspace{0.1cm}
$\big(\fffrac{\ov x\,,\,{k^\cyc}/k}{\mathfrak p} \big) \in \Gal(k^\cyc/ k)$, 
by restriction of $\big(\fffrac{\ov x\,,\,{\wt k}/k}{\mathfrak p} \big)$
to $k^\cyc$; \par\vspace{0.1cm}
$\big(\fffrac{\ov x\,,\,{K}/k}{\mathfrak p} \big) \in \Gal(K/K_e)$, by restriction
of $\big(\fffrac{\ov x\,,\,{\wt k}/k}{\mathfrak p} \big) \in \Gal(\wt k/L)$; \par\vspace{0.1cm}
$\big(\fffrac{x\,,\,{\wt k}/k}{{\ov {\mathfrak p}}} \big)$ fixes $\ov L$ since $x$ is a local 
unit in an extension $\ov L/k$ unramified at ${\ov {\mathfrak p}}$; \par\vspace{0.1cm}
$\big(\fffrac{x\,,\,{k^\cyc}/k}{{\ov {\mathfrak p}}} \big) \in \Gal(k^\cyc/ k)$, 
by restriction of $\big(\fffrac{x\,,\,{\wt k}/k}{{\ov {\mathfrak p}}} \big)$
to $k^\cyc$; \par\vspace{0.1cm}
$\big(\fffrac{x\,,\,{K}/k}{{\ov {\mathfrak p}}} \big) \in \Gal(K/K_{\ov e})$, by restriction of
$\big(\fffrac{x\,,\,{\wt k}/k}{{\ov {\mathfrak p}}} \big) \in \Gal(\wt k/\ov L)$.

\smallskip
Since all these restriction maps are isomorphisms, one obtains successively 
that the order of $\big(\fffrac{\ov x\,,\,{k_m^\cyc}/k}{\mathfrak p} \big)$ is that 
of $\big(\fffrac{\ov x\,,\,{L^m}/k}{\mathfrak p} \big)$, and that the order of 
$\big(\fffrac{\ov x\,,\,{K_n}/k}{\mathfrak p} \big)$ is that of 
$\big(\fffrac{\ov x\,,\,{L^{n - e}}/k}{\mathfrak p} \big)$.
For $n = m + e$, we get that the order of 
$\big(\fffrac{\ov x\,,\,{K_n}/k}{\mathfrak p} \big)$ is that of
$\big(\fffrac{\ov x\,,\,{k_{n - e}^\cyc}/k}{\mathfrak p} \big)$, hence 
that of $\big(\fffrac{x\,,\,{k_{n - e}^\cyc}/k}{\mathfrak p} \big)$
from Lemma \ref{ordercyc}. Similarly, the order of 
$\big(\fffrac{x\,,\,{K_n}/k}{{\ov {\mathfrak p}}} \big)$ is that of
$\big(\fffrac{x\,,\,{k_{n - \ov e}^\cyc}/k}{{\ov {\mathfrak p}}} \big)$
and that of $\big(\fffrac{\ov x\,,\,{k_{n - \ov e}^\cyc}/k}{\ov {\mathfrak p}} \big)$.
\end{proof}

But the orders of the symbols in $k_m^\cyc/k$ are easily computable:

\begin{theorem}\label{mainorder}
Let $k$ be an imaginary quadratic field and let $p \geq 3$ be a prime 
number split in $k$. Let $x \in k^\times$ be the generator of 
${\mathfrak p}^{\hp}$, where $\hp \mid \hk$ is the order 
of $\Bcl({\mathfrak p})$, and let $\delta_p(k) := v_{\mathfrak p}
(\ov x^{\,p-1}-1) - 1$. Let $K$ be a $\Z_p$-extension of $k$ such that 
${\mathfrak p}$, ${\ov {\mathfrak p}}$ are ramified from some layers
$K_e$, $K_{\ov e}$, $e \geq \ov e$. For any $n \geq 0$, 
the symbols $\big (\ffrac{\ov x\,,\,{K_n/k}}{\mathfrak p} \big)$ and
$\big (\ffrac{x\,,\,{K_n/k}}{{\ov {\mathfrak p}}} \big)$ are of orders
$p^{\max(0, n - e - \delta_p(k))}$ and $p^{\max(0, n - \ov e - \delta_p(k))}$,
respectively.
\end{theorem}

\begin{proof}
For this, it suffices from Lemma \ref{ordercyc2} to compute for $m$ large 
enough the order of the Artin symbol 
$\big (\fffrac{k^\cyc_m /k}{{\mathfrak a}_{N,\mathfrak p}}\big)$, obtained in 
expression \eqref{ideal}, for $N \gg 0$. It is equivalent to compute the order 
of the image of $\big (\fffrac{k^\cyc_m /k}
{{\mathfrak a}_{N,\mathfrak p}}\big)$ in $\Gal(\Q^\cyc_m/\Q) \simeq 
\big \{a \in (\Z/ p^{m+1}\Z)^\times, \ a \equiv 1 \pmod p \big\} \simeq 
\Z/ p^m \Z$, which is the Artin symbol
$\big (\fffrac{\Q^\cyc_m/\Q}{\BN_{k/\Q}( {\mathfrak a}_{N,\mathfrak p})}\big)$, 
whose order is that of the Artin symbol of the integer 
$\BN_{k/\Q}({\mathfrak a}_{N,\mathfrak p})^{p-1}$ 
of the form $1 + p^{1+\delta} \, u$, with $v_p(u) = 0$ and $\delta \geq 0$ 
to be computed. Before, let's prove that ${\mathfrak a}_{N,\mathfrak p} 
= (y^{}_{N,\mathfrak p}) \cdot {\mathfrak p}^{- \hp}$ is 
such that, for $N \gg 0$, $\BN_{k/\Q}({\mathfrak a}_{N,\mathfrak p}) 
\equiv x \cdot p^{- \hp} \pmod {{\mathfrak p}^N}$.

Indeed, one has $\BN_{k/\Q}({\mathfrak a}_{N,\mathfrak p}) = 
\BN_{k/\Q} (y^{}_{N,\mathfrak p}) \cdot p^{- \hp}$; 
by definition of an associate, $y^{}_{N,\mathfrak p} x^{-1} \equiv 1 
\pmod {{\mathfrak p}^N}$, $y^{}_{N,\mathfrak p} \equiv 1 
\pmod {{{\ov {\mathfrak p}}}^N}$, thus $\ov y^{}_{N,\mathfrak p} 
\equiv 1 \pmod {{\mathfrak p}^N}$, which leads to $\BN_{k/\Q} 
(y^{}_{N,\mathfrak p}) \cdot x^{-1} \equiv 1 \pmod {{\mathfrak p}^N}$.

Finally, $\BN_{k/\Q}({\mathfrak a}_{N,\mathfrak p}) \equiv 
x \cdot p^{- \hp} \pmod {{\mathfrak p}^N}$, whence 
(see Section \ref{log}):
$$v_p (\BN_{k/\Q} \big({\mathfrak a}_{N,\mathfrak p})^{p-1} - 1\big) =
v_{\mathfrak p} \big ((x \cdot p^{- \hp})^{p-1} - 1\big ) =
v_{\mathfrak p} \big ((\ov x ^{p-1} - 1\big ) = 1 + \delta_p(k); $$

\noindent
the number $\delta_p(k)$ gives the order of $\big (\ffrac{k^\cyc_m /k}
{{\mathfrak a}_{N,\mathfrak p}}\big)$ in $\Gal(k^\cyc_m/k)$, as soon 
as $m \geq \delta_p(k)$, and this order is $p^{m - \delta_p(k)}$ (for 
$m < \delta_p(k)$, this order would be $1$ since in that case,
$\BN_{k/\Q}( {\mathfrak a}_{N,\mathfrak p})^{p-1} \equiv 1 \pmod {p^{m+1}}$). 
Whence the formulas for $m = n - e$, then $m = n - \ov e$.
\end{proof}

\begin{remark}\label{order}
We will use the relation giving the minimal order, 
of the images of the two fundamental $S_k$-units $x$, $\ov x$,
in $\Omega_{K_n/k} := \big\{(s, \ov s) \in G_{n,\mathfrak p} \times 
G_{n,{\ov {\mathfrak p}}},\  s\cdot \ov s = 1 \big\}$:
\begin{equation*}
\min \big (\order \omega_{K_n/k}(x), \order\omega_{K_n/k}(\ov x)\big) =
\order \big (\fffrac{\ov x\,,\,{K_n/k}}{\mathfrak p} \big) = 
p^{\max(0, n - e - \delta_p(k))},
\end{equation*}

\noindent
where $G_{n,\mathfrak p}$ and $G_{n,{\ov {\mathfrak p}}}$ are the inertia groups 
of ${\mathfrak p}$ and $\ov{\mathfrak p}$ in $\Gal(K_n/k)$, knowing that $\ov e \leq e$;
of course, $\order \omega_{K_n/k}(E_k^{S_k}) = \order \omega_{K_n/k}
(\langle x, \ov x \rangle) = \order \omega_{K_n/k}(x)$
(definitions of the forthcoming \S\,\ref{genus}).
\end{remark}

\section{Higher rank Chevalley--Herbrand formulas in 
\texorpdfstring{$K_n/k$}{Lg}} 

The basic principle beyond Chevalley--Herbrand formula for 
$p$-class groups in any $p$-cyclic (or pro-$p$-cyclic) extension 
$K = \cup_n K_n$ of $k$, is to define filtrations of the form:
$$\big \{\CH_{K_n}^i =: \Ccl_n(\CI_{\!K_n}^i) \big\}_{i \geq 0}, $$

\noindent
for suitable ideal groups $\CI_{\!K_n}^i$, where  $G_n := \Gal(K_n/k)$, and:
\begin{equation*}
\CH_{K_n}^0 :=1,\ \  
\CH_{K_n}^1 := \CH_{K_n}^{G_n}, \ \  \CH_{K_n}^{i+1}/\CH_{K_n}^i 
:= (\CH_{K_n}/\CH_{K_n}^i)^{G_n}, 
\end{equation*}

\noindent 
up to a minimal bound $i = b_n$ for which $\CH_{K_n}^{b_n} = \CH_{K_n}$.
The higher rank Chevalley--Herbrand formulas give a computable expression of 
$\order\big( \CH_{K_n}^{i+1}/\CH_{K_n}^i\big)$ in an algorithmic point of view 
for $n \geq 2$ (see the improved english translation \cite{Gra2017a} of our 1994's 
paper and its transcript into the id\'elic form in \cite{LiYu2020}).

\subsection{Expression of \texorpdfstring{$\order \CH_{K_n}^{G_n}$}{Lg}}

Let $k = \Q(\sqrt{-m})$ and $p \geq 3$ split in $k$. Let $K$ be a 
$\Z_p$-extension of $k$ distinct from $L$ and $\ov L$, hence $K/k$ 
ramified at the two $p$-places ${\mathfrak p}$ and ${\ov {\mathfrak p}}$ 
from some layers $K_e$ and $K_{\ov e}$ of $K$, $e \geq \ov e$. 
Moreover, we have, from Diagram \ref{maindiagram}: 
$$\order \CH_k = p^{\ov e} \times [\knr : K^\nr] \times \order \CT_k =
[K_n^\nr : k] \times \order \BN_{K_n/k}(\CH_{K_n}). $$ 

The Chevalley--Herbrand formula \cite[pp.\,402-406]{Che1933} for 
any cyclic $p$-extension $M/F$ of number fields $F, M$, of 
Galois group $G$, is $\ds \order \CH_M^G = \ffrac{\order \CH_F 
\times \prod_v \,p^{n_v}}{[M : F] \times (E_F : E_F \cap \BN_{M/F}
(M^\times))}$, where $p^{n_v}$ is the ramification index of the place $v$.
In our particular case $K_n/k$, ${\mathfrak p}$ and ${\ov {\mathfrak p}}$ 
are of ramification indices $p^{\max(0,n - e)}$, $p^{\max(0,\,n - \ov e)}$, 
respectively, and $E_k \otimes \Z_p = 1$. Using the above expression 
of $\order \CH_k$ gives the final expression:
\begin{equation}\label{CH}
\order \CH_{K_n}^{G_n} = \frac{\order \CH_k \times 
p^{\max(0,\,n - e)} p^{\max(0, \,n - \ov e)} }{p^n}
 = \order \BN_{K_n/k}(\CH_{K_n}) \times p^{\max(0, n-e)},
\end{equation}

\noindent
for all $n \geq 0$; then immediately the well-known result that, for $K$ 
distinct from $L$ and $\ov L$, $\order \CH_{K_n}$ is unbounded for 
$n \to \infty$, whence $\lambda_p(K/k) \geq 1$ or $\mu_p(K/k) \geq 1$.

\smallskip
We recall, now, a more suitable writting of these formulas, by means of two
integral factors taking into account the ``Product Formula'' of class field theory.

\subsection{Genus theory form for the filtration of 
\texorpdfstring{$\CH_{K_n}$}{Lg}}\label{genus}

We still assume $K$ distinct from $L$ and $\ov L$. Denote by 
$G_{n, \mathfrak p}$ and $G_{n,{\ov {\mathfrak p}}}$ the inertia 
groups in $K_n/k$, of ${\mathfrak p}$ and ${\ov {\mathfrak p}}$, 
respectively, for all $n \geq 0$. We define the $G_n$-module: 
$$\Omega_{K_n/k} := \big\{(s, \ov s) \in G_{n,\mathfrak p} \times 
G_{n,{\ov {\mathfrak p}}},\  s\cdot \ov s = 1\ {\rm in}\ G_n \big\}, $$ 

\noindent
and we consider the 
map $\omega_{K_n/k} : \CN_{K_n/k} \to \Omega_{K_n/k}$, 
where $\CN_{K_n/k}$ is the subgroup of $k^\times$ of elements 
$a$, {\it local norms in $K_n/k$ for all places $v \nmid p$}
(recall that any $a \in k^\times$, norm of an ideal in $K_n/k$, is local 
norm at the non-ramified places):
$$\omega_{K_n/k}(a) = \big( \big(\fffrac{a\,,\,{K_n}/k}{\mathfrak p} \big),
\big(\fffrac{a\,,\,{K_n}/k}{{\ov {\mathfrak p}}} \big) \big) \in \Omega_{K_n/k}, $$

\noindent
since these norm residue symbols satisfy the product formula:
$\prod_{v} \big(\fffrac{a\,,\,{K_n}/k}{v} \big) = 1$. The order of
$\omega_{K_n/k}(a)$ is that of $\big(\fffrac{a\,,\,{K_n}/k}{\mathfrak p} \big)$
and divides $\order \Omega_{K_n/k}$.

The general expression (giving the most powerful fixed points formula
in a cyclic extension $M/F$), that  may be called 
genus theory form of Chevalley--Herbrand formulas, is the following product 
of two integers, for any submodule $\CH =: \Ccl_M(\CI)$ of $\CH_M$, where 
$\CI$ is a submodule of $I_M$ of finite type \cite[Corollary 3.7]{Gra2017a}:
\begin{equation}\label{formule}
\left\{\begin{aligned}
\order\big(\CH_M/\CH \big)^{G} & = \frac{\order \BN_{M/F}(\CH_M)}
{\order \BN_{M/F}(\CH)} \times \frac{\order \Omega_{M/F}}
{\order \omega_{M/F}(\Lbda_{M/F})}, \\
\Lbda_{M/F} & := \{a \in F^\times, \, (a) \in \BN_{M/F}(\CI)\}.
\end{aligned}\right.
\end{equation}

The first integer is called the class factor and the second one the norm factor.
Since $\CI$ is defined modulo principal ideals and since $\Lbda_{M/F}$ 
contains the group of units of $F$, $\omega_{M/F}(\Lbda_{M/F})$
is well defined and only depend on $\CH$.  In other words, $\Lbda_{M/F}/
\Lbda_{M/F} \cap \BN_{M/F}(M^\times) \simeq \omega_{M/F}(\Lbda_{M/F})$.
We apply this formula to $M/F = K_n/k$.

\begin{remark}\label{orderomega}
We have, for any $n\geq 0$, $\Omega_{K_n/k} \simeq \Z/p^{\max(0,n - e)}\Z$. 
Indeed, if $n \leq e$, there is at most a single $p$-place ramified 
in $K_n/k$, thus $\Omega_{K_n/k} = 1$ and $\ffrac{\order \Omega_{K_n/k}}
{\order \omega_{K_n/k}(\Lbda_{K_n/k})} = 1$ in \eqref{formule}.
Otherwise, product formula takes place in $\Gal(K_n/K_e)$ of order 
$p^{n-e}$ since $\ov e \leq e$, and where $p$ totally ramifies.
From \eqref{CH}, \eqref{formule} and $\order \BN_{K_n/k}(\CH_{K_n}) 
= [H_k^\nr : K_n^\nr]$, we get:
\begin{equation}\label{filtration}
\left\{\begin{aligned}
\order \CH_{K_n}^1 & := \order \CH_{K_n}^{G_n} = \order \BN_{K_n/k}(\CH_{K_n}) 
\!\times\! p^{\max(0,n - e)}, \\
\Big [& = [\knr : K^\nr] \!\times\! \order \CT_k \times p^{n - e}, \ 
{\rm if}\ n \geq e \Big], \\
\order (\CH_{K_n}^{i+1}/\CH_{K_n}^i) & 
:= \frac{\order \BN_{K_n/k}(\CH_{K_n})}
{\order \BN_{K_n/k}(\CH_{K_n}^i)} \times \frac{p^{\max(0,n - e)}}
{\order \omega_{K_n/k}(\Lbda_{K_n/k}^i)},\ \, i \geq 1,\\
\Big [& = \frac{[\knr : K^\nr] \!\times\! \order \CT_k}
{\order \BN_{K_n/k}(\CH_{K_n}^i)} \times \frac{p^{n - e}}
{\order \omega_{K_n/k}(\Lbda_{K_n/k}^i)}, \ {\rm if}\ n \geq e \Big],
\end{aligned}\right.
\end{equation}

\noindent
where the $G_n$-module $\Lbda_{K_n/k}^i := \{a \in k^\times,\ 
(a) \in \BN_{K_n/k}(\CI_{\!K_n}^i) \}$ is a subgroup of finite type of 
$k^\times$, so that $\omega_{K_n/k}(\Lbda_{K_n/k}^i) \subseteq 
\Omega_{K_n/k}$ from the product formula.
\end{remark}

\subsection{Properties of the filtration}

The behavior of the filtrations, regarding $n \to \infty$ and the number of steps 
$b_n$, determine, logacally, the parameters $\lambda_p(K/k)$, 
$\mu_p(K/k)$; we will see that this is often possible and allows to generalize
class field theory methods that we describe in Section \ref{gold}.
For this, recall some properties and algorithms. 

\smallskip
To ease notations, put $\CH_n := \CH_{K_n}$, $\Lbda_n^i := \Lbda_{K_n/k}^i$, 
$\BN_n := \BN_{K_n/k}$, and so on; let $\sigma_n$ be the generator of $G_n
= \Gal(K_n/k)$ obtained from the restriction to $K_n$ of a fixed topological 
generator $\sigma$ of $\Gal(K/k)$. Note that $k = K_0$. For $n \geq 0$ fixed: 
$$\CH_n^i = \{c \in \CH_n,\, c^{(1 - \sigma_n)^i} =1\},\ \,\hbox{for all $i \geq 0$}. $$

\subsubsection{Algoritmic aspects}\label{inclusions}

(i) Set $\CH_n^i = \Ccl_n(\CI_{\!n}^i)$ for a $G_n$-module $\CI_{\!n}^i$ of 
ideals of $K_n$, and consider the expression $\ds \ffrac{\order \BN_n (\CH_n)}
{\order \BN_n (\CH_n^i)} \cdot  \ffrac{\order\Omega_n}{\order\omega_n 
(\Lbda_n^i)}$ giving $\order (\CH_n^{i+1}/\CH_n^i)$, for any $n \geq 0$. 
Then the principal 
ideals $(\alpha)$ of $\BN_n(\CI_{\!n}^i)$, such that $\omega_n(\alpha) = 1$,
allow to find solutions $\beta$ to the norm relation
$\alpha = \BN_n(\beta) \in \Lbda_n^i \cap \BN_n(K_n^\times)$, 
by definition such that $(\alpha) = \BN_n( {\mathfrak A})$, for
${\mathfrak A} \in \CI_{\!n}^i$, yielding $(\beta) = {\mathfrak A} \cdot 
{\mathfrak B}^{1-\sigma_n}$, creating ${\mathfrak B} \in \CI_{\!n}^{i+1}$, 
and $\Ccl (\BN_n({\mathfrak B}))$ in $\CH_k$, hence building 
$\CH_n^{i+1} = \Ccl_n(\CI_{\!n}^{i+1})$ and $\BN_n(\CH_n^{i+1})$; then 
$\Lbda_n^{i+1}$ comes from principal ideals of $\BN_n(\CI_{\!n}^{i+1})$,
yielding the next step of the algorithm, whence  
$\order \BN_n (\CH_n^{i+1}) \geq \order \BN_n (\CH_n^i)$ and  
$\order\omega_n (\Lbda_n^{i+1}) \geq \order\omega_n (\Lbda_n^i)$.

\smallskip
(ii) We have the following diagram where $\BN_{K_{n+1}/K_n}$, on 
$\CH_{n+1}$ and $(\CH_{n+1})^{(1-\sigma_{n+1})^i}$, is onto for $n$ 
large enough such that $K/K_n$ is totally ramified. 
On $\CH_{n+1}^i$, the norm may be non-surjective and/or non-injective 
(a crucial point in the attempt of proving Greenberg's conjectures, somewhat
due to random arithmetic aspects):
$$\begin{array}{ccccccccc}  
1 & \too & \CH_{n+1}^i & \toooo & \CH_{n+1} & \stackrel{(1-\sigma_{n+1})^i}
{\toooo} & (\CH_{n+1})^{(1-\sigma_{n+1})^i} & \!\! \too  1 
\\ \vspace{-0.3cm} \\
&& \Big \downarrow & & \hspace{-1.4cm} \hbox{\ft$\BN_{K_{n+1}/K_n}$\ns} 
\big \downarrow \!\!\!\! \Big \downarrow && \hspace{-2.3cm} 
\hbox{\ft$\BN_{K_{n+1}/K_n}$\ns} \big \downarrow \!\!\!\! \Big \downarrow & 
\\ \vspace{-0.4cm} \\ 
1 & \too & \CH_n^i & \toooo & \CH_n & \stackrel{(1-\sigma_n)^i}
{\toooo} & (\CH_n)^{(1-\sigma_n)^i} & \!\! \too  1\,.
\end{array} $$

We have $\BN_{K_{n+1}/K_n} (\CH_{n+1}^i)  \subseteq \CH_n^i$~; so, 
for all ideal ${\mathfrak A}_{n+1} \in \CI_{\!n+1}^i$, one may write:
$$\BN_{K_{n+1}/K_n}({\mathfrak A}_{n+1}) = (\alpha_n) \, {\mathfrak A}_n, 
\ \  \alpha_n \in K_n^\times,\ \, {\mathfrak A}_n \in \CI_{\!n}^i. $$ 

Then, {\it modifying} $\CI_{\!n}^i$ modulo principal ideals of $K_n$ 
(${\mathfrak A}_n \mapsto {\mathfrak A}'_n := (\alpha_n)\,{\mathfrak A}_n$) 
does not modify $\CH_n^i$, and one may suppose $\BN_{K_{n+1}/K_n}
(\CI_{\!n+1}^i) \subseteq \CI_{\!n}^i$, whence $\BN_{K_{n+1}/k}(\CI_{\!n+1}^i) 
\subseteq \BN_{K_{n}/k}(\CI_{\!n}^i)$; this modifies $\Lbda_n^i = 
\{ a \in k^\times,\, (a) \in \BN_{K_n/k}(\CI_{\!n}^i) \}$ modulo 
$\BN_n(K_n^\times)$ not changing the norm index $\omega_n
(\Lbda_n^i) = (\Lbda_n^i : \Lbda_n^i\cap \BN_{K_{n}/k}(K_n^\times))$.
One may obtain:
\begin{equation}\label{incl}
\Lbda_{n+h}^i \subseteq \cdots \subseteq \Lbda_{n+1}^i 
\subseteq \Lbda_{n}^i, \ \, \hbox{for all $h\geq 1$}. 
\end{equation}

(iii) We assume that $p$ splits in $k$ and that $K/k$ is any $\Z_p$-extension 
such that ${\mathfrak p}$ and ${\ov {\mathfrak p}}$ ramify from some layers 
$K_e$ and $K_{\ov e}$, $e \geq \ov e$. 

\begin{lemma} \label{lem1} 
(i) For $n \geq 0$, the $i$-sequence $\order (\CH_n^{i+1} / \CH_n^i)$, 
$0 \leq i \leq b_n$, is decreasing from $\order \CH_n^1
 = \order \BN_n(\CH_n) \cdot p^{\max(0,n - e)}$ to $1$.

\smallskip
(ii) For $i \geq 0$, the $n$-sequence $\order \big( \CH_n^{i+1} / 
\CH_n^i \big)$ is increasing and stationary from some layer.

\smallskip
(iii) The $i$-sequence $\ds \lim_{n \to \infty} \order (\CH_n^{i+1} / \CH_n^i)$,
$i \geq 1$, is decreasing and stationary.
\end{lemma}

\begin{proof} 

{\bf a}) Indeed, $1-\sigma_n$ defines the injections 
$\CH_n^{i+1} / \CH_n^i \!\hookrightarrow\! \CH_n^i / \CH_n^{i-1}$, 
for all $i \geq 1$.

\smallskip
{\bf b}) Consider, for $i \geq 0$ fixed, the $n$-sequence 
$\order ( \CH_n^{i+1} / \CH_n^i) = \ds \ffrac{\order \BN_n(\CH_n)}
{\order \BN_n (\CH_n^i)} \cdot \ds \ffrac{\order\Omega_n}
{\order\omega_n (\Lbda_n^i)}$. 

\smallskip
\quad $\bullet\, $ We get $\BN_{n+1} (\CH_{n+1}^i) \subseteq \BN_n (\CH_n^i)$; 
thus, the $n$-sequence $\ds \ffrac{\order \BN_n(\CH_n)}{\order \BN_n( \CH_n^i)} 
=: p^{c_n^i}$ is {\it increasing} and stationary at a maximal value $p^{c_{\infty}^i}$ 
dividing $\order \BN_n(\CH_n)$, for all $i$ fixed. 

\smallskip
\quad $\bullet\, $ Put $\ds \ffrac{\order \Omega_n}
{\order \omega_n (\Lbda_n^i)} =: p^{\rho_n^i}$. Then $p^{\rho_{n+1}^i-\rho_n^i} 
= p \cdot \ds \ffrac{\order \omega_n(\Lbda_n^i)}{\order \omega_{n+1}
(\Lbda_{n+1}^i)}$; using \eqref{incl}, one may assume $\Lbda_{n+1}^i 
\subseteq \Lbda_n^i$, hence $\order \omega_{n+1}(\Lbda_{n+1}^i) 
\leq \order \omega_{n+1}(\Lbda_n^i)$, whence
$p^{\rho_{n+1}^i-\rho_n^i} \geq p \cdot \ds 
\ffrac{\order \omega_n(\Lbda_n^i)}{\order \omega_{n+1}(\Lbda_n^i)}$. 

Under the restriction of Hasse's norm residue symbols, one gets the map
$\Omega_{n+1} \to \hspace{-0.45cm} \to \Omega_n$, whose kernel is of 
order $p$; the image of $\omega_{n+1}(\Lbda_n^i)$ is $\omega_n(\Lbda_n^i)$, 
whence the result for the increasing $n$-sequence of norm factors $p^{\rho_n^i}$, 
stationary at some $p^{\rho_{\infty}^i}$, for all $i$ fixed.  

\smallskip
{\bf c}) Using (i), we have $\ds \lim_{n \to \infty} \order (\CH_n^{i+1} / \CH_n^i) = 
p^{c_{\infty}^i} \!\cdot p^{\rho_{\infty}^i}$, where the $i$-sequences 
$p^{c_{\infty}^i}$ and $ p^{\rho_{\infty}^i}$ are decreasing and stationary.
\end{proof}

\subsubsection{Bounds for Iwasawa's invariants}

Assume that $p \geq 3$ splits in $k$ and that the $\Z_p$-extension
$K/k$ is totally ramified at ${\mathfrak p}$ and ${\ov {\mathfrak p}}$
from the layers $K_e$ and $K_{\ov e}$, respectively. 
We will get some $S_k$-unit $y \in \Lbda_n^i$, hence 
a power of $x$ of the form $x^{p^{h}}$, $h \geq 0$,
giving $\delta_{\mathfrak p}(y) = \delta_{\mathfrak p}(x)+ h
= \delta_p(k)+ h$ (Definition \ref{wtdeltapk}), for all $i \geq 1$.
Being independent of the choice of ${\mathfrak p} \in S_k$, we put
$\delta_{\mathfrak p}(y) =: \delta_p(y)$ (see its computation from 
Formulas \ref{normPn}, \ref{ycomput}, \ref{ycomputbis}).

\smallskip 
The forthcoming computations suppose $n$ large enough as soon as we 
consider Iwasawa's invariants; in particular, we assume 
$n \geq \ov e$ to get $K_n^\nr = K^\nr$ and $\order \BN_n(\CH_n) 
= [H_k^\nr : K_n^\nr] = [H_k^\nr : K^\nr] = \order \CH_k \cdot 
p^{-\ov e}$ since $K^\nr = K_{\ov e}$, and we assume
$n \geq \max \big(\expo_k, e + \delta_p(k) \big)$
to obtain the order of $\omega_n(y)$, equal to 
$p^{n - e - \delta_p(y)}$.

\begin{lemma}\label{majorant}
We have $p^{c_n^i} = \ffrac{[H_k^\nr : K^\nr]}{\order \BN_n( \CH_n^i)}
\leq [H_k^\nr : K^\nr]$ for all $i \geq 0$ and $p^{\rho_n^i} = \ffrac{p^{n - e}}
{\order \omega_n (\Lbda_n^i)} \leq p^{\delta_p(y)}$
for all $i \geq 1$ (resp. $\ffrac{p^{n - e}}{\order \omega_n (\Lbda_n^0)}
= p^{n - e}$ for $i = 0$). 
So, $\order (\CH_n^{i+1} / \CH_n^i)$ is bounded by 
$[H_k^\nr : K^\nr]\, p^{n - e}$ for $i = 0$, and by $[H_k^\nr : K^\nr] 
\times p^{\delta_p(y)}$ for all $i \geq 1$.
\end{lemma}

\begin{proof} 
Since $y \in \Lbda_n^i$, for $i \geq 1$, with $\omega_n(y)$ of 
order $p^{n - e - \delta_p(y)}$ (Theorems \ref{mainorder}, 
Theorem \ref{stability}, Remark \ref{order}), this yields $\ffrac{p^{n - e}}
{\order \omega_n (\Lbda_n^i)} \leq \ffrac{p^{n - e}}
{\order \omega_n (\langle\, y\, \rangle)} = p^{\delta_p(y)}$. 
Since $p^{c_{n}^i} \mid [H_k^\nr : K^\nr]$, the conclusion holds. 
The case $i = 0$ gives the upper bounds $[H_k^\nr : K^\nr]$ and $p^{n - e}$, 
respectively.
\end{proof}

Since $\order \CH_n^1 = [H_k^\nr : K^\nr]\, p^{n - e}$, 
$\order (\CH_n^2 / \CH_n^1) = p^{\wt \delta_p(k) }$
(Theorem \ref{stability}), and $\order (\CH_n^{i+1} / \CH_n^i) 
\geq p$ for $i \in [2, b_n-1]$, we obtain, from the identity $\order \CH_n 
= \prod_{i = 0}^{b_n-1} \! \order (\CH_n^{i+1} / \CH_n^i)$ and the 
Iwasawa formula, a lower bound for $\order \CH_n$; whence,
for $b_n \geq 2$:
\begin{equation}\label{inequalities}
\left\{\begin{aligned}
v_p ([H_k^\nr & :  K^\nr]) + n - e + \wt \delta_p(k) + (b_n-2) \times 1 \\
 \leq \, &\lambda_p(K/k) \cdot n + \mu_p(K/k)\cdot p^n + \nu_p(K/k) \\
 & \ \ \ \leq (b_n - 1) \cdot v_p ([H_k^\nr : K^\nr]) + n - e + \wt \delta_p(k) 
+ (b_n - 2) \cdot \delta_p(y).
\end{aligned}\right.
\end{equation}

Since $\delta_p(y) := 
\delta_p(k) + \max(0, \ov e- v_p(\hp))$ 
(apply \eqref{ycomput}, \eqref{ycomputbis}
when $n \geq \ov e$), and since in the simplest
case $b_n = 1$, $\wt \delta_p(k) = 0$ and $\order \CH_n = 
[H_k^\nr : K^\nr]\, p^{n-e}$ ($\lambda_p(K/k) = 1$, $\mu_p(K/k) = 0$, 
$\nu_p(K/k) = v_p( \hk) - e - \ov e$), we deduce, after 
some obvious transformations:

\begin{theorem}\label{O}
Let $k$ be an imaginary quadratic field and let $p \geq 3$ be a prime number 
split in $k$. Let $K$ be a $\Z_p$-extension of $k$ such that ${\mathfrak p}$, 
${\ov {\mathfrak p}}$ are ramified from some layers $K_e$, $K_{\ov e}$, 
respectively, and put $\nu_K^{} := v_p ([H_k^\nr : K^\nr]) = 
v_p(\hk) - \ov e$. Then, for all $n \gg 0$ and if $b_n \geq 2$:

\smallskip
(i) The Iwasawa invariants $\lambda_p(K/k)$, $\mu_p(K/k)$ and $\nu_p(K/k)$, 
satisfy the property:
\begin{equation}\label{inequalitiesbis}
\left\{\begin{aligned}
\wt \delta_p(k) & + b_n - 2 \leq (\lambda_p(K/k) - 1)\, n
+ \mu_p(K/k)\, p^n + \nu_p(K/k) + e - \nu_K^{} \\
& \leq  (b_n - 1) \times \,\wt \delta_p(k) 
+ (b_n - 2) \times \big[\max(\ov e, v_p(\hp)) - \ov e \big].
\end{aligned}\right.
\end{equation}

(ii) If $\CH_k = 1$ and $e = \ov e = 0$ ($\wt \delta_p(k) = \delta_p(k)$,
total ramification), the inequalities become
$\delta_p(k) + b_n - 2 \leq (\lambda_p(K/k) - 1)\, n
+ \mu_p(K/k)\, p^n + \nu_p(K/k) \leq  (b_n - 1) \times \delta_p(k)$.
\end{theorem}

One may compare with some forthcoming class field theory
proofs obtained in Section \ref{gold}. 

\begin{corollary}\label{plateau}
Assume $\lambda_p(K/k) \geq 2$ or $\mu_p(K/k) \geq 1$. Then
for $n \gg 0$ fixed, there exist intervals $I_j := [i_j, i_j+ \Delta_j]$, 
with $\Delta_j = O(n)$, such that $\order (\CH_n^{i+1} / \CH_n^i)$
is constant on $I_j$.
\end{corollary}

\begin{proof}
By assumption, we have $b_n \geq O(n)$. Consider $[1, b_n]$ as 
disjoint union of intervals $I_j := [i_j, i_j+ \Delta_j]$, $\Delta_j \geq 0$, 
on which $\order (\CH_n^{i+1} / \CH_n^i) = C_j$, with $C_{j+1} > C_j$; 
since the $C_j$'s are bounded by an explicit constant (Lemmas \ref{lem1}, 
\ref{majorant}), the result is obvious.
\end{proof}

Regarding the nature of the algorithm described in \S\,\ref{inclusions}\,(i),
the Borel--Cantelli probability of such a phenomenon tends to $0$ as
$n$ tends to $\infty$, which enforces Conjecture \ref{pfinitude}.
In the semi-simple totally real case, and $p$ totally split in $k$, we 
have proven, in \cite[Th\'eor\`eme 6.3]{Gra2017b}, the inequalities:
$b_n \leq \lambda_p(k^\cyc/k)\, n+ \mu_p(k^\cyc/k)\, p^n + \nu_p(k^\cyc/k)
\leq  b_n \times v_p(\CT_k(p))$, for $n \gg 0$, and similarly we have
conjectured, for $k$ fixed, the finiteness of primes $p$ such that $\CT_k(p) \ne 1$.
For $\CT_k(p) = 1$, we find again well-known result.

\begin{remarks}
(i) $\lambda_p(K/k) = 1$ and $\mu_p(K/k) = 0$ if and only if 
$b_n$ is bounded. 

\smallskip
(ii) If $K = k^\cyc$ and $p$ splits in $k$, the norm factors, for $i = 0$, 
in the relative extensions $K_m/K_n$, have orders of the form 
$p^{m-n}$, while these orders are $\ffrac{p^{m-n}}
{(E_n : E_n \cap \BN_{m/n}(K_m^\times)}$ for $K \ne k^\cyc$;
moreover, $\CH_n$ never capitulates in $k^\cyc$ (see  
\S\,\ref{obstructions} and Remark \ref{noncap}). So, large 
$\lambda_p(k^\cyc/k)$'s may occur since $\mu_p(k^\cyc/k) = 0$
in the abelian case.

\smallskip
(iii) From our point of view, for any $n \geq 1$ and $i \geq 2$, the 
filtration does not only depend on canonical elements, as the 
$S_k$-units of $k$, but of random groups $\Lbda_n^i$.

\smallskip
By nature, Iwasawa's theory does not give an algorithm likely to 
``construct'' $\lambda, \mu, \nu$ and remains purely algebraic. 
Many heuristics are suggested for $\lambda_p(k^\cyc/k)$ as that 
of Ellenberg--Jain--Venkatesh \cite{EJV2011} using random matrices.
\end{remarks}

\section{Generalizations of Gold's criterion}\label{gold}

\subsection{The fundamental relation \texorpdfstring{$\order 
\wt \CH_k = \order (\CH_{K_n}^2/\CH_{K_n}^1)$}{Lg}} 

We focus on the elements $\CH_{K_n}^1$ and $\CH_{K_n}^2$ of the 
filtrations; they are somewhat canonical and bring a maximal information 
on the $\CH_{K_n}$'s in a general setting about ramification in $K/k$ 
and about the $p$-class group of~$k$; we will see that the order of the 
quotient gives the order of the logarithmic class group.

\begin{theorem}\label{stability}
Let $k$ be a totally $p$-adic imaginary quadratic field of class number 
$\hk$. Let $K$ be a $\Z_p$-extension distinct from $L$ and $\ov L$,
with ${\mathfrak p}$ and ${\ov {\mathfrak p}}$ totally ramified from $K_e$ and 
$K_{\ov e}$, respectively, $e \geq \ov e$ (cf. Diagram \ref{maindiagram}); 
let $p^{\expo_k}$ be the exponent of $\CH_k$ and let $\hp 
\mid \hk$ be the order of $\Bcl({\mathfrak p})$. Assume
$n \geq \max \big(\expo_k,  e + \delta_p(k) \big)$
in the two following claims (i) and (ii):

\smallskip
(i) The order of the logarithmic class group $\wt \CH_k$ is given by the order of 
the jump $\CH_{K_n}^2/\CH_{K_n}^1$, whatever $K/k$ as defined above.

\smallskip
(ii) The condition $\CH_{K_n}^2 / \CH_{K_n}^1 = 1$ 
(i.e. $\CH_{K_n} = \CH_{K_n}^1$, $b_n = 1$) holds if and only if 
$\wt \CH_k = 1$; hence, $\CH_{K_n}^2 = \CH_{K_n}^1$ is equivalent to 
satisfy the two conditions $\delta_p(k) = 0$ and $\CH_k^{S_k} = 1$. 

\smallskip
(iii) If $\CH_{K_n}^2 = \CH_{K_n}^1$, the Iwasawa formula is
$\order \CH_{K_n} = p^{n + v_p(\hk) - e - \ov e}$,
for all $n \geq e$.
\end{theorem}

\begin{proof}

{\bf a)} {\bf  Computation of the first two elements of the filtration}.
Consider Formulas \eqref{filtration} for $i \leq 1$, still using the
previous simplified notations; for the moment we are not making 
any hypothesis on $n$:
\begin{equation*}
\left\{\begin{aligned}
\order \CH_n^1 & = \order \BN_n(\CH_n) \times p^{\max(0,n - e)}, \\
\order (\CH_n^2/\CH_n^1) & = \frac{\order \BN_n(\CH_n)}
{\order \BN_n(\CH_n^1)} \times  \frac{p^{\max(0,n - e)}}
{\order \omega_n(\Lbda_n^1)}, \\
\Lbda_n^1& = \{a \in k^\times,\ \, (a) = \BN_n({\mathfrak A}_n),\, 
\Ccl_n({\mathfrak A}_n) \in \CH_n^1 \}.
\end{aligned}\right.
\end{equation*}

We have the following classical exact sequence, where $\BJ_n$ is the 
transfer map and $\CH_n^\ram$ the subgroup of $\CH_n$ generated 
by the classes of the products of ramified primes $\CP_{\!\!n} 
:= {\mathfrak P}_n^1\cdots {\mathfrak P}_n^{p^{g_n}}$, with 
${\mathfrak P}_n^j \!\mid\! {\mathfrak p}\! \mid\! p$ in $K_n/k/\Q$, 
and $\CP'_{\!\!n} := {\mathfrak P}'^1_n \cdots 
{\mathfrak P}'^{p^{\ov g_n}}_n$, with ${\mathfrak P}'^j_n 
\!\mid\! {\ov {\mathfrak p}}\! \mid\! p$ in $K_n/k/\Q$ \,\footnote{\,\label{Kprime}
To avoid any ambiguity with complex conjugation $\tau : K_n \mapsto \ov K_n$ 
(since for instance $\ov {\mathfrak P}_n$ is in $\ov K_n$, above 
${\ov {\mathfrak p}}$, but not above ${\ov {\mathfrak p}}$ in $K_n$), we use
the notations ${\mathfrak P}_n \mid {\mathfrak p}$ in $K_n$,
${\mathfrak P}'_n \mid {\ov {\mathfrak p}}$ in $K_n$.
On the other hand, $p^e$, $p^{\ov e}$, $p^f$, $p^{\ov f}$, $p^g$, $p^{\ov g}$ 
are associated to the decomposition of ${\mathfrak p}$, ${\ov {\mathfrak p}}$ in 
$K_n$; for $n \geq e$, they do not depend on $n$ and since ramification indices 
are $p^{n - e}$, $p^{n - \ov e}$, we have the classic formula $p^{f+g} = p^e$ and 
$p^{\ov f+ \ov g} = p^{\ov e}$.}:
\begin{equation}\label{ambigeES}
1 \to \BJ_n(\CH_k) \cdot \CH_n^\ram \too \CH_n^1 \too 
E_k \cap \BN_n(K_n^\times)/\BN_n(E_n) \to 1. 
\end{equation}

Since $p$ totally split in the imaginary quadratic field $k$,
one obtains the identity:
$$\CH_n^1 = \BJ_n(\CH_k) \cdot \CH_n^\ram =: \Ccl_n(\CI_n^1)$$

\noindent 
for a suitable ideal group $\CI_n^1$ built with extensions $\BJ_n({\mathfrak a}_j)$
of ideals ${\mathfrak a}_j$ of $k$ generating $\CH_k$ and with the products 
$\CP_{\!\!n}$, $\CP'_{\!\!n}$. Taking the norm, we obtain:
$$\BN_n(\CH_n^1) = \CH_k^{p^n} \cdot \BN_n(\CH_n^\ram) = \CH_k^{p^n} \cdot 
\BN_n(\Ccl_n \langle\, \CP_{\!\!n}, \CP'_{\!\!n}  \,\rangle). $$ 

Let $\CI_k =: \langle {\mathfrak a}_1, \ldots , {\mathfrak a}_r \rangle$, 
$r = \rk_p(\CH_k) \geq 0$, be a subgroup of prime-to-$p$ ideals of $k$, 
whose classes generate $\CH_k$; in other words, the ideal group 
$\CI_k^1$ of $k$, leading to $\BN_n(\CI_n^1)$ and to $\BN_n(\CH_n^1)$ 
for the computation of $\Lbda_n^1$ (\S\,\ref{inclusions}(ii)), is of the form:
\begin{equation}\label{directsum}
\BN_n (\CI_k^1) := \CI_k^{p^n} \oplus 
\BN_n(\langle\, \CP_{\!\!n}, \CP'_{\!\!n}  \,\rangle)
= \langle {\mathfrak a}_1^{p^n}, \ldots , {\mathfrak a}_r^{p^n} \rangle \oplus 
\BN_n(\langle\, \CP_{\!\!n}, \CP'_{\!\!n}  \,\rangle). 
\end{equation} 

Note that if ${\mathfrak p}$ is one of the ${\mathfrak a}_i$'s or
divides it, one may find a prime-to-$p$ representative of the class; 
in other words we have a direct sum for ideals, but not necessarily 
for the corresponding class groups. 
Then ${\mathfrak a}_j^{p^n} = (a_j^{p^{n - {\epsilon}_j}})$, $a_j \in k^\times$,
where $p^{{\epsilon}_j} \leq p^{\expo_k}$ is the order of $\Ccl({\mathfrak a}_j)$, 
under the condition $n \geq \expo_k$.
So, $\langle\, (a_1^{p^{n - {\epsilon}_1}}), \ldots , (a_r^{p^{n - {\epsilon}_r}})\,\rangle 
\subset \BN_n(\CI_{\!n}^1)$, and these $a_j^{p^{n - {\epsilon}_j}}$ are 
by definition a part of the generators of $\Lbda_n^1$ that will be completed 
by some $S_k$-unit $y$ of minimal valuation, such that
$(y) \in \BN_n(\langle\, \CP_{\!\!n}, \CP'_{\!\!n}\,\rangle)$
for which we will compute the order of $\omega_n(y)$.

But, for $n$ large enough, $\omega_n(a_j^{p^{n - {\epsilon}_j}}) = 
\omega_n(a_j)^{p^{n - {\epsilon}_j}}$ would be of order less than 
$p^{\,{\epsilon}_j}$ in the cyclic group $\Omega_n$ of order $p^{\max (0, n - e)}$; 
whence, the computation of the norm factor for $n \geq e$, where only 
the $S_k$-unit $y$ will intervene since $\omega_n(y)$ will have the large 
order $p^{n - e - \delta_p(y)}$ (Theorem \ref{mainorder} and relation \ref{order}) 
for a constant $\delta_p(y)$ computed below (see \eqref{ycomput},
\eqref{ycomputbis}):
\begin{equation}\label{deltay}
\ffrac{p^{n - e}}{\order \omega_n (\Lbda_n^1)} = 
\ds \ffrac{p^{n - e}}{\order \omega_n \big(\langle \ldots, 
a_j^{p^{n - {\epsilon}_j}}, \ldots ; y \rangle \big)} = \ds \ffrac{p^{n - e}}
{\order \omega_n (\langle y \rangle )} = p^{\delta_p(y)}
\end{equation}

All this is due to the fact that these $\delta_p$'s come from random
``Fermat quotients'', very small and conjecturally zero for almost
primes $p$ (see Remark \ref{borel-cantelli} and Conjecture \ref{pfinitude}).

\medskip
{\bf b)} {\bf Computation of $y$ and order of $\omega_n(y)$}.
Since ${\mathfrak p}$ is totally ramified in $K/K_e$ and unramified in 
$K_e/k$, we have $\BJ_n ({\mathfrak p}) = \CP_n^{\,p^{\max(0, n-e)}}$;
similarly, $\BJ_n ({\ov {\mathfrak p}}) = \CP'^{\,p^{\max(0, n - \ov e)}}_n$.
Taking the norm, we get
$\BN_n(\CH_n^\ram) = \langle\, \Ccl({\mathfrak p}^{p^{\min(n, e)}}), 
\Ccl({\ov {\mathfrak p}}^{p^{\min(n, \ov e)}})\,\rangle
= \langle\,\Ccl({\ov {\mathfrak p}}^{p^{\min(n, \ov e)}})\,\rangle = 
\langle\,\Ccl({\mathfrak p}^{p^{\min(n, \ov e)}})\,\rangle$
since $e \geq \ov e$ and ${\mathfrak p} {\ov {\mathfrak p}} = (p)$. 
Whence:
\begin{equation}\label{normPn}
\BN_n(\CH_n^\ram) = \langle\,\Ccl( {\mathfrak p}^{p^{\min(n, \ov e)}})\,\rangle, 
\ {\rm and}\  \BN_n(\CH_n^1) = 
\langle\,\Ccl( {\mathfrak p}^{p^{\min(n, \ov e)}})\, \rangle
\  {\rm for}\  n \geq \expo_k.
\end{equation}

So, $\omega_n(\Lbda_n^1)$ will be generated by the image of a suitable 
$S_k$-unit $y$, of minimal valuation, {\it  following by definition the necessary 
condition $(y) \in \BN_n(\CI_{\!n}^1)$ to be in $\Lbda_n^1$}.

\smallskip
One has $\BN_n(\CP_n) = {\mathfrak p}^{p^e}$ and
$\BN_n(\CP'_n) = {\ov{\mathfrak p}}^{\,p^{\ov e}}$, with $(x) = 
{\mathfrak p}^{\hp}$ and $(\ov x) = {\ov{\mathfrak p}}^{\,\hp}$.
Thus, if $v_p(\hp) \leq \ov e$, then $\ov x_{}^{p^{\ov e-v_p(\hp)}}$
satisfies the conditions; otherwise, when $v_p(\hp) \geq \ov e$,
the minimal power giving a solution is $p^{v_p(\hp)}$ and gives  $y = \ov x$. 

\smallskip
Whence, $y = \ov x^{\,p^{\max(0, \ov e - v_p(\hp))}}$ is the minimal solution
yielding to (for $n$ large enough):
\begin{equation}\label{ycomput}
\left\{\begin{aligned}
\delta_p(y) & = \delta_p(k) + 0,\ & 
\hbox{if $v_p(\hp) \geq \ov e$}, \\ 
\delta_p(y) & = \delta_p(k) + 
\big [\ov e - v_p(\hp) \big],
\ & \hbox{if $v_p(\hp) < \ov e$}. 
\end{aligned}\right.
\end{equation}

This formula may be written, when $n \geq e \geq \ov e$:
\begin{equation}\label{ycomputbis}
\delta_p(y) = \wt \delta_p(k) - v_p(\BN_n(\CH_n)) +
\max(0, \ov e - v_p(\hp)),
\end{equation}

\noindent
since $\Gal(H_k^\nr /K^\nr) = \Gal(H_k^\nr / K_{\ov e})$
is the norm of $\CH_n$ in $K_n/k$.

\smallskip
Whence the formulas, where ``\,$m \geq \delta_p(x) = \delta_p(k)$\,'' 
in the proof of Theorem \ref{mainorder} translates here to 
``\,$n - e \geq \delta_p(y)$\,''. So, we must have
$n \geq  e + \delta_p(y)$.
Taking into account \eqref{normPn}, the two cases in \eqref{ycomput}, 
and $[K_n^\nr : k] = [K^\nr : k] = p^{\ov e}$, leads to (with \eqref{deltay}):
\begin{equation}\label{formulas}
\left\{\begin{aligned}
\order \CH_n^1 & = [H_k^\nr : K^\nr] \times p^{n - e} =
[\knr : K^\nr] \times \order \CT_k \times p^{n - e};  \\
\order (\CH_n^2/\CH_n^1) & = 
\frac{[H_k^\nr : K^\nr]}{\order \langle \,\Ccl({\mathfrak p}^{p^{\ov e}})\, \rangle}
\times p^{\delta_p(y)}  \\
 & = \left\{\begin{aligned}
& \frac{[H_k^\nr : K^\nr]}{p^{v_p(\hp)} \cdot [K^\nr : k]^{-1}}
\times p^{\delta_p(k)}\ \, \big[{\rm if}\ v_p(\hp) \geq \ov e \big]  \\
& \frac{[H_k^\nr : K^\nr] }{1} \times
\frac{[K^\nr : k]}{p^{v_p(\hp)}} \times p^{\delta_p(k)} \ \, 
\big[{\rm if}\ v_p(\hp) < \ov e\big]
\end{aligned}\right. \\
& = \frac{[H_k^\nr : k]} {p^{v_p(\hp)}} \times p^{\delta_p(k)} 
 = \frac{\order \CH_k}{p^{v_p(\hp)}} \times p^{\delta_p(k)} 
 = \order \wt \CH_k.
\end{aligned}\right.
\end{equation}

This final formula implies a complete proof of (i) and (ii); point (iii)
comes from the fact that $\CH_k = \langle \,\Ccl({\mathfrak p})\,\rangle$,
so the the condition $n \geq \expo_k$ does not intervene
(see \eqref{directsum}).
\end{proof}

Some conditions, equivalent to the triviality of $\wt \CH_k$, often required 
in the literature, are nothing else than the stability from $i = 1$ giving
$\CH_n = \CH_n^1$, for all $n$ large enough, hence, $\mu_p(K/k) = 0$
and $\lambda_p(K/k) = 1$. 

\smallskip
A consequence of expressions \eqref{formulas} is that $\order \CH_n^1$ and 
$\order \CH_n^2$ do not depend on the $\Z_p$-extension $K/k$, assumed to 
be ramified at the $p$-places from some explicit layers (which applies in particular 
to $k^\cyc$ and $k^\acyc$). Of course, it will be necessary to study $\order \CH_n^3$ 
to see if it is random or if some general property holds for all $K/k$ or not.
We shall give more details and consequences of the above theorem, to 
improve some classical results.

\begin{remarks} \label{lognonsplit}
Consider the case where $p$ does not split in $k$; then $\wt k/k$ is totally 
ramified from $\knr = \wt k \cap H_k^\nr = k_{\wt e}^\acyc$ for some 
$\wt e \geq 0$. The prime ideal ${\mathfrak p} \mid p$ in $k$ is $p$-principal 
($v_p(\hp) = 0$), and then totally splits in $H_k^\nr$.
We then have the following properties:

(i) It is immediate that $H_k^\lc = k^\cyc H_k^\nr$, giving $\wt \CH_k \simeq 
\Gal(H_k^\lc / k^\cyc) \simeq \CH_k$. The previous Theorem \ref{stability} applies, 
defining $\delta_p(k) = 0$, which corresponds, in some sense, to the convention 
$\log_{\mathfrak p}(p) = \log_{\ov {\mathfrak p}}(p) = 0$ in \S\,\ref{log}, since the 
Fermat quotient $\delta_p(k) = \frac{1}{p} (\ov x^{\,p-1} -1)$ is, when it 
makes sense, related to $\frac{1}{p} \log(\ov x)$, and in the non-split case,
$x = \ov x = p$.

(iii) We can define $H_k^\lac$, the maximal abelian locally anti-cyclotomic 
pro-$p$-extension of $k$. Similarly, $H_k^\lac = k^\acyc H_k^\nr$, then 
$\Gal(H_k^\lac/k^\acyc) \simeq \CT_k$.

\smallskip
There is some interest in studying capitulation in the real non-split case 
of ordinary $p$-classes since one obtains exactly Greenberg's criterion 
\cite[Theorem 1]{Gree1976}, and to extend this study in the imaginary case
even if the logarithmic class group seems more suitable. 

\smallskip
It will be interesting to test capitulations in the first layers of the 
$\Z_p$-extensions in the imaginary case; this needs to be able to conjugate the 
logarithmic classes to calculate the algebraic norm of $\wt \CH_n$ giving
$\Bnu_n(\wt \CH_n) = \BJ_n \circ \BN_n(\wt \CH_n) = \BJ_n(\wt \CH_k)$ in the 
surjective case.
\end{remarks}

\subsection{Gold's criterion and Gold--Sands improvement}

\subsubsection{The \texorpdfstring{$\lambda$}{Lg}-stability theorem in the totally 
ramified case}

To give an interesting application, we will use our theorem of $\lambda$-stability
proved in \cite{Gra2022}, applied to the modules $\CX_n = \CH_{K_n}$ 
and to $X_k = E_k$, defining the number $\lambda := 
\max(0, \order S_k - 1 - \rho_k)$ (not to be confused with the Iwasawa 
invariant $\lambda_p(K/k)$), with $\rho_k = \rk_{\Z}^{}(X_k) = 0$, thus giving 
$\lambda = 1$ for totally $p$-adic imaginary quadratic fields, otherwise 
$\lambda = 0$.

\smallskip
We have the following properties of $\lambda$-stability in any totally 
ramified $\Z_p$-extension $K/k$, without making any assumption 
on $\CH_k$ \cite[Theorem 3.1]{Gra2022}:

\begin{theorem}
Let $k$ be an imaginary quadratic field. Let $\lambda := \order S_k - 1$.
Assume $p \geq 3$ and let $K/k$ be a totally ramified $\Z_p$-extension. 
We have $\order \CH_{K_n} = \order \CH_k \cdot p^{\lambda \cdot n}$ for all 
$n \geq 0$ if and only if $\order \CH_{K_1} = \order \CH_k \cdot p^{\lambda}$. 
If this property is satisfied, we have, for all $n \geq 0$, $\CH_{K_n} = \CH_{K_n}^{G_n}$, 
$\BJ_{K_n/k}(\CH_k) = \CH_{K_n}^{p^n}$ and $\Ker(\BJ_{K_n/k}) = 
\BN_{K_n/k}(\CH_{K_n}^{[p^n]})$, where $\CH_{K_n}^{[p^n]} := 
\{c \in \CH_{K_n},\ c^{p^n}=1\}$. 
\end{theorem}

This is useful in the split case of $p$ in $k$ ($\lambda = 1$) since, 
in $\wt k/k$, the $p+1$ first layers are numerically accessible, using 
Kummer's theory and reflection principle in $k(\mu_p^{})$, giving 
possibly infinitely many $\lambda_p(K/k) = 1$ as soon as $K_1/k$ is 
ramified at the two $p$-places and $\order \CH_{K_1} = \order \CH_k \cdot p$. 
See Section\,\ref{initial}.

\subsubsection{Other \texorpdfstring{$\Z_p$}{Lg}-extensions such that
\texorpdfstring{$\order \CH_n = \order \CH_k\! \cdot\! p^n$, for all
$n \geq 0$}{Lg}}

We obtain a variant of Gold's criterion which applies to all totally ramified
$\Z_p$-extensions $K/k$, assuming that $\CH_k = \big\langle\,
\Ccl({\mathfrak p})\, \big\rangle$ (hence $\order \wt \CH_k = p^{\delta_p(k)}$):

\begin{theorem}\label{goldplus}
Let $k$ be an imaginary quadratic field and let $p \geq 3$ split in $k$. 
We assume that the $S_k$-class group $\CH_k^{S_k}$ is trivial
(whence, $v_p(\hp) = v_p(\hk)$). Let $K/k$ be a $\Z_p$-extension, 
totally ramified at the $p$-places ($e = \ov e = 0$).
Then, the two following properties are equivalent:

\smallskip
\quad $(i)$ $\delta_p(k) = 0$;

\smallskip
\quad $(ii)$ $\order \CH_{K_n} = \order \CH_k \cdot p^n = 
p^{n + v_p(\hp)}$, for all $n \geq 0$.
\end{theorem}

\begin{proof}
Since $p$ totally ramifies in $K/k$ with $\CH_k = \langle\,\Ccl({\mathfrak p})\,\rangle$, 
the condition on $n$ in Remark \ref{orderomega} is simply $n\geq 0$ since 
$e = \ov e=0$; we recover that $\order \CH_n^1 = \order \CH_k \times p^n$ 
from Chevalley--Herbrand formula, and that $\order (\CH_n^2/\CH_n^1) = \ds
\ffrac{\order \CH_k}{\order \BN_n(\CH_n^1)} \times 
\ffrac{p^n}{\order \omega_n(\Lbda_n^1)}$, where one verifies that
$\Lbda_n^1$ is generated by $x$ since (from \eqref{ambigeES})
$\BN_n(\CH_n^1) = \langle \Ccl({\mathfrak p}) \rangle^{p^n} \BN_n(\CH_n^\ram) 
= \langle \Ccl({\mathfrak p}) \rangle$, 
giving $\Lbda_n^1 = \{a \in k^\times,\ (a) \in \BN_n(\CH_n^1)\} = 
\langle \,x \, \rangle$ and $\order \langle \,\omega_n(x)\, \rangle
= p^{\max (0,n-\delta_p(k))}$ (Theorem \ref{mainorder} and 
Relation \ref{order}); whence $\ds \ffrac{p^n}{\order \omega_n(\Lbda_n^1)}
= p^{\min (0,n-\delta_p(k))+\delta_p(k)}$, thus giving, for all $n \geq 0$:
\begin{equation}\label{E1} 
\left\{\begin{aligned}
{a)} \ \  \order \CH_n^1 & = p^{n+v_p(\hp)}, \\
{b)} \ \  \order \CH_n^2 & = \order \CH_n^1 \times  
p^{\min (0,n-\delta_p(k))+\delta_p(k)} .
\end{aligned}\right.
\end{equation}

(i) If $\delta_p(k) = 0$, $\order \CH_n^2 = \order \CH_n^1$ from 
\eqref{E1}\,(b), the filtration stops, and then $\order \CH_n = \order \CH_n^1 
= p^{n + v_p(\hp)}$ for all $n \geq 0$ from (a), whence 
$\lambda_p(K/k) = 1$ and $\nu_p(K/k) = v_p(\hp)$.

\smallskip
(ii) If $\order \CH_n = p^{n + v_p(\hp)}$, for all $n \geq 0$, 
then $\order \CH_n = \order \CH_n^1$ from \eqref{E1}\,(a); this 
stabilization at $i = 1$ implies necessarily $\order \CH_n^2 = \order \CH_n^1 
= p^{n+v_p(\hp)}$, but $\order \CH_n^2/\CH_n^1 = 
p^{\min (0,n-\delta_p(k))+\delta_p(k)}$ from \eqref{E1}(b); taking for 
instance $n = \delta_p(k)$, one gets $\delta_p(k) = 0$, yielding the 
reciprocal and the theorem.
\end{proof}

\subsubsection{Improvement of some classical results 
for the cyclotomic extension $k^\cyc$}

We still assume that $p$ splits in $k = \Q(\sqrt{-m})$.
We give before, in ({\bf a}), a result valid for any $\Z_p$-exten\-sion 
$K/k$ totally ramified ($e = \ov e = 0$), in particular $k^\cyc/k$:

\smallskip
({\bf a}) Case $\CH_k = 1$. Equalities \eqref{formulas} become:
\begin{equation*}
\order \CH_n^1 = p^n \ \ \ \& \ \ \ 
\order \CH_n^2 = p^{n + \delta_p(k)},\ \hbox{for all $n \geq 0$}, 
\end{equation*}

\noindent
and this is the framework of Gold criterion (as particular case of Theorem 
\ref{goldplus} when $v_p(\hk) = v_p(\hp) = 0$)
dealing, at the origin, with the cyclotomic $\Z_p$-extension $k^\cyc$ and 
Iwasawa's invariants $\lambda_p(k^\cyc/k)$, $\nu_p(k^\cyc/k)$, since 
$\mu_p(k^\cyc/k) = 0$ in this abelian case:

\begin{corollary}[Gold's criterion \cite{Gold1974}] \label{G1974}
Let $k^\cyc$ be the cyclotomic $\Z_p$-extension of $k = \Q(\sqrt{-m})$ 
and assume that $\CH_k = 1$. Then $\lambda_p(k^\cyc/k) = 1\, $ $\&$ 
$\, \nu_p(k^\cyc/k) = 0$, if and only if $\delta_p(k) = 0$. If so, $\CH_{K_n} 
= \langle\, \Ccl_n({\mathfrak P}_n)\,\rangle$, ${\mathfrak P}_n \mid (p)$
in $K_n$, of order $p^n$, for all $n \geq 0$.
\end{corollary}

The case $\CH_k = 1$, often considered in the literature, may be replaced by 
the interesting Sands remark \cite[Proposition 2.1, Theorem 3.3, Corollary 3.4]
{San1993}, only using the assumption $v_p(\hp) = 0$ (equivalent 
to $\order \wt \CH_k = \order \CH_k \times p^{\delta_p(k)}$), as follows:

\smallskip
({\bf b}) Case $v_p(\hp) = 0$. Properties of the filtration in 
the case $v_p(\hp) = 0$ and $K = k^\cyc$ give the 
Gold--Sands criterion; as we will see, this property is very specific of $k^\cyc$:

\begin{theorem}\label{sands}(Gold--Sands criterion).
Let $p \geq 3$ split in $k = \Q(\sqrt{-m})$ of class number $\hk$. 
Let $k^\cyc$ be the cyclotomic $\Z_p$-extension of $k$. When 
$\Ccl({\mathfrak p}) = 1$ in $\CH_k$ (i.e. $v_p(\hp) = 0$), then 
$\lambda_p(k^\cyc/k) \geq 2$ if and only if $v_p(\hk) \geq 1$ 
or $\delta_p(k) \geq 1$, whence if and only if $\wt \CH_k \ne 1$.
\end{theorem}

\begin{proof}
(i) Set $K = k^\cyc$ and prove that (under the condition 
$v_p(\hp) = 0$) $\wt \CH_k \ne 1$ implies 
$\lambda_p(K/k) \geq 2$ (note that $\lambda_p(K/k) \geq 1$ 
from Chevalley--Herbrand formula).

\smallskip
For this, assume that 
$\order \wt \CH_k = \order \CH_k \cdot p^{\delta_p(k)} \ne 1$ 
and $\lambda_p(K/k) = 1$ to get a contradiction:

\smallskip
Let $m \gg n \gg 0$, whence $\order \CH_m := \order \CH_{K_m} 
= p^{m + \nu}$, $\order \CH_n := \order \CH_{K_n} = p^{n + \nu}$, 
with $\nu = \nu_p(K/k)$. 
One considers the relative formulas for the filtration of $\CH_m$ with respect 
to $G_{m/n} := \Gal(K_m/K_n)$, $\omega_{m/n} := \omega_{K_m/K_n}$, and 
so on for the norms $\BN_{m/n}$ and transfers $\BJ_{m/n}$; let $\CH_m^{(1)}$ 
and $\CH_m^{(2)}$ be the first two elements of the filtration regarding $G_{m/n}$. 
We use the fact that, for all $n$, the unit group $E_n := E_{K_n}$ of 
$K_n = k\Q_n$ is that of $\Q_n$ (hence the group of cyclotomic units)
for which $\omega_{m/n}(E_{\Q_n}) = 1$ and $E_{\Q_n} \cap 
\BN_{\Q_m/\Q_n}(\Q_m^\times) = \BN_{\Q_m/\Q_n}(E_{\Q_m})$, 
then the principality of ${\mathfrak P}_m \ov {\mathfrak P}_m$, 
equal to the unique prime ideal of $\Q^\cyc_m$ above $p$:
\begin{equation*}
\left\{\begin{aligned}
\order \CH_m^{(1)} & = \order \CH_n \times 
\frac{p^{m-n}}{\order \omega_{m/n}(E_n)} = p^{(n+\nu) + (m-n)} = p^{m + \nu}, \\
\CH_m^{(1)} & = \BJ_{m/n}(\CH_n) \cdot \CH_m^\ram = 
\BJ_{m/n}(\CH_n) \cdot \langle\ \Ccl_m({\mathfrak P}_m)\ \rangle, 
\end{aligned}\right.
\end{equation*}
giving $\BN_{m/n} (\CH_m^{(1)}) = \CH_n^{p^{m-n}}\!\! \cdot 
\langle\ \Ccl_n({\mathfrak P}_n)\ \rangle = \langle\, \Ccl_n({\mathfrak P}_n)\, \rangle$
for $m \gg n$, and the following  inequality about the case $i = 1$ of the relative 
filtration: 
$$\order \big( \CH_m^{(2)}/\CH_m^{(1)} \big) = 
\frac{\order \CH_n}{\order \langle\, \Ccl_n({\mathfrak P}_n)\, \rangle}  
\times \frac{p^{m-n}}{\order \omega_{m/n}(\Lbda_{m/n}^1)} \geq  \frac{p^{n + \nu}}
{\order \langle\, \Ccl_n({\mathfrak P}_n)\, \rangle} \times \frac{p^{m - n}}{p^{m - n}}. $$

But ${\mathfrak P}_n^{p^n} = \BJ_n({\mathfrak p})$ is $p$-principal since
$v_p(\hp) = 0$ and $\Ccl_n({\mathfrak P}_n)$ is of order 
a divisor of $p^n$; then we get $\order \big( \CH_m^{(2)}/\CH_m^{(1)} \big) 
\geq p^\nu$; then $\order \CH_m^{(2)} \geq \order \CH_m^{(1)} \cdot p^\nu 
= p^{m + 2 \nu}$, whence the inequality $p^{m + \nu} = 
\order \CH_m \geq \order \CH_m^{(2)} \geq  p^{m + 2 \nu}$, 
which implies $\nu = 0$ and $\order \CH_n = p^n$, for all $n \gg 0$; but 
using the filtration with respect to $G_n = \Gal(K_n/k)$ gives: 
$$\order \CH_n = p^n \geq \order \CH_n^2 = 
\order \CH_n^1 \cdot \order \CH_k \,p^{\delta_p(k)}
= \order \CH_k \, p^n \cdot \order \CH_k \,  p^{\delta_p(k)}
= (\order \CH_k)^2 \times p^{n + \delta_p(k)}, $$

\noindent 
from \eqref{formulas} since $v_p(\hp) = 0$, whence 
$(\order \CH_k)^2 \times p^{\delta_p(k)} = 1$ (contradiction).

\smallskip
(ii) For the reciprocal (i.e. $\lambda_p(K/k) \geq 2$ implies 
$\order \wt \CH_k = \order \CH_k \times  p^{\delta_p(k)} \geq p$), 
let's assume $\lambda_p(K/k) \geq 2$ and prove that 
$\CH_k \ne 1$ or $\delta_p(k) \ne 0$ by contradiction:

\smallskip
$\bullet$ Assume $\CH_k = 1$; then, usual Gold's criterion 
(Corollary \ref{G1974}) implies $\delta_p(k) \geq 1$. 

\smallskip
$\bullet$ If $\delta_p(k) = 0$, then the hypothesis $\CH_k = 1$ of Gold's 
criterion can not hold, otherwise we would have $\lambda_p(K/k) = 1$;
so $\CH_k \ne 1$.
\end{proof}

\begin{remarks}
(i) Contrary to our elementary class field theory proof, 
the main part in the Sands improvement of Gold's criterion \cite[proof 
of Proposition 2.1 and Remark 2.2]{San1993}, comes from the deep 
general analytic result of Federer--Gross--Sinnott saying that the 
leading term of the characteristic polynomial for the basic unramified 
$p$-adic Iwasawa module of $K$ is, in our case (up to a $p$-adic unit),
$\hk \cdot\frac{1}{p}\log_p(x) \cdot T$ \cite[Proposition 3.9, 
Corollary 5.8 to Proposition 5.6]{FGS1981}, which is clearly 
related to $\order \wt \CH_k = \order \CH_k \cdot p^{\delta_p(k)}$.

\smallskip
(ii) The filtration may be somewhat important, giving large 
$\lambda_p(k^\cyc/k)$'s (examples in Dum\-mit--Ford--Kisilevsky--Sands 
\cite{DFKS1991}). For the three examples of $\lambda_p(k^\cyc/k) \geq 2$ 
given in \cite[Section 1, (4)] {Oza2001}, we compute $\order \wt \CH_k = 
p^{\wt \delta_p(k)} = p^{\delta_p(k)}$, in the case $v_p(\hp) = 0$, 
and check that $\delta_p(k) \ne 0$ ($\delta_p(k) = 1$ for $k = \Q(\sqrt{-52391})$, 
$p = 5$; $\delta_p(k) = 2$ for $k = \Q(\sqrt{-1371})$, $p = 7$; $\delta_p(k) = 3$ for 
$k = \Q(\sqrt{-23834})$, $p = 3$). In the case $k = \Q(\sqrt{-239})$ and $p = 3$, for 
which $\order \CH_k = 3$, we find $\wt \delta_p(k) = 1$, even if ${\mathfrak p}$ is 
$3$-principal and $\delta_p(k) = 0$ (thus $\wt \delta_3(k) = 
\delta_3(k) + v_3(\hk) - v_3(\hp) = 1$, as expected). 
Statistical results on  $\lambda_p(k^\cyc/k)$, regarding $\rk_p(\CH_k)$, are 
given in Ray \cite{Ray2023}, varying imaginary quadratic fields $k$, for $p$ fixed.
\end{remarks}

\subsection{Structure of the \texorpdfstring{$p$}{Lg}-class groups in 
\texorpdfstring{$k^\acyc$}{Lg}}

We have privileged the most difficult case $p$ split in $k$ in the study of the 
$\Z_p$-extensions of $k$; in that case the filtrations of the $p$-class groups 
depend on the class factor and on the more tricky norm factor. In the non-split 
case, the norm factor becomes trivial, and, both the formula of Chevalley--Herbrand  
and that giving $\order (\CH_n^2/\CH_n^1)$ allow immediate proofs of some 
interesting results, as that of Kundu--Washington \cite[Theorem 6.1]{KW2023}, 
for the anti-cyclotomic $\Z_p$-extension. Nevertheless we can obtain 
more general results, whatever the decomposition of $p$ in $k$:

\begin{theorem}\label{noncyclic}
Let $k^\acyc$ be the anti-cyclotomic $\Z_p$-extension of $k$, for $p \geq 3$. 
We assume that $\CH_k$ is cyclic non-trivial and that the $p$-Hilbert class 
field $H_k^\nr$ is not contained in $k^\acyc$. Put $k^\acyc \cap H_k^\nr 
= k^\acyc_e$, $e \geq 0$. Then, for all $n \geq e + 1$, the $p$-class group 
$\CH_{k^\acyc_n}$ is not cyclic.
\end{theorem}

\begin{proof}
Let $\CH_k{'} := \Gal(H_k^\nr/k^\acyc_e)$; by assumption, $\CH_k{'}$ is 
cyclic and non-trivial.
It suffices to assume that, for the layer $e+1$, $\CH_{k^\acyc_{e+1}}$ is cyclic 
and then to show that this leads to a contradiction, which implies the non-cyclicity 
of the $\CH_{k^\acyc_n}$'s for all $n \geq e+1$ (surjectivity of the norms in the 
totally ramified extensions $k_n^\acyc/k^\acyc_{e+1}$):
\unitlength=0.8cm
\begin{equation}\label{schema2}
\begin{aligned}
\vbox{\hbox{\hspace{-3.6cm}  
\begin{picture}(10.5,5.4)
\put(4.0,4.50){\line(1,0){2.6}}
\put(4.2,2.50){\line(1,0){2.5}}
\bezier{300}(3.5,4.85)(7.0,5.4)(10.5,4.85)
\put(7.0,5.3){\ft$\langle \varphi \rangle$}
\put(8.3,4.50){\line(1,0){2.0}}
\put(10.3,4.4){$H_{k^\acyc_{e+1}}^\nr \!=:\! H_{e+1}^\nr$}
\put(9.0,4.65){\ft$p$}
\put(4.9,4.68){\ft$\langle \ov \varphi \rangle$}
\put(3.50,2.9){\line(0,1){1.35}}
\bezier{200}(4.15,2.7)(4.5,3.5)(4.15,4.3)
\put(4.36,3.44){\ft$\langle \ov \sigma \rangle$}
\put(3.15,3.45){\ft$p$}
\put(3.6,1.5){\ft$p^e$}
\put(3.50,1.0){\line(0,1){1.2}}
\put(7.50,2.9){\line(0,1){1.2}}
\bezier{280}(3.1,0.6)(2.3,2.5)(3.1,4.4)
\put(2.0,2.4){\ft$\langle \sigma \rangle$}
\bezier{200}(3.9,0.5)(5.7,0.5)(7.2,2.1)
\put(6.25,0.9){$\CH_k$}
\put(4.9,2.7){\ft$\CH_k{'} \ne 1$}
\put(6.8,4.4){$k_{e+1}^\acyc H_k^\nr$}
\put(3.2,4.4){$k_{e+1}^\acyc$}
\put(6.8,2.4){$H_k^\nr$}
\put(7.9,2.50){\line(1,0){2.3}}
\put(10.6,2.9){\line(0,1){1.3}}
\put(10.4,2.4){$Z$}
\put(3.2,2.4){$k_e^\acyc$}
\put(3.4,0.40){$k$}
\end{picture}}} 
\end{aligned}
\end{equation}
\unitlength=1.0cm

Put $\CH_{k^\acyc_n} =: \CH_n$ and set $s = 2$ (resp. $s = 1$) if $p$ 
splits in $k$ (resp. if not). Since $p$ is unramified in $k^\acyc_e/k$ 
and totally ramified in $k^\acyc/k^\acyc_e$, we have the following 
formulas in $k^\acyc_{e+1}/k$: 
\begin{equation*}
\left \{\begin{aligned}
\order \CH_{e+1}^1 & = \order \BN_{e+1}(\CH_{e+1}) \cdot p^{s-1} =
\order \CH_k{'} \cdot p^{s-1}, \ \hbox{since $\order \Omega_{e+1} = p^{s-1}$, }\\
\order (\CH_{e+1}^2/\CH_{e+1}^1) & = \frac{\order \BN_{e+1}(\CH_{e+1})}
{\order \BN_{e+1}(\CH_{e+1}^1)} \cdot \frac{\order \Omega_{e+1}}
{\order \omega_{e+1}(\Lbda_{e+1})}
= \frac{\order \CH_k{'}}{\order \BN_{e+1}(\CH_{e+1}^1)} \cdot 
\frac{p^{s-1}}{\order \omega_{e+1}(\Lbda_{e+1})}, \\
\CH_{e+1}^1 & = \BJ_{e+1}(\CH_k) \cdot \CH_{e+1}^\ram ,\ \
\BN_{e+1}(\CH_{e+1}^1) = \CH_k^{p^{e+1}} \cdot 
\langle  \Ccl( {\mathfrak p} )^{p^e} \rangle^{s-1}
\hbox{(from \eqref{normPn})},
\end{aligned} \right .
\end{equation*}

\noindent
where $\CH_{e+1}^\ram = \Ccl\big(\langle\, \CP_{\!\! e+1},
\ov \CP_{\!\! e+1}\,\rangle \big)$, 
$\CP_{\!\! e+1} := \prod_{{\mathfrak P}_{e+1} 
\mid {\mathfrak p}} {\mathfrak P}_{e+1}$ in $k_{e+1}^\acyc/k$
($\ov \CP_{\!\! e+1} = \CP_{\!\! e+1}$ if $s = 1$).

\smallskip
So, in the split case $s = 2$, the above formulas become by cyclicity of $\CH_k$:
\begin{equation*}
\left \{\begin{aligned}
\order \CH_{e+1}^1 & = \order \CH_k{'} \cdot p, \\
\order (\CH_{e+1}^2/\CH_{e+1}^1) & = \frac{\order \CH_k{'}}
{\order \big(\CH_k^{p^{e+1}}\!\! \cdot \langle \Ccl({\mathfrak p})^{p^e} \rangle\big)} 
\cdot \frac{p}{\order \omega_{e+1}(\Lbda_{e+1})} \\
& = \frac{\order \CH_k{'}}
{\order \big(\CH_k{'}{}^p \cdot \langle \Ccl( {\mathfrak p} )^{p^e} \rangle \big)} 
\cdot \frac{p}{\order\omega_{e+1}(\Lbda_{e+1})};
\end{aligned} \right.
\end{equation*}

\noindent
since $\CH_k{'} \ne 1$, this yields $\order (\CH_{e+1}^2/\CH_{e+1}^1) \in \{1, p, p^2\}$.

(i) If $s = 1$, $\order (\CH_{e+1}^2/\CH_{e+1}^1) = 
\ds\ffrac{\order \CH_k{'}}{\order \CH_k{'}{}^p} = p$. 

(ii) Let's examine the case $\CH_{e+1}^2 = \CH_{e+1}^1$ in the split case
(whence $\CH_{e+1} = \CH_{e+1}^1$, {\it assumed cyclic of order 
$\order \CH_k{'} \cdot p$}), yielding $\CH_k{'}{}^p \cdot 
\langle  \Ccl( {\mathfrak p} )^{p^e} \rangle = \CH_k{'}$; this 
means that $\Ccl( {\mathfrak p})^{p^e} $ is a generator of $\CH_k{'}$, 
in other words that $\Ccl({\mathfrak p})$ generates $\CH_k$. 
Thus, $\CH_{e+1} = \CH_{e+1}^1 = \Ccl_{e+1} 
\big (\langle \BJ_{e+1}({\mathfrak p}),
\CP_{\!\! e+1}, \ov \CP_{\!\! e+1} \rangle \big )$; but 
$\BJ_{e+1}({\mathfrak p}) = \CP_{\!\! e+1}^p$, giving
$\CH_{e+1} = \Ccl_{e+1} \big (\langle \CP_{\!\! e+1}, 
\ov \CP_{\!\! e+1} \rangle \big )$ cyclic. 

So, $\Ccl_{e+1} (\ov \CP_{\!\! e+1}) = 
\Ccl_{e+1} (\CP_{\!\! e+1})^u$, for $u$ prime to $p$, and
$\Ccl_{e+1} \big (\CP_{\!\! e+1} \big)$ is a generator of $\CH_{e+1}$.
It follows that the class of all the conjugate ideals ${\mathfrak P}_{e+1} 
\mid {\mathfrak p}$ in $k_{e+1}^\acyc$ are of same order 
$\order \CH_k{'} \cdot p$, and that a Frobenius $\varphi := 
\big (\ffrac{H_{e+1}^\nr/k_{e+1}^\acyc}{{\mathfrak P}_{e+1}} \big)$ 
is of same order and generates $\Gal(H_{e+1}^\nr/k_{e+1}^\acyc)$.

Let $\ov \sigma = \sigma^{p^e}$ be a generator of $\Gal(k_{e+1}^\acyc/k_e^\acyc)$, 
and let $Z$ be the inertia field of ${\mathfrak P}_{e+1}$ in $H_{e+1}^\nr/k$
(see Diagram \ref{schema2}); then $\Gal(H_{e+1}^\nr/Z) \simeq 
\langle \ov \sigma \rangle$ and:
$$\varphi^{\ov \sigma} 
= \Big ( \ffrac{H_{e+1}^\nr/k_{e+1}^\acyc}{{\mathfrak P}_{e+1}} \Big)^{\ov \sigma}
= \Big ( \ffrac{H_{e+1}^\nr/k_{e+1}^\acyc}{{\mathfrak P}_{e+1}^{\ov \sigma}} \Big) 
= \Big ( \ffrac{H_{e+1}^\nr/k_{e+1}^\acyc}{{\mathfrak P}_{e+1}} \Big) = \varphi, $$

\noindent
proving that $\Gal(H_{e+1}^\nr/k_e^\acyc)$ is abelian and that $Z$
does not depend on the choice of ${\mathfrak P}_{e+1}$; thus, 
$Z/k_e^\acyc$ is abelian unramified of degree $\order \CH_k{'}\! \cdot\! p$ 
(absurd). So, $\order (\CH_{e+1}^2/\CH_{e+1}^1)\in \{p, p^2\}$. 

\smallskip
(iii) By the assumed cyclicity of $\CH_{e+1}$, $\CH_{e+1}^2$ and 
$\CH_{e+1}^1$ are cyclic; let $h_2$ and $h_1 = h_2^{p^t}$, $t \in \{1, 2\}$,
be 
generators of $\CH_{e+1}^2$ and $\CH_{e+1}^1$, respectively.
We have the exact sequence, where $\sigma$ is a generator of 
$\Gal(k^\acyc_{e+1}/k)$:
$$1 \to \CH_{e+1}^1 \too \CH_{e+1}^2 \mathop {\tooo}^{1 - \sigma} X 
\to 1,\ \, X \subseteq \CH_{e+1}^1\ \&\ \hbox{$\order X = p^t $}. $$

There exists a $p$-power $a$ such that $h_2^{a p^t}$ is a generator of 
$X \ne 1$, hence of the form $h_2^{b (1 - \sigma)}$, with $p \nmid b$, 
which gives $h_2^{a p^t - b (1-\sigma)} = 1$. 

\smallskip
Let $h \in \CH_{e+1}^2$; then 
$h^{(1-\sigma)^2} = 1$ and $(h^{(1-\sigma)^2})^\tau = h^{[\tau(1-\sigma)^2]} = 1$ 
(with left $\Z[\tau]$-modules laws), but $\tau(1-\sigma)^2 = (1-\sigma^{-1})^2 \tau
= \sigma^{-2} (\sigma - 1)^2 \tau$; this shows that $h^\tau \in \CH_{e+1}^2$. 
Since $\CH_{e+1}^2$ is cyclic and stable by $\tau$, $h_2^\tau = 
h_2^u$, with $p \nmid u$, whence $h^\tau = h^u$ for all $h \in \CH_{e+1}^2$. 
The action of $\tau$ on the relation $h_2^{a p^t - b (1 - \sigma)} = 1$ implies:
$$1 = h_2^{[\tau [a p^t - b (1 - \sigma)]]} 
= h_2^{[a p^t \tau - b(1 - \sigma^{-1})\tau]} 
= h_2^{[[a p^t  - b \sigma^{-1} (\sigma -1)] \,u]}. $$

From the relations $h_2^{[a p^t - b (1 - \sigma)]} = 1$
and $h_2^{[a p^t  - b\sigma^{-1} (\sigma -1)]} = 1$, we deduce that:
$$a p^t - b (1 - \sigma)-[a p^t  - b \sigma^{-1} (\sigma -1)]
= b (1 + \sigma^{-1}) (\sigma -1) $$

\noindent
annihilates $\CH_{e+1}^2$ with $b (1 + \sigma^{-1})$ invertible (absurd 
since $t \in \{1, 2\}$ implies $\CH_{e+1}^2 \ne \CH_{e+1}^1$).
Whence the non-cyclicity of the $p$-class groups of $k_n^\acyc$ for
all $n \geq e+1$.
\end{proof}

The case $e = 0$ corresponds to the total ramification of $k^\acyc/k$, 
with $\CH_k \ne 1$, cyclic, giving $\CH_n$ non-cyclic for all $n \geq 1$; 
it is the context which may be numerically checked (see numerical results 
given by Program \ref{P7} showing that all the hypothesis are necessary).

\subsection{Specificity of \texorpdfstring{$k^\cyc$}{Lg} regarding others
\texorpdfstring{$K$}{Lg}'s} \label{obstructions}

We will try to explain why unit groups, even for the non-Galois  
$\Z_p$-extensions, play a fundamental role for Iwasawa's invariants, contrary
to the unique case $K = k^\cyc$. Indeed, one may ask why Theorem \ref{sands} 
is not general.

\begin{remark}\label{noncap}
Recall that minus parts of $p$-class groups and logarithmic class groups 
of $k$ do not capitulate in $k^\cyc$ for $k$ imaginary, and that the modules 
$\CT_k$ never capitulate. The fact that $\CT_k$ is isomorphic to a subgroup 
of $\CH_k$ (Diagram \ref{maindiagram}) and can not capitulate, contrary to 
$\CH_k$, is not a contradiction since $\BJ_n$ does not necessarily commutes 
with isomorphisms; indeed, $\Ccl_k({\mathfrak a}) \in \CH_k$ 
capitulates in $F/k$ if $({\mathfrak a})_F = (\alpha_F)$, $\alpha_F \in F^\times$; but 
$\Ccl_{k,\infty}({\mathfrak a}) \in \CT_k$ capitulates in $F$ if, for all $N$, 
$({\mathfrak a})_F = (\alpha_{F,p^N})$, $\alpha_{F,p^N} \equiv 1 \pmod{p^N}$,
which implies ${\mathfrak a} = (a_{k,p^N})$ in $k$, for all $N$, hence
$\Ccl_{k,\infty}({\mathfrak a}) = 1$ (injectivity of the transfer $\CT_k \mapsto \CT_F$
under Leopoldt's conjecture).
\end{remark}

We describe the four cases of units groups, where $k$ is abelian real or imaginary:

\smallskip
{\bf a}) {\bf $K = k^\cyc$, for $k$ real}.
Recall that Greenberg's conjecture \cite{Gree1976} ($\lambda_p(K/k) 
= \mu_p(K/k) = 0$ for totally real base field $k$, whatever the splitting of $p$), 
is equivalent to the capitulation of $\wt \CH_k$ in $K$. 
The phenomenon is due to the existence of $E_{K_n}$ of maximal $\Z$-rank 
$[K_n : \Q]-1 = 2 p^n - 1$, giving, conjecturally, trivial norm factors in the 
tower from some layer large enough. Real Greenberg's conjecture has been 
largely checked in Kraft--Schoof--Pagani \cite{KS1995,Pag2022}, for $p = 3$ 
and $p = 2$, respectively, then in \cite{MPS2025}.

\smallskip
{\bf b}) {\bf $K = k^\cyc$, for $k$ imaginary}.
In this case, $E_{K_n} = E_{\Q_n^\cyc}$, and we can say, by abuse of 
language, that, in the norm point of view, $K$ has no proper 
units, thus giving largest $\order \CH_{K_n}^i$'s
(e.g., $\order \CH_{K_m}^{G_{m/n}} = \order \CH_{K_n} \cdot p^{(m-n)\,
(\order S_k-1)}$), and, as we have seen, $\lambda_p(K/k)$ may be large.

\smallskip
{\bf c}) {\bf $K = k^\acyc$ or $K$ is a non-Galois $\Z_p$-extension of 
$k$ imaginary}. 
We have $\rk_\Z(E_{K_n}) = p^n - 1$, and for $n \gg 0$,
$E_{K_n} \cap E_{\Q^\cyc} = E_{\Q^\cyc_{\beta}}$, where $\beta$ depends 
on $K \cap k^\cyc$ (see Appendix \ref{B}); this group of real units 
is not necessarily norm in relative $K_m/K_n$'s and does not matter. 
So, any $K/k$ distinct from $k^\cyc/k$, behaves like the totally real 
case for $k^\cyc/k$ because, still conjecturally, $\omega_{m/n}(E_{K_n}) 
= \Omega_{m/n}$ for $m \gg n \gg 0$.

\smallskip
Numerical examples are given \S\,\ref{verifcap}.

\smallskip
In conclusion, we propose the following conjecture in the semi-simple case
(an essential context, otherwise examples of $\mu_p(K/k) \geq 1$ may be 
obtained; see \S\,\ref{nonzero}): 

\begin{conjecture}\label{mainconj} 
Let $k$ be an imaginary abelian field, and let $p \nmid [k : \Q]$ be a prime totally 
split in $k$. Let $K/k$ be a $\Z_p$-extension, $K \ne k^\cyc$.
If the set $S_{K/k}$ of $p$-places of $k$ ramified in $K/k$ is totally 
ramified, then $\mu_p(K/k) = 0$ and $\lambda_p(K/k) = \order S_{K/k} - 1$. 
\end{conjecture}

It is possible that the condition $\CH_k = 1$ be necessary. In a complete 
general situation, the conjecture may depend on a suitable definition of 
``minimal Iwasawa's invariants''.

\subsection{Additional bibliographical remarks}

In relation with the deep results of Sands \cite{San1991,San1993}, Ozaki 
\cite{Oza2004}, Fujii \cite{Fu2013}, Itoh--Takakura \cite{IT2014} 
and others, there is no contradiction with Conjecture \ref{mainconj}, in the 
imaginary quadratic case, since these authors deduce, from specific 
assumptions (involving properties of $k^\cyc$ in general), that 
$\lambda_p(K/k)$ is less than some integer. A good overview 
of these properties may be found in Fujii's paper \cite{Fu2013} giving results 
for any $\Z_p$-extension of $k$, and in particular about $\mu_p(K/k)$. The 
case of multiple $\Z_p$-extensions is the object of \cite{MO2022, Oka2025}.

\smallskip
The triviality of $\mu_p(K/k)$ is known for some $K/k$'s, as the $\Z_p$-extensions
$L, \ov L$ of Diagram \ref{ramification} (see, e.g., Gillard \cite{Gil1985}, Schneps
\cite{Sch1987} for $p \geq 5$, then Oukhaba--Vigui\'e \cite[Theorem 1.2, Section 7]
{OV2016} for all $p$ by means of the theory of Katz $p$-adic $L$-functions). 

\smallskip
In Ozaki--Minardi \cite{Oza2001,Min1986}, the main theorem claims, in the split 
case under the assumption $\CH_k = 1$, that ``\,$\lambda_p(K/k) = 1$ and 
$\mu_p(K/k) = 0$ for all but finitely many $\Z_p$-extension $K/k$.'' In Kataoka 
\cite[Theorems 1.3, 5.3, 5.9]{Kat2017} many generalizations are proved by means 
of defining a topology and a measure on the set of $\Z_p^i$-extensions of $k$.
In \cite{Hor1987} was proved that there exist infinitely many imaginary 
quadratic fields $k$ in which $p$ does not split, and such that 
$\lambda_p(k^\cyc/k) = 0$; but the reason relies on the non-trivial 
existence of infinitely many fields $k$ such that $p \nmid \hk$; 
whence the result from Chevalley--Herbrand formula. 

\smallskip
In Kundu--Washington \cite{KW2023, KW2024} and in 
Ellenberg--Jain--Venkatesh \cite{EJV2011},
heuristics, statistics and conjectures are given for the Iwasawa 
invariants of $k^\acyc/k$, also in the non-split case, which is of another 
nature compared to our study of the totally $p$-adic case; but this suggests 
that the filtrations are bounded so that $\lambda_p(k^\acyc/k)$ and 
$\mu_p(k^\acyc/k)$ are minimal for these imaginary quadratic fields, contrary 
to the examples of $\mu_p(K^\iw/k) \geq 1$ needing a non semi-simple 
context and a specific topological choice of the $\Z_p$-extension  
(see \S\,\ref{nonzero} and Diagram \ref{pathological}). 
Moreover, their paper gives some computations of the first layer of 
$k^\acyc/k$ and examples of capitulation of $\CH_k$ in $k^\acyc$.

\smallskip
Some studies are of cohomological nature, finding again some 
generalizations of Gold's criterion by means of Massey Products 
depending also on the valuation $\delta_p(k)$ of the Fermat quotient
of the $S_k$-unit.\footnote{\,I would 
like to thank Christian Maire who 
helped me to understand the relationship with Massey Products, 
which are especially useful in inverse Galois problems for pro-$p$ 
groups, and not specifically in class field theory. Class field theory is 
more of a basic model; indeed, this Galois context is now based on the 
work of Scholz-Reichardt, revisited for instance in \cite[Theorem A]{HLM2025} 
from the perspective of Massey Products when $k$ does not contain the 
group $\mu_p^{}$ of $p$-th roots of unity (see also \cite{MRT2024}).}
See for instance Peikai \cite[\S\S\,5.1.1, 5.1.2]{Pei2024} (following works 
of \cite {Sha2007, LSW2023}) applied to imaginary quadratic fields, where 
an algorithm is related to a filtration, by means of norm computations, 
but without any programming, and under numerous restrictive assumptions 
($K = k^\cyc$, $\CH_k = 1$, then one must know in advance that 
$\lambda_p(k^\cyc/k)<p$, the algorithm is valid at the first layer (which 
is very rare), the known result $\mu_p(k^\cyc/k) = 0$ is needed, etc.). 

\smallskip
It is then easy to see an interconnection with the use of the filtrations, 
for which many simplifications arise (in particular for the algorithm of
\S\,\ref{inclusions}\,(i)), and the role played by $\delta_p(k)$ from
the ${\mathfrak p}$-unit $x$ (denoted $\alpha$ in \cite{Pei2024}), 
illustrating the long history of such principles having started with the 
famous ``R\'edei matrices'' synthesized in \cite{Gra2017a}; moreover, 
it is almost certain that many cohomological computations have a relationship
with $\wt \CH_k$.

\smallskip
Whatever the method, we think that, ultimately, practical results use the 
same arithmetic processes, especially the most elaborate points of classical 
number theory, with the common problem of separating accessible 
algebraic informations (in general of hight technicality), from random ones 
leading to inaccessible conjectures.

\smallskip
Many studies are based on the General Greenberg Conjecture 
\cite[\S\,4, Conjecture\,3.5]{Gree1998}, as for instance \cite{Min1986,
Oza2001, Tak2020, Mur2023}, among others; in \cite[Theorem\,1.4]{LiQi2023}, 
some links, between the General Greenberg Conjecture and $p$-rationality 
(see Footnote\ref{prat}), are given in \cite{OT1995}, and improved in 
\cite[Th\'eor\`emes 6, 10]{Jau2024a}. See \cite{Mai2018, RMM2025} for 
generalizations. In \cite{Sto2024} a weaker interpretation of Gold's criterion 
is given from the theory of elliptic curves.

\section{Notion of smooth complexity of a 
\texorpdfstring{$\Z_p$}{Lg}-extension}\label{complexity}

Let's begin by a kind of counterexample to ``smooth complexity'' which 
will allow to test if capitulation problem is or not of particular nature 
(Programs \ref{P6}, \ref{P7}).

\subsection{Iwaswa's construction of \texorpdfstring{$K^\iw/k$}{Lg} 
with \texorpdfstring{$\mu_p(K^\iw/k) > 0$}{Lg}}\label{nonzero}

Let $k_0$ be an imaginary quadratic field and let $p \geq 3$ be a fixed prime. 
Let $\{\ell_1,\ldots, \ell_t\}$, $t \geq 2$, be a set of distinct primes 
$\ell_i \equiv 1 \pmod p$, totally inert in $k_0^\cyc/\Q$.
Let $F$ be a degree-$p$ cyclic field of conductor $\ell_1\cdots \ell_t$ 
and set $k := k_0F$. Since the $\ell_i$'s are inert in $k_0/\Q$, there 
are $t$ ramified primes in $k/k_0$.
Being unique, since the $\ell_i$'s do not split in $k_0/\Q$
the decomposition group $D_{\ell_i} \simeq \Z_p$ of $\ell_i$ in 
$\wt {k_0}/k_0$ is invariant by $\tau$; so $D_{\ell_i} = 
\Gal(\wt {k_0}/k_0^\acyc)$ is the unique solution, 
and the $\ell_i$'s totally split in $k_0^\acyc/k_0$. 

\smallskip
Consider the $\Z_p$-extension $K^\iw := k k_0^\acyc$ of $k$;
in some sense, use Diagram \eqref{ramification} with 
the base-field $k_0$ and a ``translation'' of $k_0^\acyc/k_0$ 
into $K^\iw/k$ with the field $F$. 
\unitlength=0.7cm 
\begin{equation}\label{pathological}
\begin{aligned}
\vbox{\hbox{\hspace{-2.0cm} 
\begin{picture}(10.0,7.1)
\put(3.9,6.4){\line(2,0){4.3}}
\put(4.1,4.1){\line(2,0){4.2}}
\put(3.8,1.5){\line(2,0){4.4}}
\put(3.25,4.04){$k_{0,n}^\acyc$}
\put(8.35,4.0){\ft$K^\iw_n$}
\put(8.35,1.35){$k \!=\! k_0$\ft$F$}
\put(5.3,7.1){\ft$\gamma\! \simeq \!\Z/p\Z$}
\put(4.6,0.25){\ft$p$}
\put(0.8,1.0){\small$\{\ell_i\}_{i\in[1,t]}$}
\put(1.05,0.6){\tiny${\rm inert}$}
\put(10.4,1.45){\small$\{{\mathfrak l}_i\}_{i\in[1,t]}$}
\put(-0.5,4.05){\small$\{{\mathfrak l}_i^j\}_{i\in[1,t];\,j\in[1,p^n]}$}
\put(9.4,4.05){\small$\{{\mathfrak L}_i^j\}_{i\in[1,t];\,j\in[1,p^n]}$}
\bezier{300}(3.8,6.75)(6.0,7.1)(8.2,6.75)
\put(3.50,1.75){\line(0,1){2.1}}
\put(8.45,1.75){\line(0,1){2.1}}
\put(3.65,2.8){\ft$p^n$}
\put(8.55,2.75){\ft$G_n$}
\put(3.50,4.4){\line(0,1){1.8}}
\put(8.45,4.35){\line(0,1){1.85}}
\put(8.3,6.35){\ft$K^\iw \!=\! kk_0^\acyc$}
\put(3.30,6.35){$k_0^\acyc$}
\put(3.28,1.35){$k_0$}
\put(0.6,2.85){\tiny${\rm total\ splitting}$}
\put(1.2,2.5){\tiny${\rm of\ the} \ \ell_i$'s}
\put(4.7,4.28){\tiny$t p^n$ {\rm ramified} ${\mathfrak L}_i^j$}
\put(4.8,1.7){\tiny$t$ {\rm ramified} $\ell_i$}
\put(3.3,1.25){\line(-1,-1){1.0}}
\put(3.05,0.65){\ft$2$}
\put(8.3,1.3){\line(-1,-1){1.0}}
\put(6.9,0.0){\ft$F$}
\put(1.9,0.0){\ft$\Q$}
\put(10.4,0.0){\ft$\Q(\mu_{\ell_1 \cdots \ell_t}^{})$}
\put(7.3,0.08){\line(1,0){3.0}}
\put(2.4,0.08){\line(1,0){4.4}}
\end{picture}}}
\end{aligned}
\end{equation}
\unitlength=1.0cm

The $\Z$-rank of the unit group $E_{k_{0,n}^\acyc}$ is $p^n-1$ and the 
module $\Omega_{K^\iw_n/k_{0,n}^\acyc}$ is isomorphic to 
$(\Z/p\Z)^{t \cdot p^n - 1}$. The Chevalley--Herbrand formula, in 
$K^\iw_n/k_{0,n}^\acyc$, yields (for $\CH_n := \CH_{K^\iw_n}$):
\begin{equation*}
\order \CH_n \geq \order \CH_n^{\,\gamma} = \order \CH_{k_{0,n}^\acyc} 
\times \frac{p^{t \cdot p^n-1}_{}}{\omega(E_{k_{0,n}^\acyc})} 
\geq \frac{p^{t \cdot p^n-1}_{}}{p^{p^n - 1}} = p^{(t-1) \cdot p^n},
\end{equation*}

\noindent
giving $\mu_p(K^\iw/k) \geq t-1$, only by considering the first step of the 
$\gamma$-filtration. Note that, a fortiori, $\order \CT_n$ and $\order 
\wt \CH_n$ have same order of magnitude as $\order \CH_n$.

\smallskip
It will be interesting to understand how the algorithm of \S\,\ref{inclusions}\,(i)
works in that pathological example. In the semi-simple case ($p \nmid [k : k_0]$), 
the notion of $\varphi$-objects in a cyclic $p$-tower, defined in the 1970's 
and revisited in \cite[Section 4]{Gra2023}, leads to effective results,
especially for cyclotomic $\Z_p$-extensions. The non-semi-simple case
remains open; but considering:
$$\CH_n^* := \big\{ c \in \CH_n,\ \BN_{K^\iw_n/K^\iw_{n-1}}(c) = 1\big \}, $$

\noindent
one obtains, for $n \gg 0$, the relation
$v_p\big( \order \CH_n^* \big) = \lambda_p(K^\iw/k) 
+ \mu_p(K^\iw/k) \cdot p^{n-1}\,(p-1)$, 
with, possibly, $\lambda_p(K^\iw/k) = \lambda_p(k_0^\acyc/k_0)$, 
$\mu_p(k_0^\acyc/k_0) = 0$. Then,
$\order \CH_n = \order \CH_k \times \prd_{i = 1}^n\order \CH_i^*$, 
is true for all $n \geq 0$ as soon as $k_0^\acyc/k_0$ is totally ramified. 
Since $\Bnu_{\!K^\iw_i/K^\iw_{i-1}} = \Phi_{\!p^i}(\sigma_i)$, where 
$\Phi_{\!p^i}$ is the $p^i$-th cyclotomic polynomial, this suggests a more 
suitable algebraic interpretation since the $\CH_i^*$'s, annihilated 
by $\Phi_{\!p^i}(\sigma_i )$, are finite $\Z_p[\mu^{}_{p^i}]$-modules.

\smallskip
Many variants of this process are in the literature from Iwasawa's examples \cite{Iw1973}, 
e.g., Hubbard--Washington \cite{HW2018} where numerical computations are done, 
then Ozaki \cite{Oza2004} where $\lambda$ and $\mu$ can be prescribed; in any 
case, the genus theory form of Chevalley--Herbrand formula is needed in suitable 
finite $p$-sub-extensions of $K^\iw_n$ (here $K^\iw_n/k_{0,n}^\acyc$) in which the 
number of tame ramified primes is for instance of order $O(p^n)$, due to splitting. 

\smallskip
See important generalizations by Perbet \cite{Per2011}, Hajir--Maire \cite{HM2019} 
where it is shown that large $\mu$'s are related to the existence of tame primes 
that split completely in the tower.

\subsection{Attempt of definition of the complexity}

We note that for the previous example giving $\mu_p(K^\iw/k) \geq t-1$, inequalities
\eqref{inequalities} show that the number $b_n$ of steps of the algorithm at the layer 
$K^\iw_n$ is of the order of $p^n$, which probably defines the biggest ``complexity'' 
that may exist, the problem being to define this notion. An heuristic, that we have 
analyzed and numerically tested in \cite[Theorems 1.1, 1.2]{Gra2024}, for totally 
ramified cyclic $p$-extensions $M/F$, suggests that if the complexity of $\CH_M$ 
is at a low level and $[M : F] \gg 0$, then, in a not very intuitive way, often $\CH_F$ 
capitulates in $M$. On the contrary, in case of hight complexity, capitulation is 
much less easy; we give with Program \ref{P6} some numerical results.

\smallskip
We have proven in \cite[Corollary 2.11]{Gra2024} a {\it sufficient condition} of 
capitulation from the property that $\Bnu_{\!M/F}(\CH_M) = \BJ_{M/F} \circ \BN_{M/F}
(\CH_M) = \BJ_{M/F}(\CH_F)$ in the totally ramified case giving $\BN_{M/F}(\CH_M) =
\CH_F$ by class field theory; so, forcing $\Bnu_{\!M/F}(\CH_M) = 1$ implies capitulation:

\begin{theorem}\label{maincoro}
Let $M/F$ be a totally ramified cyclic $p$-extension of degree $p^N$, $N \geq 1$, 
of Galois group $G = \langle \sigma \rangle$. Let $b(M) \geq 0$ be the minimal 
integer such that $(\sigma - 1)^{b(M)}$ annihilates $\CH_M$ and let $p^{\expo(M)}$
be the exponent of $\CH_M$. 
Then a sufficient condition of capitulation of $\CH_F$ in $M$ is that:
$\expo(M) \in[1, N - s(M)]$ if there exists $s(M) \in [0, N-1]$ such that
$b(M) \in [p^{s(M)}, p^{s(M)+1}-1]$.
If $M/F$ is not totally ramified, the sufficient condition of capitulation applies to 
$\CH'_F = \BN_{M/F}(\CH_M)$.
\end{theorem}

We propose to state a general definition, from \cite[Definition 2.12]{Gra2024}:

\begin{definition} \label{defsmooth}
We say that $M/F$ is of {\it smooth complexity} when $\expo(M) \leq N - s(M)$
if there exists $s_N \in [0, N-1]$ such that $b_N \in \big [p^{s_N}, p^{s_N+1}-1 \big]$.
If $M/F$ is a $\Z_p$-extension $K/k$, we say that $K/k$ is of {\it smooth 
complexity} if for all $n \gg 0$, $K_n$ is of smooth complexity (i.e. with simplified 
notations, $\expo_n \leq n - s_n$ if there exists $s_n \in [0, n-1]$ such that
$b_n \in \big [p^{s_n}, p^{s_n+1}-1 \big]$, 
where $b_n$ is the length of the filtration of $\CH_n$).
\end{definition}

It remains to be seen whether the invariant $\wt \CH_k$ is suitable for 
analogous capitulation criteria (see Conjecture \ref{mainconj}) and if 
the smoothness of $K/k$ implies that $\lambda_p(K/k)$ and $\mu_p(K/k)$ 
are subject to certain constraints and are possibly minimal. For this, inequalities 
\eqref{inequalitiesbis} allow some discussion for $\Z_p$-extensions of an 
imaginary quadratic field in the conditions of Theorem \ref{O}. See \cite{Gra2021} 
about the notion of complexity applied to Greenberg's conjecture.

\section{Initial layers of the \texorpdfstring{$\Z_p$}{Lg}-extensions of 
\texorpdfstring{$k = \Q(\sqrt {-m})$}{Lg}}\label{initial}

In \cite{HW2010,HW2018,KW2023,KW2024}, Hubbard--Washington, 
Kundu--Washington perform a study of this problem, especially for 
the layers of $k^\acyc$, by means of various approaches including complex 
multiplication and Kummer theory. 
Explicit results are given for $p = 3$ (often assumed non-split in $k$). 

\smallskip
After the first findings by Carroll--Kisilevsky \cite{CK1976},
similar literature may include \cite{Br2007,VanH2016}. We have given in
\cite[Sections 2, 7]{Gra2026} (to which one can refer for some proofs)
the general method and four programs computing this first layer and the 
main invariants of $k$ and $k_1^\acyc$, without any hypothesis. 
Thus, we only recall, here, the practical results giving the Kummer radicals 
needed for our programs with $p = 3$; they are obtained from that of 
the field $M=k(\mu_3^{})$ as follows.

\subsection{Radicals and Galois groups}
Let $\chi$ be the character defining $k$ and let $\omega$ be that of 
$\Q(\mu_3^{})$; the mirror field $k^*$ is of character $\chi^* = \omega\,\chi^{-1}$.
For $H_{M,1}^\pr$ fixed by $\Gal(H_{M,1}^\pr/M)^3$, we have
$\Rad(\hbox{\ft$H_M^\pr$}) := \big\{w M^{\times 3}, 
\ M(\sqrt[3]w) \subseteq H_{M,1}^\pr \big \} = \big\{w M^{\times 3},\ (w) \in
I_M^3 \langle\, \, S_M \,\rangle\big \}$. 

\smallskip
If $K/k$ is a $\Z_3$-extension and if $K_1$ is the first layer in $K/k$, then 
it is defined by means of a degree-$3$ irreducible polynomial $Q \in k[x]$.
But the cyclic $3$-extension $MK_1/M$ is a $3$-ramified Kummer extension 
of the form $MK_1 = M(\sqrt[p] w)$, where the radical $w \in M^\times$ may 
be known from $3$-classes and $S_M$-units of $M$; then one deduces easily 
$Q$ from $w$.

\smallskip
The Galois descent of $M H_{M,1}^\pr/M$ to $H_{k,1}^\pr/k$, gives rise 
to the reflection principle. Let $\gamma = \Gal(M/\Q) \simeq (\Z/2\Z)^2$ 
and $\Gamma := \Gal(H_{M,1}^\pr/M)$ 
as $\gamma$-module; then $\Gal(H_{k,1}^\pr/k) \simeq \Gamma_{\chi}$, 
so that the radical of $M \,H_{k,1}^\pr/M$ is given by the component
$(\Rad(\hbox{\ft$H_{M,1}^\pr$}))_{\chi^*}$ which lies in 
$k^{*\times}/(k^{*\times})^3$. 

Thus, the sub-radical giving 
$M \,k_1^\acyc/M$ is the most tricky part of the algorithm of \S\,\ref{algo}.

\begin{proposition}\label{conjugates}
Let $w \in k^*$ (of conjugate $w'$) be the solution giving the radical, of 
character $\chi^*$, associated to the first layer $k_1^\acyc$ of the anti-cyclotomic 
$\Z_3$-extension of $k$, via the identity $M\,k_1^\acyc = M({\sqrt[3]w})$. 

(i) The first layer $k_1^\acyc$ is given by the irreducible 
polynomial of $\sqrt[3]w+\sqrt[3]{w'}$, and is
$Q^\acyc = x^3 - 3a x - t$, where $a^3 = w w'$ and $t = w + w'$, $\in \Q$. 

(ii) The first layers of two conjugate non-Galois $\Z_3$-extensions of $k$, 
are given by
$Q := x^3-3a x-t$, $Q' := x^3-3a x- t'$, 
$a^3 = w w'$, $t := {\rm Tr}_{k(j)/k}(j w)$, $t' := {\rm Tr}_{k(j)/k}(j^2 w')$.
The traces $t, t'$ are given as follows:

\smallskip
Case $3 \nmid m$. Set $w = \frac{1}{2}(u+v \sqrt{3m})$; then the traces $t, t'$ 
are $\frac{-1}{2}(u \pm 3v \sqrt{-m})$.

Case $3 \mid m$. Set $w = \frac{1}{2}(u+v \sqrt{m/3})$; then the traces $t, t'$ 
are $\frac{-1}{2}(u \pm v \sqrt{-m})$.
\end{proposition}

\begin{proof}
In cases (i) (resp. (ii)), it suffices to compute the irreducible polynomial of 
$\sqrt[3]{w}+\sqrt[3]{w'}$ over $\Q$ (resp. $\sqrt[3]{jw}+\sqrt[3]{j^2 w'}$ 
over $k$ since $t, t' \in k$).
\end{proof}

\subsection{Algorithm of computation of \texorpdfstring{$\CT_k$}{Lg}}
For these questions of ramification theory, the computation of $\CT_k$ is crucial. 
The following program comes from the general one \cite{Gra2019c} and 
computes the four invariants $\CT_k$, $\CH_k$, $\CT^\bp_k$, $\CW^\bp_k$, 
$k = \Q(\sqrt{-m})$, for {\it any} given prime $p$ and positive square-free integer 
$m \ne 3$ when $p = 3$.

\smallskip
The principle is to compute a ray class group ${\sf Kpnu = bnrinit(k,p^{nu})}$
of suitable modulus ${\sf p^{nu}}$ so that the components of indices $j \in [3, d]$
($d$ being the number of components given by {\sc pari/gp})
give the abelian structure of $\CT_k$ (the components of indices 
$i = 1, 2$ coming from the two independent $\Z_p$-extensions; these
compon ents tend to infinity as $\nu \to \infty$):

\ft\begin{verbatim}
{p=3;m=6789;k=bnfinit(x^2+m);Tk=List;nu=valuation(k.no,p)+2;
Kpnu=bnrinit(k,p^nu);Hknu=Kpnu.cyc;ot=1;d=matsize(Hknu)[2];
for(j=3,d,c=Hknu[j];v=valuation(c,p);if(v>0,listput(Tk,p^v);ot=ot*p^v));
ow=1;if(p==3 & Mod(m,9)==3,ow=3);otbp=ot/ow;print("p=",p);
print("m=",m," Tk=",Tk," Hk=",k.cyc," #Tk^bp=",otbp," #Wk^bp=",ow)} 
m=6789  Tk=[9]  Hk=[6,6,2]  #Tk^bp=3  #Wk^bp=3
\end{verbatim}\ns

One proves in \cite[Theorem~3.4]{Gra2026} from \cite[\S\,I.5.8, p.\,16]{FV2002}, 
that the best modulus in the imaginary quadratic case is $p^{\expo_k+2}$, where 
$p^{\expo_k}$ is the exponent of $\CH_k$; to simplify, the program uses 
$p^{v_p(\hk) + 2}$. This does not hold in the real case.

\subsection{Algorithm of computation of \texorpdfstring{$k_1^\acyc$ --}{Lg}
Program for \texorpdfstring{$p = 3$}{Lg}}\label{algo}\label{nonkummercase}

The method is based on the properties of the $\Log_p$-function leading to 
$[H_k^\nr \cap \wt k : k] = [\log_p(I_k \otimes \Z_p) : \log_p(P_k \otimes \Z_p)]$
(for $k = \Q(\sqrt{-m})$, $\Log_p = \log_p := (\log_{\mathfrak p}, \log_{\ov {\mathfrak p}}$), 
where $\log_p({\mathfrak a}) = \frac{1}{h}\log_p(\alpha)$ for ${\mathfrak a}^h = 
(\alpha)$,$\alpha \in k^\times$).

\smallskip
All the details about $\Log_p$-functions and Artin groups in the 
extension $H_k^\pr/k$ are recalled in \cite[Section\,2]{Gra2026}.
Whence the diagram, where $\CA^-_k = \Gal(H_k^\pr/k)^-$,
whose elements are the minus components 
of Artin symbols $\big (\ffrac {H_k^\pr/k}{{\mathfrak a}}\big)$, 
${\mathfrak a} \in I_k \otimes \Z_p$:
\unitlength=0.75cm
\begin{equation}\label{DLog}
\begin{aligned}
\vbox{\hbox{\hspace{-1.5cm}  
\begin{picture}(11.5,5.5)
\put(8.3,4.50){\line(1,0){2.2}}
\put(4.0,4.50){\line(1,0){2.3}}
\put(4.3,2.50){\line(1,0){2.3}}
\bezier{350}(4.0,4.9)(7.3,5.5)(10.6,4.9)
\put(7.2,5.35){\ft$\CT_k$}
\put(3.50,2.85){\line(0,1){1.4}}
\put(3.50,0.8){\line(0,1){1.3}}
\put(7.50,2.85){\line(0,1){1.4}}
\bezier{280}(3.2,0.6)(2.1,2.5)(3.2,4.4)
\put(-0.2,2.4){\ft$\Log_p(I_k \!\otimes\! \Z_p)^-$}
\put(4.1,3.4){\ft$p \,\Log_p(I_k \!\otimes\! \Z_p)^-$}
\bezier{280}(3.8,2.7)(4.2,3.4)(3.8,4.4) 
\bezier{200}(3.9,0.5)(5.7,0.5)(7.5,2.0)
\put(6.25,0.9){\ft$\CA^-_k/(\CA^-_k)^p$}
\bezier{200}(7.8,4.2)(9.3,3.7)(10.8,4.2)
\put(9.0,3.55){\ft$\CT_k^p$}
\bezier{600}(11.3,4.0)(11.4,0.0)(3.9,0.35)
\put(9.8,1.05){\ft$\CA^-_k$}
\bezier{300}(7.8,2.5)(9.8,2.6)(11.0,4.2)
\put(9.7,2.7){\ft$(\CA^-_k)^p$}
\put(10.6,4.4){$H_k^\pr$}
\put(6.5,4.4){${k^\acyc} H_{k,1}^\pr$}
\put(3.4,4.4){${k^\acyc}$}
\put(6.7,2.4){$H_{k,1}^\pr$}
\put(3.2,2.3){$k_1^\acyc$}
\put(3.4,0.3){${k}$}
\put(3.64,1.44){\ft$p$}
\end{picture}}} 
\end{aligned}
\end{equation}
\unitlength=1.0cm

In an algorithmic point of view, we will take prime ideals ${\mathfrak q}$, 
above a prime number $q$, so that the above condition
$\log_3({\mathfrak q})^- \equiv 0 \pmod {3^{\expo_k+2}}$ be satisfied. This 
implies {\it the splitting of ${\mathfrak q}$ in $k_1^\acyc$}.
More precisely, for a choice of $w$ in the list of Kummer radicals 
(${\sf Lw}$ in the program),
this defines a cubic field $K_1 \subset H_k^\pr$, and one
hopes that some ${\mathfrak q}$'s are inert in $K_1$, thus 
eliminating the corresponding radical ${\sf w}$ (inertia being equivalent 
to the irreducibility in $\Q_q[x]$ of the polynomial defining $K_1$). 

\smallskip
Then (if any) the {\it unique} solution ${\sf w^\acyc}$ defines
${\sf k_1^\acyc}$, because all ideals ${\mathfrak q}$, such that 
$\Log_3({\mathfrak q})^- \in 3 \,\Log_3(I_k \otimes \Z_3)^-$, 
are split in this selected cubic field, but not all split in the other 
cubic fields.
Therefore, ${\sf w^\acyc}$ passed the decomposition test for 
all ${\mathfrak q}$, while each of the others ${\sf w}$'s failed for at
least an ideal ${\mathfrak q}$; for 
instance, for $k = \Q(\sqrt{-107})$, the first occurrence is for $q = 79$, 
giving ${\sf Qq = x^3+(76+78*79)*x+(35+5*79)+O(79^2)}$ in $\Q_{79}$, 
eliminating ${\sf w = Mod(12*x-215,x^2-321)} \in k^*$. 

\smallskip
On the contrary, if 
this never holds, for $q$ varying in large interval, this gives a good 
candidate, {\it then the solution giving the first layer $k_1^\acyc$
as soon as it is unique in the list ${\sf L_w}$}.

\smallskip
See Program \ref{P7} and an excerpt of the results. The number $\Val$  
gives the test implying $\log_3({\mathfrak q})^- \in 3 \log_3(I_k \otimes \Z_3)^-$
in Diagram \ref{DLog}. The program also computes $\CH_{k_1^\acyc}$.

\subsection{Capitulations in non-cyclotomic
\texorpdfstring{$\Z_p$-extensions}{Lg}}\label{verifcap}

We obtain general circumstances of (partial) capitulation of $\CH_k$ in the 
first layer of $k^\acyc$ and of the non-Galois $K$'s; it is clear that 
$\CH_k$ may capitulate unconditionally in some larger layer, what 
is difficult to verify numerically. Concerning the subjects of 
capitulation phenomena in unramified extensions and Galois structure of
class groups, we refer to the interesting Bembom Thesis 
\cite{Bem2012} which give many technical properties, and \cite{Th2000}
for capitulations, in the genus fields, of invariants ideals.
We refer to Section \ref{why} for properties of capitulation of the 
logarithmic class group in connection with Greenberg's conjectures.

\subsubsection{Case of the anti-cyclotomic $\Z_p$-extension
with $p$ split in $k$}

For $k^\acyc/k$ we assume the total ramification because the capitulation of 
$\CH_k$ in the Hilbert's class field may occur in $k^\acyc \cap H_k^\nr$, as 
shown by some examples with Program \ref{P8} and $m \in$\,$\{$$362$, $367$,
$\ldots$$\}$ for $p = 3$ (besides, this is going in the right direction, but a total 
ramification is more convincing):

\begin{theorem}\label{capitule}
Assume that $p \geq 3$ splits in $k = \Q(\sqrt{-m})$, that $k^\acyc/k$ is totally ramified, that 
$\CH_k \simeq \Z/p\Z$ and $\CH_{k_1^\acyc}^{} \simeq \Z/p\Z \times \Z/p\Z$.
Then $\CH_k$ capitulates in $k_1^\acyc$. For $p = 3$, this holds for 
$m \in$$\{$$302$, $602$, $617$, $713$, $863$, $1007$, $1046$, 
$1427$, $1454$, $\ldots$$\}$.
\end{theorem}

\begin{proof}
The Chevalley--Herbrand formula, applied in $k_1^\acyc/k$ with $G_1 = 
\Gal(k_1^\acyc/k)$, gives by assumption $\order \CH_{k_1^\acyc}^{G_1} = p^2$,
whence $\CH_{k_1^\acyc} = \CH_{k_1^\acyc}^{G_1}$. Let $\Bnu_{k_1^\acyc/k}$ 
be the algebraic norm equal to $\BJ_{k_1^\acyc/k} \circ \BN_{k_1^\acyc/k}$; so, 
$\Bnu_{k_1^\acyc/k}(\CH_{k_1^\acyc}) = \CH_{k_1^\acyc}^p = 1$ proving the 
result since $\BN_{k_1^\acyc/k}$ is surjective. This is an
illustration of a ``smooth complexity'' (see Theorem \ref{maincoro} with
$M/F = k_1^\acyc/k$, $\expo(M) = 1$, $s(M) = 0$, $b(M) \in [1, p-1]$ 
since $b(M) = 1$ from $\CH_{k_1^\acyc} = \CH_{k_1^\acyc}^{G_1}$).
\end{proof}

\subsubsection{Case of the anti-cyclotomic $\Z_p$-extension 
for $p$ non-split in $k$}

We assume that $k^\acyc/k$ is totally ramified and that $\CH_k \simeq \Z/p\Z$.
We have, in $k_1^\acyc/k$, the relations:
\begin{equation*}
\order \CH_{k_1^\acyc}^{G_1} = \order \CH_k = p, \ \ 
\CH_{k_1^\acyc}^{G_1} = \BJ_{k_1^\acyc/k}(\CH_k) \oplus \CH_{k_1^\acyc}^\ram ,
\ \ \CH_{k_1^\acyc}^\ram =: \langle \Ccl_1 ({\mathfrak P}_1) \rangle, \ 
{\mathfrak P}_1 \mid p,
\end{equation*}

\noindent
the direct sum being due to the fact that in the non-split case, $\CH_k^+ = 1$
and $\CH_{k_1^\acyc}^\ram{}^- = 1$.
We deduce that $\CH_k$ capitulates if and only if $\Ccl_1 ({\mathfrak P}_1)$, 
is of order $p$. With a  {\sc pari/gp} program we have obtained 
many capitulations for $p = 3$, where, as predicted by the Kundu--Washington result 
(Theorem \ref{noncyclic}), $\CH_{k_1^\acyc} \simeq \Z/3\Z \times \Z/3\Z$ or
$\Z/9\Z \times \Z/9\Z$. This holds for $m \in$\,$\{$$298$, $397$, $622$, $643$,
$685$, $706$, $835$, $901$,\,$\ldots$$\}$ if $\CH_{k_1^\acyc} 
\simeq \Z/3\Z \times \Z/3\Z$ and for $m \in$\,$\{$$3073$, $3469$, 
$4405$,\,$\ldots$$\}$ if $\CH_{k_1^\acyc} \simeq \Z/9\Z \times \Z/9\Z$.
More precisely, we have:

\begin{theorem}\label{capitule2}
Assume that $3$ does not split in $k = \Q(\sqrt{-m})$, that $k^\acyc/k$ is totally 
ramified, that $\CH_k \simeq \Z/3\Z$ and that $\CH_{k_1^\acyc}^{} \simeq 
(\Z/3^t\Z)^2$, $t \geq 1$. Then $\CH_k$ capitulates in $k_1^\acyc$. 
For all the other structures of $\CH_{k_1^\acyc}^{}$, $\BJ_1$ is injective.
\end{theorem}

\begin{proof}
Let's give a proof for $t = 1$ (case of ``smooth complexity'').
Using for short the notations ``$\CH_n^i, \BN_n, \Bnu_n$'' for the filtrations 
in $k^\acyc$, and the norms, we have $\order \CH_1^1 = \order \CH_k = 3$ 
and $\order(\CH_1^2/\CH_1^1) = 3$ since $\BN_1(\CH_1^1) = 
(\CH_k)^3 \cdot \langle \Ccl({\mathfrak p}) \rangle = 1$; whence $\order \CH_1^2 = 3 
\cdot \order \CH_1^1 = \order \CH_1 = 9$ by assumption, and $\CH_1 = 
\CH_1^2 \simeq (\Z/3\Z)^2$; but $\Bnu_1 = 1+\sigma_1+\sigma_1^2 =
3 - 3 (1- \sigma_1) + (1- \sigma_1)^2$ yields $\Bnu_1(\CH_1) = \Bnu_1(\CH_1^2) = 1$
(this case also illustrates Theorem \ref{maincoro}, with 
$\expo(M) = 1$, $s(M) = 0$, $b(M) \in [1, p-1] = [1, 2]$ since $b(M) = 2$).
The proof for $t \geq 2$ uses \cite[Proposition 9.8, Theorem C.2]{Gra2026}.
\end{proof}

\subsubsection{Case of ramified non-Galois $\Z_3$-extensions $K, \ov K$}

Assume that $\wt k \cap H_k^\nr = k$ and that $3$ splits in $k$. 
Then $K_1$ and $\ov K_1$ are defined by means of conjugate cubic 
polynomials $Q, Q'$ in $k[x]$ and $R = Q Q'$ is irreductible over $\Q$.
In this context, $K_1$ (resp. $\ov K_1$) coincides with the first layer of the 
inertia field $L$ of ${\mathfrak p}$ (resp. $\ov L$ of ${\ov {\mathfrak p}}$) 
in $\wt k/k$; so $K_1/k$ is ramified at ${\ov {\mathfrak p}}$ and not at 
${\mathfrak p}$, and $\ov K_1/k$ is ramified at ${\mathfrak p}$ and not 
at ${\ov {\mathfrak p}}$ (see Diagram \ref{maindiagram}). 

\begin{theorem}\label{capituleK}
Assume that $3$ splits in $k = \Q(\sqrt{-m})$, that $\CH_k \simeq \Z/3\Z$ and that 
$\wt k \cap H_k^\nr = k$. Let $Z$ (equal to a $k$-conjugate 
of $K_1$ or of $\ov K_1$) be a field defined over $\Q$ by $R$.
We assume that $\CH_Z \simeq \Z/3\Z$; then $\CH_k$ capitulates 
in $K_1$ and $\ov K_1$. This holds for $m \in$\,$\{$$107$, $302$, $503$, 
$509$, $533$, $602$, $617$, $713$, $863$, $1007$,\,$\ldots$\,$\}$.
\end{theorem}

\begin{proof}
We are in the case $3 \nmid m$ of Proposition \ref{conjugates}, with
$w = \frac{1}{2}(u+v \sqrt{3m})$ and $t, t' = \frac{-1}{2}(u \pm 3v \sqrt{-m})$.
We have $R = (x^3-3ax)^2+(x^3-3ax) (w+w') + (w+w')^2 - 3 w w' \in \Q[x]$.
For a root $\alpha$ of $Q$, $\Q(\alpha)$ gives, for instance, the first layer $K_1$. 
Since $3$ splits in $k$ and ramifies in $K_1$ at a single $3$-place, 
the Chevalley--Herbrand formula applied in $K_1/k$ leads to $\CH_{K_1} = 
\CH_{K_1}^{G_1}$, giving analogous result as for the anti-cyclotomic extension.
\end{proof}

These properties are related to the notion of ``smooth complexity'' of the 
$\Z_p$-extension implying capitulations (Section \ref{complexity}), and 
this comforts the philosophical comments given in \S\,\ref{obstructions} 
for capitulations in all the $\Z_p$-extensions $K/k$, except the cyclotomic 
one, thus a similarity with real Greenberg's conjecture suggesting 
Conjecture \ref{mainconj}.

\section{Generalizations of Ozaki's example \texorpdfstring{$k = \Q(\sqrt{-239}),
\lambda_3(k^\acyc/k) \geq 3$}{Lg}}\label{D}  

The numerical example given by Ozaki in \cite[Section 3, p. 390]{Oza2001}, using 
Minardi's thesis \cite{Min1986} and some improvements, claims that $\lambda_3(k^\acyc/k) 
\geq 3$ for the anti-cyclotomic $\Z_3$-extension of $k = \Q(\sqrt{-239})$, for which
$\CH_k \simeq \Z/3 \Z$, knowing that the $3$-Hilbert class field $H_k^\nr$ of $k$ 
is contained in $k^\acyc$. When $H_k^\nr \subset k^\acyc$, it 
follows that $\CH_k \simeq \Z/p^e \Z$, $e \geq 0$, and that $\CT_k = 1$ (see 
Diagram \eqref{ramification}), hence, that $k$ is $3$-rational.

\smallskip
This suggests general statements, on Iwasawa's invariants, under analogous 
conditions, and gives rise to easy proofs using fixed points formulas; numerical 
examples with a program are given in \S\,\ref{ozakiP}. We give first the case of 
$k^\acyc$, then the non-Galois case:

\subsection{Case of the anti-cyclotomic 
\texorpdfstring{$\Z_p$-extension of $k$}{Lg}}
For a $\Z_p$-extension $K/k$, let $K_n$ be its layer of degree $p^n$ 
over $k$ and let's use the simplified notations at each level $n$
($\CH_n := \CH_{K_n}$, and so on).

\begin{theorem}\label{H}
Let $k$ be an imaginary quadratic field, and let $p \geq 3$ be a 
prime number split in $k$ into $(p) = {\mathfrak p} {\ov {\mathfrak p}}$.
Let $K := k^\acyc$ be the anti-cyclotomic $\Z_p$-extension of $k$. 
We assume the following conditions, where $x$ is the fundamental 
$\{{\mathfrak p}\}$-unit of $k$, given by 
${\mathfrak p}^{\hp} = (x)$:

\smallskip
(i) The $p$-Hilbert class field $H_k^\nr$ of $k$ is contained in $K$
(whence $\CH_k \simeq \Z/p^e \Z$, $e \geq 0$).

\smallskip
(ii) The prime ideal ${\mathfrak p}$ is $p$-principal (hence 
$p \nmid \hp$).

\smallskip
(iii) The ${\mathfrak p}$-valuation, of the ${\mathfrak p}$-Fermat quotient of 
$\,\ov x$ is zero, that is to say, $\delta_p(k) := v_{\mathfrak p}(\ov x^{\,p-1} - 1)-1 = 0$. 
Then:

\smallskip
\quad ({\rm a}) The Iwasawa invariants $\lambda_p^S(k^\acyc/k)$, 
$\mu_p^S(k^\acyc/k)$, $\nu_p^S(k^\acyc/k)$, defined by the $S_n$-class 
groups of the $K_n$'s, are zero; then $\CH_n = \Ccl_n(\langle S_n\rangle)$, 
for all $n \geq e$, and $\CH_n$ is of exponent divisor of $p^{n - e}$.

\smallskip
\quad ({\rm b}) We have $[p^{n - e}]^{p^e} \leq \order \CH_n \leq 
[p^{n-e}]^{2p^e-1}$, for all $n \geq e$, whence $\mu_p(k^\acyc/k) = 0$, 
and $p^e \leq \lambda_p(k^\acyc/k) \leq 2p^e-1$.
\end{theorem}

\begin{proof}
We will prove these results using, in a suitable order, various Chevalley--Herbrand 
formulas. Since ${\mathfrak p}$ is $p$-principal, ${\mathfrak p}$ and ${\ov {\mathfrak p}}$ 
completely split in $K_e = H_k^\nr$, and $\CH_e = 1$:

\medskip
(i) {\bf Invariant classes in $K_n/K_e$}.
Let $g_n := \Gal(K_n/K_e)$, for all $n \geq e$. Chevalley--Herbrand 
formula in $K_n/K_e$ (totally ramified at the $p$-places of $K_e$) gives:
\begin{equation}\label{chs}
\order \CH_n^{g_n} = \frac{\order \CH_e \times[ p^{n-e}]^{2p^e}} 
{[K_n : K_e] (E_e : E_e \cap \BN_{n/e} (K_n^\times))} = 
\frac{[ p^{n-e}]^{2p^e - 1}} {(E_e : E_e \cap \BN_{n/e} (K_n^\times))}. 
\end{equation}

Since the $\Z$-rank of $E_e$ is $p^e - 1$, we have 
$(E_e : E_e \cap \BN_{n/e} (K_n^\times)) \leq [p^{n-e}]^{p^e-1}$, 
whence $\order \CH_n \geq \order \CH_n^{g_n} \geq [p^{n-e}]^{p^e}$, proving the 
minorations of the class groups in ({\rm b}).

\smallskip
(ii) {\bf Invariant classes in $K_n/k$}.
In $K_n/k$, $\order \CH_n^{G_n} = \ffrac{\order \CH_k \times[ p^{n-e}]^{2}}
{[K_n : k] \times 1} = \ffrac{p^e \times[ p^{n-e}]^{2}}{p^n} = p^{n-e}$, for
all $n \geq e$.

\smallskip
Let ${\mathfrak P}_n \mid {\mathfrak p}$  be a fixed prime ideal of $K_n$
and $\ov {\mathfrak P}_n := {\mathfrak P}_n^\tau \mid {\ov {\mathfrak p}}$
(which makes sense since $K = k^\acyc$ is Galois).
Let $\sigma$ be a topological generator of $\Gal(K/k)$.
Note that ${\mathfrak P}_n^{\sigma^{p^e}} = {\mathfrak P}_n$,
$\ov {\mathfrak P}_n^{\sigma^{p^e}} = \ov {\mathfrak P}_n$, and 
$\tau({\mathfrak P}_n^{\sigma^i}) = 
(\tau {\mathfrak P}_n)^{\sigma^{-i}} = (\ov {\mathfrak P}_n)^{\sigma^{-i}}$. 
Then, in $K_n$, $n \geq e$, we have the prime decompositions
$\BJ_n({\mathfrak p}) = (\CP_{\!n})^{p^{n-e}}$,
$\BJ_n({\ov {\mathfrak p}}) = (\ov \CP_{\!n})^{p^{n-e}}$, where, 
due to the commutation of $\tau$ and $1+\sigma + \cdots + \sigma^{p^e - 1}$,
$\CP_{\!n} = {\mathfrak P}_n^{\,1+\sigma + \cdots + \sigma^{p^e - 1}}$,
$\ov \CP_{\!n} = \ov {\mathfrak P}_n^{\,1+\sigma + \cdots + \sigma^{p^e - 1}}
= \CP_{\!n}^{\,\tau}$.

\smallskip
So, $\BJ_n({\mathfrak p})$ being $p$-principal,
the common order of $\Ccl_n (\CP_n)$ and $\Ccl_n 
(\CP_n^{\,\tau})$ divides $p^{n-e}$.

\smallskip
Since $\CH_k$ capitulates in $H_k^\nr = K_e$, $\CH_n^{G_n} = \CH_n^\ram = \Ccl_n
(\langle \CP_n , \CP_n^{\,\tau} \rangle)$, for all $n \geq e$ (from \eqref{ambigeES}).
It is clear that $\tau$ operates on $\CH_n^{G_n}$, so that we can define
$\CH_n^{G_n}{}^+$ and $\CH_n^{G_n}{}^-$; these components are of orders
$p^{\rho_n^+}$ and $p^{\rho_n^-}$, respectively, such that 
$\rho_n^+ + \rho_n^- = n-e$.

\smallskip
The second element of the filtration is of order $p^n$ for $n \geq e$; 
indeed, by assumption (iii), $\order \wt \CH_k = \order \CH_k = p^e$
(see Theorem \ref{stability}).

\smallskip 
(iii) {\bf Invariant $S_n$-classes in $K_n/k$}.
A generalization of the Chevalley--Herbrand formula in $K_n/k$ is that for 
the $S_n$-class groups $\CH_n^{S_n} = \CH_n/\Ccl_n(\langle S_n \rangle)$
(see \cite[Th\'eor\`eme III.1.9, p. 177]{Jau1986}) or \cite[Corollary 3.9]{Gra2017a}:
\begin{equation}\label{CHS}
\order (\CH_n^{S_n})^{G_n} = \frac{\order \CH_k^{S_k} 
\prod_{v \in S_k} d_v} {[K_n : k] (E_k^{S_k} :
E_k^{S_k} \cap \BN_n (K_n^\times))}, 
\end{equation}

\noindent
where $d_v$ is the local degree of $K_n/k$ at the place $v$. 
Since $\Ccl_k(\langle S_k \rangle) = 1$
by assumption (ii), $p$ totally splits in $K_e/k$, $\CH_k^{S_k} = \CH_k$ 
and $\order (\CH_n^{S_n})^{G_n} = \ffrac{p^e \times [p^{n-e}]^2} 
{p^n \, \order \omega_n(\langle x, \ov x \rangle)} = 
\ffrac{p^{n-e}}{\order \omega_n(\langle x, \ov x \rangle)}$.

Since, from Theorem \ref{mainorder} and Relation \ref{order}, 
$\omega_n( x )$ and $\omega_n( \ov x )$ 
have same order $p^{n - e - \delta_p(k)} = p^{n-e}$ under (iii), 
$\order \omega_n(\langle x, \ov x \rangle) = p^{n-e}$, for all $n \geq e$; then 
$\order (\CH_n^{S_n})^{G_n} = 1$, hence:
\begin{equation}\label{hn}
\CH_n = \Ccl_n(\langle S_n\rangle) = \Ccl_n(\langle {\mathfrak P}_n , 
{\mathfrak P}^{\sigma}_n, \ldots, {\mathfrak P}_n^{\sigma^{p^e-1}},\ \ov {\mathfrak P}_n, 
(\ov {\mathfrak P}_n)^{\sigma} , \ldots ,(\ov {\mathfrak P}_n)^{\sigma^{p^e-1}} \rangle). 
\end{equation}

The ideal ${\mathfrak P}_e$ is $p$-principal in $K_e = H_k^\nr$ 
and gives ${\mathfrak P}_n^{p^{n-e}}$ by extension in $K_n$; so, 
$\Ccl_n ({\mathfrak P}_n)$ and its $2p^e$ conjugates have 
same order divisor of $p^{n-e}$. Which concludes ({\rm a}). 

\medskip
(iv) {\bf Majoration of $\order \CH_n$ and $\lambda_p(k^\acyc/k)$}.
Since the $2p^e$ conjugates of ${\mathfrak P}_n$ are invariant by 
$g_n = \Gal(K_n/K_e)$, and since $\BJ_{n/e}(\CH_e) = \BJ_{n/e}(1) = 1$, 
we have the exact sequence (cf. \ref{ambigeES}):
\begin{equation*}
\begin{aligned}
1 \too  \Ccl_n(\langle {\mathfrak P}_n , 
{\mathfrak P}^{\sigma}_n, \ldots, {\mathfrak P}_n^{\sigma^{p^e-1}},  \ \ov {\mathfrak P}_n, 
(\ov {\mathfrak P}_n)^{\sigma} , & \ldots ,(\ov {\mathfrak P}_n)^{\sigma^{p^e-1}} \rangle) 
\too \CH_n^{g_n} \\
& \too E_e \cap \BN_{n/e}(K_n^\times)/\BN_{n/e}(E_n) \too 1.
\end{aligned}
\end{equation*}

From Formula \eqref{hn}, $\CH_n = \CH_n^{g_n}$ and $E_e \cap 
\BN_{n/e}(K_n^\times) = \BN_{n/e}(E_n)$. From Formula \eqref{chs}: 
$$\order \CH_n = \frac {[ p^{n-e}]^{2p^e - 1}} 
{\order \omega_{n/e}(E_e) } \leq [ p^{n-e}]^{2p^e - 1} =
p_{}^{(2p^e - 1)\, n - (2p^e - 1)\,e}. $$

This implies $\mu_p(k^\acyc/k) = 0$ and completely proves ({\rm b}) 
and the theorem.
\end{proof}

\begin{corollary}
Under the assumptions of Theorem \ref{H}, $\Omega_{n/e} \simeq 
\Z/[p^{n-e}]^{2 p^e - 1}\Z$ and
$\order \omega_{n/e}(E_e^{S_e}) = [p^{n-e}]^{2 p^e - 1}$, 
for all $n \geq e$, whence
$\order \big[ E_e^{S_e} \cap \BN_{n/e}(K_n^\times) / 
(E_e^{S_e})^{p^{n-e}}\big] = [p^{n-e}]^{p^e}$.
\end{corollary}

\begin{proof}
The Chevalley--Herbrand formula, for $S_n$-class groups in $K_n/K_e$, gives,
knowing that $\CH_n^{S_n} = 1$ for all $n \geq e$,
$\order ( \CH_n^{S_n})^{g_n} = \ffrac{\order \CH_e^{S_e} 
\times [p^{n-e}]^{2 p^e}}{[K_n : K_e] \times \big(E_e^{S_e} : E_e^{S_e}
 \cap \BN_{n/e}(K_n^\times) \big)} 
= \ffrac{ [p^{n-e}]^{2 p^e - 1}}{\omega_{n/e}(E_e^{S_e})} = 1$. 
\end{proof}

Theorem \ref{H} may be improved with the following reciprocal aspect:

\begin{proposition}\label{reciprocal} 
Let $K = k^\acyc$, in the split case. Put $K \cap H_k^\nr =: K_e$, $e \geq 0$
is also $\wt e$.
Assume that there exists at least one integer $n \geq e$ such that $\CH_n^{S_n} = 1$. 
This implies the following:

\smallskip
(i) $\CH_n^{S_n}=1$ for all $n \geq e$.

\smallskip
(ii) $p$ totally splits in $K_e/k$;

\smallskip
(iii) $\delta_p(k) = 0$ and $\wt \CH_k \simeq \Z/p^e \Z$.
\end{proposition}

\begin{proof}
Let $H^\spl_k$ be the decomposition field of $p$ in $H_k^\nr/k$ (subfield of 
$H_k^\nr$ fixed by the group generated by the decomposition groups
of the two $p$-places of $k$) and let $K_{e'} := K \cap H^\spl_k \subseteq K_e$; 
so $[H_k^\nr : H^\spl_k] = \order \Ccl(\langle {S_k} \rangle)$ and the 
$p$-places are inert in $K_e/K_{e'} ]$ and totally ramified in 
$K/K_e$. Formula  \eqref{CHS} becomes, for $n \geq e$:
$$\order (\CH_n^{S_n})^{G_n} = \frac{\order \CH_k \cdot \big [p^{n-e'} \big ]^2}
{\order \Ccl(\langle {S_k} \rangle) \cdot p^n \cdot
p^{n - e-\delta_p(k)}} = [H^\spl_k : K_{e'} ] \times p^{e - e' + \delta_p(k)}.$$

\noindent
The condition $\CH_n^{S_n} = 1$ is independent of $n \geq e$, and is 
satisfied if and only if $e' = e$, $\delta_p(k) = 0$, $H^\spl_k = K_{e'} = K_e$;
then $\order \wt \CH_k = \order \CH_k^{S_k} = p^e$.
Under the supplementary condition $\Ccl({\mathfrak p}) = 1$, we get
$H_k^\nr = K_e$.
\end{proof}

\begin{remark}\label{logsplit}
Let $k$ be an imaginary quadratic field and $p \geq 3$ split in $k$.
We have defined $H_k^\lac$ in Remark \ref{lognonsplit}, the maximal 
$p$-split extension of $k^\acyc$ in $H_k^\pr$. We know (see the Diagram 
\ref{Tpsplit2}) that $p$ totally splits in $\Gal(H_k^\pr/\wt k)$; thus,
$H_k^\lac$ is the decomposition field of the $p$-places in the unramified 
extension $H_k^\pr/k^\acyc$. Then $H_k^\lac/K_e$ is of residue degree $1$
and $\wt k \cap H_k^\lac/k^\acyc$ is totally split. Knowing the decomposition 
of the $p$-places in $\wt k/k$, one deduce that $\wt k \cap H_k^\lac = k^\acyc$,
whence $H_k^\lac = k^\acyc H^\spl_k$.
\end{remark}

\begin{corollary}
Let's assume, from Proposition \ref{reciprocal} with $K = k^\acyc$, 
that $\CH_e^{S_e} = 1$. 
Then, the anti-cyclotomic logarithmic class group $\Gal(H_k^\lac/k^\acyc)$ is trivial.
\end{corollary}

\begin{proof}
Results from $\delta_p(k) = 0$ and the equality $H^\spl_k = K_e$. 
\end{proof}

\begin{remarks}
(i) Examples, analogous to that given by Ozaki, may be computed to 
obtain $p^e \leq \lambda_p(k^\acyc/k) \leq 2p^e-1$, $\mu_p(k^\acyc/k) = 0$.
Program \ref{ozakiP} finds all such examples, and it is not difficult to
conjecture that they are infinite in number. Unfortunately, no
computation in the tower is accessible and it is difficult to see if
$\lambda_p(k^\acyc/k) > p^e$ is possible.

\smallskip
(ii) In \cite[Section 3\,(3), p. 390]{Oza2001}, the author proves the inequality
$\lambda_3(K/k) \geq 3$ as follows (we take $K = k^\acyc$ to simplify and 
put $H_k^\nr =: K_e$, $e \geq 1$, $p \geq 3$).
He considers $\wt K_e$, the compositum of all the $\Z_p$-extensions 
of $K_e$, for which $\Gal(\wt K_e/K_e) \simeq \Z_p^{p^e+1}$, whence 
$\Gal(\wt K_e/K) \simeq \Z_p^{p^e}$ and $KK_e^\cyc = K^\cyc = \wt k$. 
The extensions $K^\cyc/K_e^\cyc$ and $K^\cyc/K$ are unramified (see Diagram \eqref{ramification}).
Then, he refers that the inertia group $I_{{\mathfrak P}_e}$ of some 
${\mathfrak P}_e \mid p$ in $\Gal(\wt K_e/K_e)$, is isomorphic to $\Z_p$, 
and deduces that $\wt K_e/K$ is unramified, giving $\lambda_p(k^\acyc/k) \geq 3$.
\end{remarks}

\subsection{Case of the non-Galois \texorpdfstring{$\Z_p$-extensions
of $k$}{Lg}}

Let $K/k$ be a $\Z_p$-extension defined, from Theorem \ref{6cases}, 
by $G_K = \langle \sigma_-^{p^\alpha u} \!\cdot \sigma_+ \rangle$,
$\alpha \geq 0$, $u \sim 1$, $u \ne 1$; this means that
$K \cap k^\cyc = k$ and $K \cap k^\acyc = k_\alpha^\acyc$. Recall 
that $p^{\wt e} := [k^\acyc \cap H_k^\nr : k]$, only depends on $k$. 

\begin{theorem}\label{nongal}
Let $k$ be an imaginary quadratic field, and let $p \geq 3$ split in $k$.
Assume that $H_k^\nr \subset k^\acyc$, that ${\mathfrak p}$ is $p$-principal 
and that $\delta_p(k) = 0$. Let $K/k$ be a $\Z_p$-extension such that $\alpha 
\geq \wt e$ (Cases (iii) and (iv) of Theorem \ref{6cases}, with $u-1 \sim 1$ 
if $\wt e = 0$, in which case, $e = \ov e = \wt e$). Then $\lambda_p(K/k) \geq p^e$ 
for this infinite family of $\Z_p$-extensions $K/k$.
\end{theorem}

\begin{proof}
The above conditions imply $H_k^\nr \subset K$ because 
$\alpha \geq \wt e$ means $[K \cap k^\acyc : k] = p^\alpha$, hence
$H_k^\nr \subseteq K \cap k^\acyc$, with 
$L \cap \ov L = H_k^\nr$; then $K/H_k^\nr$ is totally ramified at the 
two $p$-places and point (i) of the proof of Theorem \ref{H} gives the 
result ($e = \ov e = v_p(\order \CH_k)$). 
\end{proof}

\begin{theorem}\label{deltanul}
Let $k$ be an imaginary quadratic field, and let $p \geq 3$ split in $k$. 
Assume that $\CH_k = 1$ and $\delta_p(k) = 0$. Let $K/k$ be a 
$\Z_p$-extension, $K \ne L$, $\ov L$. 
Then $\mu_p(K/k) = 0$ and $\lambda_p(K/k) = 1$.
\end{theorem}

\begin{proof}
This may be deduced from Theorem \ref{stability} (iii), but we prove again
this result by adapting the proof of Theorem \ref{H}. Then $\wt e = 0$
by assumption, which 
implies $\ov e = 0$, $e \geq 0$ depending on $u$ in the expression
$G_K = \langle \sigma_-^{p^\alpha u} \!\cdot \sigma_+ \rangle$
(Case \ref{6cases}\,(iii), $e = v_p(u-1)$, with $u-1 \sim 1$). 
Hence, ${\mathfrak p}$ is unramified in $K_e/k$, totally ramified 
in $K/K_e$, and ${\ov {\mathfrak p}}$ is totally ramified in $K/k$.
Since $\wt \CH_k = 1$, $H_k^\lc = k^\cyc$, then $p$ is totally inert in 
$\wt k/k^\cyc$. Hence, ${\mathfrak p}$, ${\ov {\mathfrak p}}$ do not split in $K_e/k$.

\medskip
(i) {\bf Invariant classes in $K_n/k$}.
Chevalley--Herbrand formula in $K_n/k$ gives, for $n \geq e$,
$\order \CH_n^{G_n} = \ffrac{\order \CH_k \times[p^n \, p^{n-e}]} 
{[K_n : k] \times 1} = p^{n-e}$.
Let ${\mathfrak P}_n$, ${\mathfrak P}'_n$ be the prime 
ideals of $K_n$ above ${\mathfrak p}$, ${\ov {\mathfrak p}}$.
Then, for $n \geq e$, we have $\BJ_n({\mathfrak p}) = 
({\mathfrak P}_{\!n})^{p^{n-e}}$, $\BJ_n({\ov {\mathfrak p}}) 
= ({\mathfrak P}'_{\!n})^{p^{n}}$.
So, $\CH_n^{G_n} = \CH_n^\ram = \Ccl_n
(\langle  {\mathfrak P}_n , {\mathfrak P}'_n \rangle)$ (from \eqref{ambigeES})
of order $p^{n-e}$.

\medskip
(ii) {\bf Invariant $S_n$-classes in $K_n/k$}.
We have, for $S_n = \{{\mathfrak P}_n , {\mathfrak P}'_n\}$:
\begin{equation*}
\order (\CH_n^{S_n})^{G_n} = \frac{\prod_{v \in S_k} d_v} {[K_n : k] (E_k^{S_k} :
E_k^{S_k} \cap \BN_n (K_n^\times))} = \frac{p^n \times p^n} 
{p^n \cdot \omega_n (\langle x, \ov x \rangle)} = \frac{p^n}
{\order \omega_n(\langle x, \ov x \rangle)} = 1;
\end{equation*}

\noindent
indeed, since $\wt \CH_k = 1$, there is no splitting in $\wt k/k$, thus
$d_v = p^n$ for all $v \in S_k$; then $\omega_n(x)$ is of order 
$p^{n - \ov e} = p^n$ and $\omega_n(\ov x)$ is of order $p^{n - e}$,
which ends the computation. Then $\CH_n = \Ccl_n(\langle S_n\rangle)$,
for all $n\geq e$.

\medskip
(iii) {\bf Majoration of $\order \CH_n$ and $\lambda_p(K/k)$}.
Since the two ideals of $S_n$ are invariant by $g_n = \Gal(K_n/K_e)$, 
we have the exact sequence (cf. \eqref{ambigeES}):
$$1 \to  \Ccl_n(\langle {\mathfrak P}_n , {\mathfrak P}'_n \rangle) = \CH_n
\too \CH_n^{g_n} \too E_e \cap \BN_{n/e}(K_n^\times)/\BN_{n/e}(E_n) \to 1. $$

So, $\CH_n = \CH_n^{g_n}$ of order $\ffrac {p^{n-e}\, p^{n-e}}{p^{n-e} \times 
(E_e : E_e \cap \BN_{n/e} (K_n^\times))} \leq p^{n-e}$, yielding to $\mu_p(K/k) 
= 0$ and $\lambda_p(K/k) \leq 1$, hence $\lambda_p(K/k) = 1$; since the
$n$-sequence $(E_e : \BN_{n/e}E_n)$ is decreasing, stationary from some
layer $n_1$, we have $\nu_p(K/k) = -e - v_p((E_e : \BN_{n_1/e}E_{n_1}))$.
\end{proof}

When $K \cap k^\cyc =k_\beta^\cyc$, with $\beta \ne 0$, the proof is similar
with $e = \ov e = 0$ since $k^\cyc/k$ is totally ramified at ${\mathfrak p}$,
${\ov {\mathfrak p}}$. If $\delta_p(k)$ is not assumed trivial in the 
theorem, then $\order (\CH_n^{S_n})^{G_n} = p^{\delta_p(k)}$.

\begin{appendices}

\section{Logarithmic classes of imaginary abelian number fields
\texorpdfstring{\\}{Lg}
by J-F.  Jaulent} \label{A}

Let $k$ be an abelian number field. It is well known that such a field 
satisfies the so-called Gross--Kuz'min conjecture for every prime 
number $p$, i.e. that its logarithmic class group $\wt \CH_k$ has finite 
order, see \cite{Jau1994}. The computation of this last group for 
arbitrary number fields via the {\sc pari/gp} system is given in 
Belabas--Jaulent \cite{BJ2016}.

\smallskip
We are interested, in this appendix, to give an explicit formula for the orders 
of the imaginary components of $\wt \CH_k$ in the semi-simple case.
\smallskip

We assume from now on that $k$ is an imaginary abelian number field and that 
the prime number $p$ does not divide $d=[k:\Q]$. Thus the $p$-adic algebra 
$\Z_p[\Delta]$ of the Galois group $\Delta=\Gal(k/\Q)$ is a direct product indexed 
by the primitive idempotents $e_\varphi=\frac{1}{d}\sum_{\tau\in\Delta}
\varphi(\tau)\tau^{-1}$ associated to the $\Q_p$-irreducible characters $\varphi$:

\smallskip
\centerline{$\Z_p[\Delta] = \prod_\varphi \Z_p[\Delta]\,e_\varphi=
\prod_\varphi \Z_\varphi$,}

\smallskip
\noindent 
where $\Z_\varphi=\Z_p[\Delta]\,e_\varphi$ is an unramified extension 
of $\Z_p$ of degree $d_\varphi={\rm deg}\,\varphi$.

As a consequence, the $p$-class group $\CH_k$, as well as the logarithmic 
$p$-class group $\wt\CH_k$, decompose as the product of their $\varphi$-components; 
likewise do their respective wild subgroups $\CH_k^w$ and $\wt\CH_k^w$ 
constructed on the places $\mathfrak p$ over $p$. Now, we have:

\begin{lemma}\label{LA1}
For any irreducible character $\varphi$ the $\varphi$-components of the tame 
quotients $\CH_k':=\CH_k/\CH_k^w$ and  $\wt\CH_k':=\wt\CH_k/\wt\CH_k^w$ 
are $\Z_\varphi$-isomorphic. From where the equality between orders: 
$\order(\wt\CH_k)^{e_\varphi}/\order(\wt\CH_k^w)^{e_\varphi}=
\order(\CH_k)^{e_\varphi}/\order(\CH_k^w)^{e_\varphi}$.
\end{lemma}

\proof Let us denote $H_k^\lc$ the maximal abelian pro-$p$-extension of $k$ which 
completely splits, at every place, over the cyclotomic $\Z_p$-extension $k^\cyc$; 
and let $H_k'$ be its maximal sub-extension which completely splits over $k$ at 
every place dividing $p$. By $p$-adic class field theory (Jaulent \cite{Jau1998}), 
we have the isomorphisms:
\smallskip

\centerline{$\wt\CH_k \simeq \Gal(H_k^\lc/k^\cyc)$
\qquad\&\qquad$\CH'_k \simeq \Gal(H'_k/k)$.}

\smallskip
\noindent Now, for the canonical map $\Gal(H_k^\lc/k^\cyc) \to \Gal(H'_k/k)$ 
given by the restriction of Galois automorphisms, the kernel is isomorphic to 
the wild subgroup $\wt\CH_k^w$ of $\wt\CH_k$ and the cokernel 
$\Gal((k^\cyc \cap H_k')/k)$ is here trivial, since we have $p\nmid d$.

\medskip
Let us assume now that $\varphi$ is imaginary and denote 
$[\mathfrak p_\varphi]=[\mathfrak p]^{e_\varphi}$ the $\varphi$-component 
of the class in $\CH_k$ of any prime $\mathfrak p$ above $p$. 
Let $w_\varphi\in\mathbb N$ be the order of $[\mathfrak p_\varphi]$. 
As a consequence, taking the $p$-adifications $\mathcal D_k:=D_k\otimes_\Z\Z_p$ 
and $\mathcal E'_k:=E'_k\otimes_\Z\Z_p$ of the groups of divisors 
and $p$-units of $k$, we can write:

\smallskip
\centerline{$\mathfrak p_\varphi^{w_\varphi}=\mathfrak p^{w_\varphi e_\varphi}
=(\alpha_\varphi)$, with $\alpha_\varphi$ isotypic in $\mathcal E'_k$.}

\smallskip
Let us observe that $\alpha_\varphi$ is uniquely defined (except in the special 
case where $k$ contains the $p$th roots of unity and $\varphi$ is the Teichm\"uller 
character $\omega$, where it is defined up to a root of unity) and that 
$\alpha_\varphi$ generates the $\varphi$-component of~$\mathcal E'_k$. Thus:

\smallskip
\centerline{$\order(\CH_k^w)^{e_\varphi} = (\Z_\varphi:w_\varphi\Z_\varphi)=
w_\varphi^{\,\deg\varphi}$.}

\smallskip
Likewise the logarithmic valuation of $\alpha_\varphi$ is 
$\wt\nu_{\mathfrak p}(\alpha_\varphi):= 
\frac{1}{\deg\mathfrak p}\log_p(N_{k_{\mathfrak p}/\Q_p}(\alpha_\varphi))$, 
where $\log_p$ is the Iwasawa logarithm in $\Q_p$ (see \cite{Jau1994}). 
So, we get (with the convention $x \sim y$ whenever $x/y$ is a $p$-adic unit):

\smallskip
\centerline{$\order(\wt\CH_k^w)^{e_\varphi} = (\Z_\varphi:\wt w_\varphi\Z_\varphi)=
\wt w_\varphi^{\,\deg\varphi}$,}

\smallskip
\noindent where $\wt w_\varphi \sim \frac{1}
{\deg\mathfrak p}\log_p(N_{k_{\mathfrak p}/\Q_p}(\alpha_\varphi))$ is 
the order in $\mathbb N$ of the class $[\,\wt{\mathfrak p}_\varphi]= 
e_\varphi[\,\wt{\mathfrak p}\,]$ of the logarithmic 
divisor constructed on $\mathfrak p$.

\smallskip
In summary, we have:

\begin{theorem}\label{JFJ}
Let $k$ be an abelian imaginary number field, $p\nmid d=[k:\Q]$ be a prime number, 
$\Delta_p$ be the decomposition subgroup of $p$ in $\Delta=\Gal(k/\Q)$ and 
$\chi_p:={\rm Ind}_{\Delta_p}^\Delta 1_{\Delta_p}$ the character of $\Delta$ 
induced by the unit character of $\Delta_p$. Then:
\begin{itemize}
\item[(i)] For each irreducible character $\varphi$ of $\Delta$ that does not appear 
in $\chi_p$, the $\varphi$ components of the logarithmic class group 
$\wt\CH_k^{e_\varphi}$ and of the classical class group $\CH_k^{e_\varphi}$ 
are $\Z_\varphi$-isomorphic; so, they have same orders: $\wt h_\varphi^{\,\deg\varphi} 
= h_\varphi^{\,\deg\varphi}$.

\item[(ii)] For each imaginary $\varphi$ that does appear in $\chi_p$, the integer 
$\,\wt h_\varphi$ is given by the formula: $\wt h_\varphi \sim  \frac{1}{\deg\mathfrak p} 
\log_p(N_{k_{\mathfrak p}/\Q_p}(\beta_\varphi))$, where $\beta_\varphi$ is 
an isotypic generator of the principal divisor $\mathfrak p^{h_\varphi e_\varphi}$.
\end{itemize}
\end{theorem}

\proof 
Let us observe first that every finite $\Z_\varphi$-module is the direct sum of cyclic 
submodules. So, we can write its order as the $\deg\varphi$-th power of an integer.

($i$) If the irreducible character $\varphi$ does not appear in $\chi_p$, the 
$\varphi$-components of the wild subgroups $\wt\CH^w_k$ and $\CH^w_k$ 
are trivial. So, the Lemma \ref{LA1} gives directly:

\smallskip
\centerline{$\wt\CH_k\simeq\CH_k$.} 

($ii$) Otherwise, we get however, accordingly to Lemma \ref{LA1}:

\smallskip
\centerline{$\wt h_\varphi = \dfrac{h_\varphi \wt w_\varphi}{w_\varphi} 
\sim \dfrac{h_\varphi}{w_\varphi}\, \big(\frac{1}{\deg{\mathfrak p}}
\log_p(N_{k_{\mathfrak p}/\Q_p}(\alpha_\varphi))\big)
= \frac{1}{\deg{\mathfrak p}} \log_p(N_{k_{\mathfrak p}/\Q_p}(\beta_\varphi))$,}

\smallskip
\noindent  
where $\beta_\varphi=\alpha^{h_\varphi/w_\varphi}$ is an isotypic 
generator of the principal divisor $\mathfrak p_\varphi^{h_\varphi}=
\mathfrak p^{h_\varphi e_\varphi}$.

\begin{corollary}
In case the prime $p$ completely splits in $k$, we get anyway for every imaginary 
irreducible character $\varphi$:

\smallskip
\centerline{$\wt h_\varphi \sim  (\frac{1}{p} \log_p(\beta_\varphi))^{\deg \varphi}$.}
\end{corollary}

\begin{corollary}\label{qc}
In the particular case $k$ is an imaginary quadratic field and the prime $p$ 
decomposes in $k$ as $(p)=\mathfrak p\bar{\mathfrak p}$, the order 
$\wt h=\order\wt\CH_k$ satisfies the identity:

\smallskip
\centerline{$\wt h \sim  (\frac{1}{p} \log_p(\beta))$, where $\beta$ is a generator 
of $(\mathfrak p/\bar{\mathfrak p})^{h_{k,p}}$ with $h_{k,p}=\order\CH_k$.}
\end{corollary}

\proof 
In this later case there are exactly two irreducible characters $1$ and $\varphi$. 
The 1-components, i.e. the real components of ordinary or logarithmic $p$-class groups 
coincide with the respective groups of $\Q$, which are trivial. So, the whole groups 
coincide with their $\varphi$-components, i.e. the imaginary ones, and we have the 
equality between the whole orders: $h:=h_1h_\varphi=h_\varphi$; and 
$\wt h:=\wt h_1\wt h_\varphi=\wt h_\varphi$.

\section{Parametrisation and Ramification of the 
\texorpdfstring{$\Z_p$}{Lg}-extensions in 
\texorpdfstring{$\wt k/k$}{Lg}}\label{B}

Let $p \geq 3$, split in $k$; set $(p) = {\mathfrak p} {\ov {\mathfrak p}}$. We refer to 
Diagram \eqref{maindiagram} and we will justify the existence then give a description 
of the non-Galois $\Z_p$-extensions $K/k$, totally ramified at ${\mathfrak p}$
from some layer $K_e$ and at ${\ov {\mathfrak p}}$ from some layer $K_{\ov e}$, 
assuming for instance $e \geq \ov e$, so that $K_{\ov e} = K^\nr := K \cap H_k^\nr$, 
the maximal unramified sub-extension of $K$; whence, $K^\nr \subseteq \knr 
:= \wt k \cap H_k^\nr = k^\acyc \cap H_k^\nr$. We will show that, in general, 
$e = \ov e$, but there are infinitely many counterexamples 
and the case $p = 3$ yields to a pathological situation.

\subsection{Parametrisation of \texorpdfstring{$\Gal(K/k)$}{Lg}
in \texorpdfstring{$\Gal(\wt k/k)$}{Lg}}
Let $\sigma_-$ and $\sigma_+$ be topological generators of 
$\Gamma^- = \Gal(\wt k/k^\cyc)$ and $\Gamma^+ = \Gal(\wt k/k^\acyc)$, 
respectively; we define any sub-extension $F$ of $\wt k$ by means of 
generators of $G_F := \Gal(\wt k/F) \subseteq \Gamma := \langle \sigma_-, 
\sigma_+ \rangle = \Gamma^- \oplus \Gamma^+$ and compute the quotient 
$\Gamma/G_F$. Any subgroup $G$ of $\Gamma$, $G \ne 1, \Gamma$, is 
free, whence of $\Z_p$-rank $1$.

\smallskip
Recall that the $\Z_p$-extensions $L$ and $\ov L$ are, in 
$\wt k/k$, the inertia fields of ${\mathfrak p}$ and ${\ov {\mathfrak p}}$, 
respectively, and that $L \cap \ov L = \knr$; so, $K_{\ov e} = 
K \cap \ov L \subseteq K_e = K \cap L$.

\smallskip
A field $K \subset \wt k$ is a $\Z_p$-extension of $k$ if and only if $G_K = 
\langle \sigma_-^a \! \cdot \sigma_+^b \rangle$ with $a, b \in \Z_p$ such that 
$a$ or $b$ is a $p$-adic unit (otherwise, $\Gal(K/k) \simeq \Gamma/G_K$ would 
have a non-trivial torsion element). 
Then $G_L = \langle \sigma_-^x \! \cdot \sigma_+^y \rangle$, where 
$x, y \in \Z_p$ are both non-zero and not both in $p\Z_p$,
$G_{\ov L} = G_L^\tau := \tau \circ G_L \circ \tau^{-1}
= \langle \sigma_-^{-x} \cdot \sigma_+^y \rangle$ since $\tau$ 
commutes with $\sigma_+$ and acts by inversion on $\sigma_-$.

\smallskip
Set $U \sim V$ in $\Q_p \backslash \{0\}$, if $V \,U^{-1} \in \Z_p^\times$ and put:
$$a = p^\alpha u,\  b  =  p^\beta v, \ \, \alpha \cdot \beta = 0, \ \ \
x = p^s u',\  y = p^t v', \ \, s \cdot t = 0,\ \ \
u \sim v \sim u' \sim v' \sim 1.$$

\subsection{Ramification of the \texorpdfstring{$p$}{Lg}-places 
in \texorpdfstring{$K/k$}{Lg}}

One may choose the generators $\sigma_-$ and $\sigma_+$ such that $u' = v' = 1$,
so that $x = p^s$, $y = p^t$, and one may choose the generator 
$\sigma_-^{p^\alpha u} \! \cdot \sigma_+^{p^\beta v}$ of $G_K$ such
that $v = 1$, giving the bijective parametrization:
$$G_K = \langle \sigma_-^{p^\alpha u} \!\cdot \sigma_+^{p^\beta} \rangle,
\ \  \hbox{$\alpha = 0$ or $\beta = 0$, $u \sim 1$}. $$

We deduce that $[K \cap k^\cyc : k] = p^\beta$ and $[K \cap k^\acyc : k] = 
p^\alpha$. Put $\wt e := v_p([\knr : k])$ and consider the following cases:

\smallskip
(i) Case $K \cap k^\cyc \ne k$ ($\beta \ne 0$, $\alpha = 0$); then
$K/k$ is totally ramified at the two $p$-places ($\ov e = e = 0$).

\smallskip
(ii) Case $K \cap k^\acyc \ne k$, with $\wt e = 0$ (linear disjunction of
$k^\acyc$ and $H_k^\nr$); similarly $\alpha \ne 0$, $\beta = 0$
and $\ov e = e = 0$. 

\smallskip
(iii) When $\wt e = 0$ and $K \cap k^\cyc = K \cap k^\acyc = k$ 
($\alpha = \beta = 0$), many cases occur with $K/k$ totally ramified
at ${\ov {\mathfrak p}}$ ($\ov e = 0$, $e \geq 0$).

\smallskip
(iv) When $\wt e \ne 0$, the discussion depends on $K \cap k^\acyc$, 
hence of $\alpha$ and $u$; if $K \cap k^\acyc \ne k$, it is clear that 
$\ov e$ and $e$ are both non-zero.

\smallskip
So we have the following array giving for each field $F$, the
Galois group $G_F$ and the required structure of $\Gal(F/k) 
\simeq \Gamma/G_F$ as $\Z_p$-module; if $\beta \ne 0$ (i.e. 
$K \cap k^\cyc \ne k$), then $\alpha = e = \ov e = 0$ (Case (i) 
of Theorem \ref{6cases}). So, we may assume $\beta = 0$ in 
what follows (recall that $s \cdot t=0$):
\begin{equation}
\left\{\begin{aligned}
& K  & G_K &= \langle \sigma_-^{p^\alpha u} \!\cdot 
\sigma_+  \rangle,  
\hspace{0.2cm} &&  \Z_p & \\
& L ,\ \ov L & G_L &= \langle \sigma_-^{p^s} \!\cdot 
\sigma_+^{p^t} \rangle, \hspace{0.2cm} G_{\ov L} 
= \langle \sigma_-^{-p^s}\! \cdot 
\sigma_+^{p^t} \rangle \hspace{0.2cm}  
\hspace{0.2cm} && \Z_p,\  \Z_p & \\
& L \cap \ov L = \knr  & G_{L \cap \ov L} &
= \langle \sigma_-^{p^s} \!\cdot 
\sigma_+^{p^t} , \sigma_-^{-p^s}\! \cdot \sigma_+^{p^t}
\rangle = \langle \sigma_-^{p^s}, \sigma_+^{p^t} \rangle 
\hspace{0.2cm}  && \Z/p^{\wt e} \,\Z &  \\
& K \cap L = K_e& G_{K \cap L} &= \langle \sigma_-^{p^\alpha u} 
\!\cdot \sigma_+ , \sigma_-^{p^s}\! \cdot \sigma_+^{p^t}
\rangle \hspace{0.2cm} && \Z/p^e \Z & \\
& K \cap \ov L = K_{\ov e} & G_{K \cap \ov L}  &= 
\langle \sigma_-^{p^\alpha u} \!\cdot 
\sigma_+ , \sigma_-^{-p^s}\! \cdot \sigma_+^{p^t}
\rangle  \hspace{0.2cm} && \Z/p^{\ov e} \Z & \\
\end{aligned} \right .
\end{equation}

Let's denote with the symbol $\approx$ the matricies
equivalence associated to the theory of free $\Z_p$-modules.

\smallskip
Since $ L \cap \ov L = \knr \subset k^\acyc$, $\Gamma/G_{L \cap \ov L}$ 
is a cyclic group of order $p^{\wt e}$; this implies that the matrix:
$$\begin{pmatrix}
p^s & p^t  \\
-p^s & p^t  \\
\end{pmatrix} \approx
\begin{pmatrix}
p^s & 0  \\
0 & p^t  \\
\end{pmatrix}$$
must be equivalent to the matrix 
$\begin{pmatrix}
p^{{\wt e}}  & 0 \\
0 & 1  \\
\end{pmatrix}$;
this yields $(s,t) \in \{ ( {\wt e}, 0), (0, {\wt e})\}$. So, $t = 0$, $s = \wt e$, 
since the projections of $G_L$, $G_{\ov L}$ on $\Gal(k^\cyc/k)$ are 
onto. 

\smallskip
Since $\Gal(K \cap \ov L/k) \simeq \Z/p^{\ov e} \Z$ and 
$\Gal(K \cap L/k) \simeq \Z/p^e \Z$, necessarily we obtain:
$$\begin{pmatrix}
p^\alpha u  & 1 \\
-p^{\wt e} & 1 \\
\end{pmatrix} 
\approx \begin{pmatrix}
p^{\ov e}  & 0 \\
0 & 1  \\
\end{pmatrix} \ \ {\rm and}\ \  
\begin{pmatrix}
p^\alpha u  & 1 \\
p^{\wt e} & 1 \\
\end{pmatrix} 
\approx \begin{pmatrix}
p^e  & 0 \\
0 & 1  \\
\end{pmatrix}. $$

If $\wt e = 0$, $\alpha \ne 0$ (Case \ref{6cases}\,(ii)), then
$\ov e = e = 0$.

\smallskip
If $\wt e = 0$, $\alpha = 0$ (Case \ref{6cases}\,(iii)), then 
$\ov e = 0$, $e = v_p(u-1)$ (the case $\alpha = 0$, $u = 1$ or
$u = -1$ do not occur since in that cases, $K = L$ or
$K = \ov L$, excluded).

\smallskip
Now, assume $\wt e \ne 0$, and consider the two matrices 
$\begin{pmatrix}
p^\alpha u & 1 \\
\pm p^{\wt e} & 1 \\
\end{pmatrix}$:

If $\alpha \ne \wt e$, the two matrices are equivalent, so $\ov e = e = \min(\alpha, \wt e)$
(Case \ref{6cases}\,(iv)). 

If $\alpha = \wt e$, the two matrices are equivalent to
$\begin{pmatrix}
p^{\wt e} (u \pm 1) & 0 \\
0 & 1  \\
\end{pmatrix}$
giving $\ov e = e = \wt e$ as soon as $u \pm 1$ are both $p$-adic 
units (Case \ref{6cases}\,(v))), otherwise, if $u + 1 \sim 1$ and 
$u - 1 \sim p^{\rho}$, $\rho \geq 1$, then $\ov e = \wt e$ and 
$e = \wt e + \rho$ (Case \ref{6cases}\,(vi)).

\smallskip
In summary, we have obtained the following description of the ramification
indices of ${\mathfrak p}$ and ${\ov {\mathfrak p}}$ in a non-Galois 
$\Z_p$-extension $K$, distinct from $L$ and $\ov L$, in the split case:

\begin{theorem}\label{6cases}
Let $K$ be the $\Z_p$-extension fixed by $G_K = \langle 
\sigma_-^{p^\alpha u} \! \cdot \sigma_+^{p^\beta} \rangle$, $u \sim 1$,
$u \ne \pm 1$, with $\alpha \cdot \beta = 0$.
This describes all the non-Galois $\Z_p$-extensions
distinct from $L$ and $\ov L$. Put $\wt e = v_p([\knr : k])$, where 
$\knr := \wt  k \cap H_k^\nr = k^\acyc \cap H_k^\nr$.
We have the following values of $\ov e$ and $e$:
\begin{equation*}
\left \{\begin{aligned} 
&\hbox{$(i)$}  &&\hbox{$ \ov e = e = 0$, if $K \cap k^\cyc \ne k$ 
($\alpha = 0$, $\beta \ne 0$);} \\
&\hbox{$(ii)$}  &&\hbox{$ \ov e = e = 0$, if $\wt e = 0$ and 
$K \cap k^\acyc \ne k$ ($\alpha \ne 0$, $\beta = 0$);} \\
&\hbox{$(iii)$}  &&\hbox{$ \ov e = 0$, $e = v_p(u - 1)$, if $\wt e = 0$, 
$K \cap k^\cyc = K \cap k^\acyc = k$ ($\alpha = \beta = 0$);} \\
&\hbox{$(iv)$}  &&\hbox{$ \ov e = e = \min(\alpha, \wt e)$, if $\wt e \ne 0$, 
$\alpha \ne \wt e$, $\beta = 0$;} \\
&\hbox{$(v)$}  &&\hbox{$ \ov e = e = \wt e$, if $\wt e \ne 0$, $\alpha = \wt e$, 
$u + 1 \sim 1$, $u - 1 \sim 1$, $\beta = 0$;} \\
&\hbox{$(vi)$}  &&\hbox{$ \ov e = \wt e$, $e = \wt e + v_p(u - 1)$, if $\wt e \ne 0$, 
$\alpha = \wt e$, $u + 1 \sim 1$, $\beta = 0$.}
\end{aligned}\right .
\end{equation*}
\end{theorem}

\subsection{Case \texorpdfstring{$p=3$}{Lg} and examples}
For $p = 3$, split in $k$, and $\alpha = \wt e$, the above last exceptional Case
\ref{6cases}\,(vi) always occurs and Case \ref{6cases}\,(v) never occurs
since $\Z_3^\times = \{\ov 1, \ov 2\}$, which gives $e > \ov e$; numerical 
examples require too large degrees, except if $\wt e = 0$ and $K_1 = L_1$, 
which yields the total ramification of ${\ov {\mathfrak p}}$ ($\ov e = 0$, $e \geq 1$). 
In that case, any non-Galois $\Z_3$-extension $K/k$, linearly disjoint from 
$k^\cyc$ and $k^\acyc$ (i.e. $\alpha = \beta = 0$), can not be totally 
ramified at ${\mathfrak p}$; we have the following numerical illustrations 
for some $k = \Q(\sqrt{-m})$ with $3$ split in $k$:

\smallskip
(i) Two cases ($\Q(\sqrt{-107})$, $\Q(\sqrt{-302})$) where $\wt e = 0$ 
with the single prime ideal ${\ov {\mathfrak p}}$ ramified in $K_1 = L_1$; 
in the first (resp. second) case, ${\mathfrak p}$ splits (resp. is inert).

\smallskip
(ii) An example ($\Q(\sqrt{-362})$) where $\wt e \geq 1$, $\alpha = \ov e = e = 0$ with 
non-Galois $K_1\ne  k_1^\cyc, k_1^\acyc$. 

\smallskip
We compute the degree $6$ defining polynomial of $K_1$, and the 
factorisation of $(3)$; for each prime ${\mathfrak P}_1 \mid 3$ of 
$K_1$, the third and fourth components give the ramification index 
and the residue degree; the last one ($1$ or $3$) is the power 
of ${\mathfrak P}_1$ in $(3)$, also giving the ramification:

\ft\begin{verbatim}
m=107 kronecker(-m,3)=1 Hk=[3] k_1^ac/k Ramified
Q^ac=x^3+6*x-17 Hkacyc=[9,3]
x^6+2178*x^5+1978263*x^4+959122620*x^3+261780977439*x^2
+38136284905122*x+2316602639725929
[[3,[-1,0,-1,0,-1,-1]~,3,1,
  [0,12,18,6,-6,27;2,2,-5,-11,-5,8;2,-4,-8,1,-8,-4;
  2,5,7,13,4,-4;2,-4,-5,-8,-2,-1;1,10,8,2,5,-8]],3;
[3,[-1,1,1,0,0,-1]~,1,1,
  [0,6,3,3,-6,6;1,1,0,-5,-2,3;0,0,-3,3,0,-3;
  1,1,3,4,1,0;0,0,0,-3,-3,0;0,3,3,0,0,0]],1;
[3,[2,-1,-1,-1,1,-1]~,1,1,
  [0,3,6,6,-3,12;2,3,-1,-3,-4,4;1,-3,-5,0,-2,-4;
  0,0,3,6,3,-3;1,0,-2,-3,-2,-1;0,6,6,3,3,-3]],1;
[3,[2,0,0,-1,0,1]~,1,1,
  [0,3,6,6,6,3;0,3,0,0,-3,3;1,-4,-1,-2,0,2;
  0,0,0,3,0,-3;0,0,-3,0,0,0;1,2,-1,1,0,-4]],1]
  
m=302 kronecker(-m,3)=1 Hk=[12] k_1^acyc/k Ramified
Q^ac=x^3-93*x-458 Hkacyc=[12,3]
x^6+936*x^5+338256*x^4+58423680*x^3+4781680128*x^2
+156378968064*x+1923066261504
[[3,[-1,1,0,0,0,0]~,3,1,
  [0,-15,-15,3,27,18;1,1,-1,-2,4,12;1,-2,-1,1,-2,-9;
  0,-6,-6,0,0,9;0,-6,0,3,-3,0;0,0,0,-6,3,3]] 3]
[[3,[1,0,0,-1,1,0]~,1,3,
  [0,-9,-33,-1,44,29;1,-1,-2,5,1,12;2,-2,-1,-2,-7,-10;
  1,-10,-8,2,-2,15;0,-6,3,4,-5,6;0,-3,-3,-7,2,5]] 1]
  
m=362 kronecker(-m,3)=1 Hk=[18] k_1^ac/k Unramified
Q^ac=x^3+5*x-14 Hkacyc=[6]
x^6+6720*x^5+18817440*x^4+28104987008*x^3+23613547649280*x^2
+10582052588943360*x+1976055196443480064
[[3,[0,0,0,-1,0,-1]~,3,1,
  [0,-24,-54,-72,48,0;0,3,0,33,6,24;2,7,-2,1,32,44;
  2,-5,4,-5,2,-10;1,-4,-1,-7,4,-44;0,-3,3,0,6,0]],3;
[3,[0,1,0,-1,0,1]~,3,1,
  [0,-54,-24,-24,-48,0;2,5,2,39,-4,-44;0,3,0,9,30,-24;
  2,-1,-4,-3,-4,10;1,4,-5,-3,-2,-34;0,-3,3,-6,6,0]],3]
\end{verbatim}\ns

For $k = \Q(\sqrt{-362})$, $K_1$ is obtained as one of the degree $6$ 
subfield of the compositum $M := k_1^\cyc\, k_1^\acyc$, of degree 
$18$, and the instruction ${\sf nfsubfields(M,6)}$ giving the list 
$\{k_1^\cyc, K_1, \ov K_1, k_1^\acyc\}$.

\section{Numerical computations -- {\sc pari/gp} programs} \label{C}

\subsection{Presentation of the programs}

The programs use the standard {\sc pari/gp} instructions,
computing invariants of algebraic number fields, that we recall
with our variable names:

\begin{itemize}

\item ${\sf k := bnfinit(x^2+m)}$: computes the arithmetic of the field $k = \Q(\sqrt{-m})$,

\item ${\sf basis := k.zk}$: gives an integral basis of the field $k$,

\item ${\sf bnflog(k,p)}$: gives 
$\big[\wt \CH_k,\ \wt \Ccl(S_k),\ \CH_k^{S_k} = \CH_k/\Ccl(S_k)\big ]$,

\item ${\sf Clog := bnflog(k,p)[1]}$: group structure of $\wt \CH_k$,

\item ${\sf Hk := k.cyc}$: group structure of the class group $\BH_k$, 

\item${\sf rk}$: $p$-rank of the $p$-class group $\CH_k$ ($p=3$),

\item ${\sf hk := k.no}$: class number of $k$, 

\item ${\sf hta = hk/3^{valuation(hk,3)}}$: 
tame part of $\order \BH_k$, 

\item ${\sf k.clgp}$: gives the class group $\BH_k$ with generating ideals,

\item ${\sf Kpnu := bnrinit(k,p^{nu})}$: computes the ray class field mod $p^\nu$,

\item ${\sf HKnu := Kpnu.cyc}$: group structure of the corresponding ray class group,

\item ${\sf Tk}$: gives the group structure of $\CT_k$, from that of ${\sf HKnu}$, 

\item ${\sf ot}$: gives $\order \CT_k$, ${\sf otbp}$: gives $\order \CT_k^\bp$,

\item ${\sf P := component(idealfactor(k,p),1)[1],
ovP := component(idealfactor(k,p),1)[2]}$: prime
ideals ${\mathfrak p}$, ${\ov {\mathfrak p}} \mid p$,

\item ${\sf bnfisprincipal(k,A)}$: test if the ideal  ${\mathfrak a}$
is principal or not. It gives for instance:

\quad ${\sf bnfisprincipal(k,A) = [[0,0,0],[24409,2926,-9378,2778,2724,-698]]}$, 

\noindent
where the second part gives the components of a generator of 
${\mathfrak a}$ on the integral basis. Thus, ${\mathfrak a}$ is principal 
{\it if and only if} ${\sf bnfisprincipal(k,A)[1]}$ is the zero vector as above. 
Otherise, one finds ${\sf bnfisprincipal(k,A)[1] = [3,0]}$), giving the 
components on the generating classes $h_i$ of $\BH_k$ (in this example 
$\BH_k \simeq \Z/9\Z \times \Z/3\Z$, $\Ccl({\mathfrak a}) = h_1^3 \cdot h_2^0$), 

\item ${\sf delta := idealval(k,x^{(p-1)}-1,ovP)-1}$, 
${\sf Delta := delta + [vhk-vh]}$: ${\mathfrak p}$-valuation of the Fermat quotient 
of $x$ (generator of $ {\mathfrak p}^{\hp}$) and $p$-valuation of $\wt \CH_k$,

\item For the arithmetic of the field $k^* = \Q(\sqrt{3m})$, the programs
replace any variable name ${\sf z}$ by ${\sf zstar}$ (e.g., ${\sf Pstar = x^2-3*m}$, 
${\sf kstar = bnfinit(Pstar)}$, etc.),

\item ${\sf Lw}$: list of radicals $w \in k^*$ giving the $3$-ramified
cyclic cubic extensions of $k$
\end{itemize}

For Galois computations, the programs use:

\begin{itemize}
\item ${\sf G := nfgaloisconj(K)}$: Galois group of $K/\Q$ for ${\sf K := bnfinit(R)}$ 
\item ${\sf nfgaloisapply(K,S,A)}$: gives the image by $s \in G$ of the object $A$,
\item ${\sf polgalois(R)}$: computes the Galois group of the polynomial $R$,
\item ${\sf polredbest(Q)}$: polynomial, with smaller coefficients,
defining the same field
\end{itemize}

\subsection{Program computing the Ozaki family}\label{ozakiP}
The search of the imaginary quadratic fields of Theorem \ref{H} 
for $p = 3$ is very easy.
The program gives the class groups and logarithmic class groups of $k$ 
and $k_1^\acyc$ for information. Even if the knowledge of $k^\acyc$ is not
necessary, recall that for $m \not\equiv 3 \pmod 9$, $\rk_3(\CH_{k^*}) = 
\rk_3(\CT_k)$, where $k^* = \Q(\sqrt{3m})$ \cite[Theorem 4.3]{Gra2026}; 
this simplifies the search of radicals since $\CT_k = 1$ (equivalent  
to $H_k^\nr \subset k^\acyc$). So the unique 
radical is given by the fundamental unit of $k^*$; it gives the unramified
layer $k_1^\acyc$. We only consider the cases $\CH_k \ne 1$:

\ft\begin{verbatim}
{for(m=2,10^4,if(core(m)!=m || Mod(-m,3)!=1,next);P=x^2+m;
k=bnfinit(P);hk=k.no;vhk=valuation(hk,3);if(vhk==0,next);hkta=hk/3^vhk;
\\TESTS: PRINCIPALITY OF P and delta_3(k)=0:
Prime=component(idealfactor(k,3),1);Y=idealpow(k,Prime[1],hkta);
Z=bnfisprincipal(k,Y);if(Z[1]==0,
X=Mod(Z[2][1]*k.zk[1]+Z[2][2]*k.zk[2],P);
delta=idealval(k,X^2-1,Prime[2])-1;if(delta==0,
\\TEST: T_k=1:
nu=vhk+2;ot=1;Kpnu=bnrinit(k,3^nu);Hknu=Kpnu.cyc;
d=matsize(Hknu)[2];for(j=3,d,c=Hknu[j];ot=ot*3^valuation(c,3));
if(ot==1,print();
\\MAIN INVARIANTS:
print("m=",m," S-unit x=",X);print("hk=",hk," #Tk=",ot," Clogk=",bnflog(k,3));
\\FUNDAMENTAL UNIT OF k* AND FIRST LAYER OF k^ac:
dstar=quaddisc(3*m);w=quadunit(dstar);
Q=polredbest(x^3-3*norm(w)*x-trace(w));
R=polcompositum(P,Q)[1];kacyc=bnfinit(R,1);
print("Q^ac=",Q," Hkacyc=",kacyc.cyc," Clogkacyc=",bnflog(kacyc,3))))))}

   m=239 S-unit x=Mod(x+2,x^2+239)
hk=15 #Tk=1 Clogk=[[3],[],[3]]
Q^ac=x^3-x-3 Hkacyc=[5] Clogkacyc=[[3],[3],[]]

   m=461 S-unit x=Mod(5*x-218,x^2+461)
hk=30 #Tk=1 Clogk=[[3],[],[3]]
Q^ac=x^3-7*x-18 Hkacyc=[10] Clogkacyc=[[3],[3],[]]
          (...)
   m=9770 S-unit x=Mod(10252*x-4674391,x^2+9770)
hk=84 #Tk=1 Clogk=[[3],[],[3]]
Q^ac=x^3-97*x-414 Hkacyc=[14,2] Clogkacyc=[[3],[3],[]]

   m=9857 S-unit x=Mod(329*x+49192,x^2+9857)
hk=60 #Tk=1 Clogk=[[3],[],[3]]
Q^ac=x^3-34*x-108 Hkacyc=[20,5,5] Clogkacyc=[[3],[3],[]]
time = 3,051 ms.
\end{verbatim}\ns

In the following examples, $\CH_k \simeq \Z/27\Z$ (resp. $\Z/81\Z$);
whence $\lambda_3(k^\acyc/k) \geq 27$ (resp. $81$, the only one for 
$m \leq 10^6$). So, $k_1^\acyc$ is the first layer of $H_k^\nr$ whose 
$3$-class group is of order $9$ (resp. $27$). 

\ft\begin{verbatim}
   m=174431 S-unit x=Mod(61*x-22654,x^2+174431)
hk=513 #Tk=1 Clogk=[[27],[],[27]]
Q^ac=x^3-x^2+6*x-243 Hkacyc=[171] Clogkacyc=[[27],[3],[9]]

   m=1449311 S-unit x=Mod(251*x-53246,x^2+1449311)
hk=1863 #Tk=1 Clogk=[[81],[],[81]]
Q^ac=x^3-x^2+48596722899993379033514972856892525213677868304877600822041889024 
Hkacyc=[621] Clogkacyc=[[27,3],[3],[27]]
\end{verbatim}\ns

\subsection{Programs computing \texorpdfstring{$\wt \CH_k$, $\CT_k$, 
$\CH_k$}{Lg} (in \texorpdfstring{${\sf Clogk[1], Tk, Hk}$}{Lg})}\label{P1}

\subsubsection{General program}
The integer $m$ and the prime $p$, split in $k = \Q(\sqrt{-m})$, vary in any given
interval. We give some excerpts in the case $\wt \CH_k \ne 1$ ($\wt \CH_k$ 
is computed from the {\sc pari/gp} instruction ${\sf Clogk = bnflog(k,p)}$
\cite{BJ2016} and not from expression of Theorem \ref{fq}):

\ft\begin{verbatim}
{for(m=2,10^4,if(core(m)!=m,next);k=bnfinit(x^2+m);Hk=k.cyc;
nu=valuation(k.no,3)+2;forprime(p=3,10^4,if(kronecker(-m,p)!=1,next);
Clog=bnflog(k,p);if(Clog[1]!=[],
\\COMPUTATION OF Tk (cf. previous program):
Tk=List;Kpnu=bnrinit(k,p^nu);HKnu=Kpnu.cyc;d=matsize(HKnu)[2];
for(j=3,d,c=HKnu[j];w=valuation(c,p);if(w>0,listput(Tk,p^w)));
\\COMPUTATION OF THE USUAL 3_CLASS GROUP:
Hkp=List;d=matsize(Hk)[2];for(j=1,d,
c=Hk[j];w=valuation(c,p);if(w>0,listput(Hkp,p^w)));
print("m=",m," p=",p," Clogk=",Clog," Tk=",Tk," Hkp=",Hkp))))}

m=3     p=13      Clogk=[[13],[13],[]]       Tk=[]    Hkp=[]
m=3     p=181     Clogk=[[181],[181],[]]     Tk=[]    Hkp=[]
m=3     p=2521    Clogk=[[2521],[2521],[]]   Tk=[]    Hkp=[]
m=5     p=5881    Clogk=[[5881],[5881],[]]   Tk=[]    Hkp=[]
(...)
m=1987  p=163     Clogk=[[26569],[26569],[]] Tk=[]    Hkp=[]
m=2239  p=5       Clogk=[[25],[5],[5]]       Tk=[]    Hkp=[5]
m=2334  p=11      Clogk=[[11],[11],[]]       Tk=[11]  Hkp=[11]
m=2351  p=3       Clogk=[[729],[729],[]]     Tk=[9]   Hkp=[9]
m=2759  p=3       Clogk=[[3],[3],[]]         Tk=[3]   Hkp=[27]
m=2915  p=3       Clogk=[[243],[243],[]]     Tk=[3]   Hkp=[3]
m=2963  p=13      Clogk=[[169],[169],[]]     Tk=[13]  Hkp=[13]
(...)
m=18665 p=3       Clogk=[[81],[81],[]]       Tk=[]    Hkp=[]
m=18665 p=7       Clogk=[[7],[7],[]]         Tk=[]    Hkp=[]
(...)
m=83909 p=3       Clogk=[[],[],[]]           Tk=[]    Hkp=[81]
m=83909 p=5       Clogk=[[25],[25],[]]       Tk=[]    Hkp=[]
\end{verbatim}\ns

\subsubsection{Examples of \texorpdfstring{$\wt \CH_k$}{Lg}'s of 
\texorpdfstring{$3$-rank $3$}{Lg}}\label{P2}

This table uses the previous program for $p = 3$. We do not repeat 
identical structures for the three invariants of ${\sf Clogk}$:

\ft\begin{verbatim}
m=207143     Clogk=[[9,3,3],[9],[3,3]]   Tk=[3,3]    Hkp=[3,3]
m=1654781    Clogk=[[3,3,3],[3],[3,3]]   Tk=[3,3]    Hkp=[9,3]
m=2688977    Clogk=[[9,9,3],[9],[9,3]]   Tk=[9,3]    Hkp=[9,3]
m=2750507    Clogk=[[3,3,3],[3],[3,3]]   Tk=[3,3]    Hkp=[3,3]
m=3334937    Clogk=[[9,3,3],[3],[9,3]]   Tk=[3,3]    Hkp=[9,3]
m=3527078    Clogk=[[3,3,3],[3],[3,3]]   Tk=[3,3]    Hkp=[3,3]
m=4201313    Clogk=[[3,3,3],[3],[3,3]]   Tk=[3,3]    Hkp=[9,3]
m=4455293    Clogk=[[9,3,3],[9],[3,3]]   Tk=[3,3]    Hkp=[3,3]
m=4996655    Clogk=[[3,3,3],[3],[3,3]]   Tk=[3,3]    Hkp=[243,3]
m=5176481    Clogk=[[9,3,3],[3],[9,3]]   Tk=[9,3]    Hkp=[9,3]
m=5546015    Clogk=[[9,3,3],[9],[3,3]]   Tk=[9,3]    Hkp=[27,3]
m=5894459    Clogk=[[9,3,3],[3],[9,3]]   Tk=[3,3]    Hkp=[27,3]
m=6493730    Clogk=[[3,3,3],[3],[3,3]]   Tk=[3,3]    Hkp=[27,3]
m=6740726    Clogk=[[27,9,3],[9],[27,3]] Tk=[27,3]   Hkp=[27,3]
m=7054241    Clogk=[[27,3,3],[27],[3,3]] Tk=[3,3]    Hkp=[3,3]
m=7347419    Clogk=[[3,3,3],[3],[3,3]]   Tk=[3,3]    Hkp=[81,3]
m=8596733    Clogk=[[81,3,3],[81],[3,3]] Tk=[3,3]    Hkp=[3,3]
m=9098093    Clogk=[[81,9,3],[81],[9,3]] Tk=[9,3]    Hkp=[9,3]
m=9483965    Clogk=[[9,3,3],[3],[3,3,3]] Tk=[3,3]    Hkp=[3,3,3]
m=10446266   Clogk=[[9,3,3],[3],[9,3]]   Tk=[9,3]    Hkp=[9,3]
m=10566149   Clogk=[[9,3,3],[9],[3,3]]   Tk=[9,3]    Hkp=[9,3]
m=11458877   Clogk=[[9,3,3],[3],[9,3]]   Tk=[3,3]    Hkp=[9,3]
m=11584382   Clogk=[[9,3,3],[9],[3,3]]   Tk=[9,3]    Hkp=[81,3]
m=11702378   Clogk=[[3,3,3],[3],[3,3]]   Tk=[3,3]    Hkp=[81,3]
\end{verbatim}\ns

\subsubsection{Experiments on the invariants 
\texorpdfstring{$\wt \CH_k$, $\CT_k$, $\CH_k$}{Lg}}\label{P9}

The following program computes these main invariants for the real 
or imaginary quadratic fields $k = \Q(\sqrt{\pm m})$, with $p = 3$ 
split or non-split in $k$, and with conditions at will on these invariants
(for $k$ real, ${\sf nu}$ must be large enough because of the normalized 
regulator $\CR_k$):

\ft\begin{verbatim}
\p 400
{\\choose s (imaginary fields: s=-1; real fields: s=1)
 \\choose splitting (split: kro=split; non-split: kro=nsplit)
s=-1;kro=split;d0=(5-s)/2;print();print("s=",s," kro=",kro);
for(m=5,10^5,if(core(m)!=m,next);mod3=lift(Mod(s*m,3));
if(kro==split & mod3!=1,next);if(kro==nsplit & mod3==1,next);
k=bnfinit(x^2-s*m,1);Hk=k.cyc;vhk=valuation(k.no,3);basis=k.zk;
D=divisors(k.no);N=numdiv(k.no);
\\COMPUTATION OF Clog:
Clog=bnflog(k,3);vlog=0;d=matsize(Clog[1])[2];for(j=1,d,c=Clog[1][j];
w=valuation(c,3);if(w>0,vlog=vlog+w));
\\COMPUTATION OF THE 3_CLASS GROUP:
vhk=0;Hkp=List;d=matsize(Hk)[2];for(j=1,d,c=Hk[j];
w=valuation(c,3);if(w>0,vhk=vhk+w;listput(Hkp,3^w)));
\\COMPUTATION OF Tk:
nu=vhk+2;if(s==1,nu=vlog+2);
vtk=0;Tk=List;Kpnu=bnrinit(k,3^nu);
HKnu=Kpnu.cyc;d=matsize(HKnu)[2];for(j=d0,d,c=HKnu[j];
w=valuation(c,3);if(w>0,vtk=vtk+w;listput(Tk,3^w)));
\\IMAGINARY OR REAL NON-SPLIT CASE:
if(kro==nsplit,hp=1;delta=0);
\\IMAGINARY OR REAL SPLIT CASE:
if(kro==split,
\\COMPUTATION OF hp=#Cl(S) AND delta:
P=component(idealfactor(k,3),1)[1];ovP=component(idealfactor(k,3),1)[2];
for(j=1,N,hh=D[j];Y=idealpow(k,P,hh);Z=bnfisprincipal(k,Y);
if(Z[1]!=0,next);hp=hh;vhp=valuation(hp,3);break);
\\COMPUTATION OF delta:
if(s==-1,X=Mod(Z[2][1]*basis[1]+Z[2][2]*basis[2],x^2-s*m);
delta=idealval(k,X^2-1,ovP)-1);if(s==1,delta=0));
\\ CONDITIONS ON #Clog,#Tk,#Hk,delta;to be modified at will:
if(vlog==1 & vhk>=4,
print("m=",s*m," Clogk=",Clog," Tk=",Tk," Hkp=",Hkp);
print(" HS=",Clog[3]," delta=",delta," vhp=",valuation(hp,3))))}

  m=-3671 Clogk=[[3],[3],[]] Tk=[3] Hkp=[81]  HS=[] delta=1 vhp=4
  m=-7727 Clogk=[[3],[],[3]] Tk=[] Hkp=[81]  HS=[3] delta=0 vhp=3
  m=-13829 Clogk=[[3],[],[3]] Tk=[3] Hkp=[27,3]  HS=[3] delta=0 vhp=3
  m=-38231 Clogk=[[3],[3],[]] Tk=[3] Hkp=[243]  HS=[] delta=1 vhp=5
\end{verbatim}\ns

\subsubsection{Successive maxima of 
\texorpdfstring{$\wt \delta_p(k)$}{Lg}}\label{P4}

The program shows that $\wt \delta_p(k)$ is, as expected for a Fermat 
quotient, very small, giving possibly $\lambda_p(K/k) = 1$ for $p \gg 0$ 
(see Conjecture \ref{pfinitude}):

\ft\begin{verbatim}
\p 800
{p=3;Dmax=0;for(m=2,10^6,if(core(m)!=m,next);
if(kronecker(-m,p)!=1,next);k=bnfinit(x^2+m);
basis=k.zk;Hk=k.cyc;hk=k.no;D=divisors(hk);N=numdiv(hk);
\\COMPUTATION OF THE ORDER OF THE CLASS OF A PRIME P DIVIDING p:
P=component(idealfactor(k,p),1)[1];ovP=component(idealfactor(k,p),1)[2];
for(j=1,N,hh=D[j];Y=idealpow(k,P,hh);Z=bnfisprincipal(k,Y);
if(Z[1]!=0,next);h=hh;break);vhk=valuation(hk,p);
vh=valuation(h,p);X=Mod(Z[2][1]*basis[1]+Z[2][2]*basis[2],x^2+m);
Delta=idealval(k,X^(p-1)-1,ovP)-1+vhk-vh;if(Delta>Dmax,Dmax=Delta;
\\COMPUTATION OF T_k AND H_k:
Tk=List;nu=vhk+2;Knu=bnrinit(k,p^nu);CKnu=Knu.cyc;
d=matsize(CKnu)[2];for(j=3,d,c=CKnu[j];w=valuation(c,p);
if(w>0,listput(Tk,p^w)));
Hkp=List;for(j=1,matsize(Hk)[2],c=Hk[j];w=valuation(c,p);
if(w>0,listput(Hkp,p^w)));
print("m=",m," d(X)=",Delta," Tk=",Tk," Hkp=",Hkp)))}

p=3
  m=14   d(X)=1 Tk=[]  Hkp=[]      m=2351   d(X)=6  Tk=[9]  Hkp=[9]
  m=41   d(X)=3 Tk=[]  Hkp=[]      m=8693   d(X)=7  Tk=[]   Hkp=[3]
  m=365  d(X)=4 Tk=[]  Hkp=[]      m=26243  d(X)=8  Tk=[]   Hkp=[3]
  m=971  d(X)=5 Tk=[]  Hkp=[3]     m=30611  d(X)=13 Tk=[]   Hkp=[]
p=5
  m=11   d(X)=1 Tk=[]  Hkp=[]      m=16249 d(X)=5 Tk=[5]    Hkp=[5]
  m=51   d(X)=3 Tk=[]  Hkp=[]      m=88654 d(X)=6 Tk=[]     Hkp=[]
  m=4821 d(X)=4 Tk=[]  Hkp=[]      m=173611 d(X)=8  Tk=[]   Hkp=[5]
\end{verbatim}\ns

\subsubsection{Research of primes \texorpdfstring{$p$}{Lg} such that 
\texorpdfstring{$\delta_p(k) \geq 1$}{Lg}}\label{P5}

The program gives, for $k = \Q(\sqrt{-m})$, split primes $p$ in a given 
interval, solutions to $\delta_p(k) \geq 1$, without computation of the precise 
value of $\lambda_p(k^\cyc/k)$ (cf. Theorem \ref{sands}), which needs to 
know approximations of the $p$-adic $L$-functions as is done in 
Dummit--Ford--Kisilevsky--Sands \cite{DFKS1991}. 

\smallskip
Let $x$ be the fundamental ${\mathfrak p}$-unit.
Put $x = u+v \sqrt{-m}$. The principle is to test the principality of 
${\mathfrak p}^{\hp}$, from increasing divisors 
$\hp$ (in ${\sf hP}$) of $\hk$ (in ${\sf hk}$). 
If $\hp \ne 1$, $v \sqrt{-m} \equiv -u \pmod {{\mathfrak p}^2}$, 
then $\ov x \equiv 2u \pmod {{\mathfrak p}^2}$, and it suffices to compute the 
$p$-valuation of the Fermat quotient $q = \frac{1}{p}\big((2u)^{p-1}-1\big)$. 
If $\hp = 1$, we use the relations $u^2-m v^2+2uv \sqrt {-m} 
\equiv 0 \pmod {{\mathfrak p}^2}$ and $u^2 + m v^2 = p$.
Then $v \sqrt {-m} \equiv \ffrac{p}{2u} - u \pmod {{\mathfrak p}^2}$, 
and replacing in $\ov x^{\,p-1} - 1 = (u - v \sqrt {-m})^{p-1} - 1$ yields
$\frac{1}{p} \big(\ov x^{\,p-1} - 1\big) \equiv \ffrac{(2u)^2 q + 1}{(2u)^2} 
\pmod {\mathfrak p}$, whose $p$-valuation is that of $(2u)^2 q + 1$, and 
gives $\delta_p(k) = 0$ or $\delta_p(k) \geq 1$ (see Remark \ref{borel-cantelli}).

\smallskip
All the results coincide with that obtained in the table of \cite{DFKS1991} 
by means of analytic methods with $L_p$-functions for computing 
$\lambda_p(k^\cyc/k)$, and we can check all the cases, but our program
may be used for any interval and for instance the case $k = \Q(\sqrt{-923})$, 
$\order \CH_k = 1$, $\hp = 10$, $p = 85985437$, 
does not appear in their table limited to $p < 10^7$.

\smallskip
We then have $\delta_p(k) \geq 1$ (i.e. $\lambda_p(k^\cyc/k) 
\geq 2$ when $v_p(\hp) = 0$) if and 
only if the valuation ${\sf w}$ in the program is non-zero; 
we give only few excerpts of the results:

\ft\begin{verbatim}
\p 400
{for(m=3,10^4,if(core(m)!=m,next);k=bnfinit(x^2+m);
basis=k.zk;hk=k.no;D=divisors(hk);N=numdiv(hk);
forprime(p=3,10^7,if(kronecker(-m,p)!=1,next);
\\ORDERS OF THE CLASSES OF P AND ovP:
P=component(idealfactor(k,p),1)[1];
ovP=component(idealfactor(k,p),1)[2];
for(n=1,N,hP=D[n];Y=idealpow(k,P,hP);Z=bnfisprincipal(k,Y);
if(Z[1]==0,X=Z[2][1]*basis[1]+Z[2][2]*basis[2];
\\FERMAT QUOTIENT OF THE P_UNIT x:
a=Mod(2*polcoeff(X,0),p^2);q=(lift(a^(p-1)-1))/p;
if(hP!=1,z=q);if(hP==1,z=a^2*q+1);w=valuation(z,p);
if(w>0,print("m=",m," hk=",hk," hP=",hP," p=",p);break)))))}

m=3    hk=1   hP=1     p=13            m=47   hk=5   hP=5     p=3
m=3    hk=1   hP=1     p=181           m=47   hk=5   hP=5     p=17
m=3    hk=1   hP=1     p=2521          m=47   hk=5   hP=5     p=157
m=3    hk=1   hP=1     p=76543         m=47   hk=5   hP=1     p=1193
m=3    hk=1   hP=1     p=489061        m=47   hk=5   hP=5     p=1493
m=3    hk=1   hP=1     p=6811741       m=47   hk=5   hP=5     p=1511    
m=941  hk=46  hP=23    p=5             m=998  hk=26  hP=26    p=3
m=941  hk=46  hP=46    p=7             m=998  hk=26  hP=26    p=13
m=941  hk=46  hP=46    p=11            m=998  hk=26  hP=13    p=1279
m=995  hk=8   hP=8     p=3             m=998  hk=26  hP=26    p=847277
\end{verbatim}\ns

\subsection{Complexity of the \texorpdfstring{$\Z_3$}{Lg}-extensions
\texorpdfstring{$K^\iw = k k_0^\acyc $}{Lg} having 
\texorpdfstring{$\mu_3(K^\iw/k) >0$}{Lg}}\label{P6}

We consider the $\Z_3$-extensions $K^\iw = k k_0^\acyc$ built \S\,\ref{nonzero}, 
over $k = k_0 F$, $k_0 = \Q(\sqrt{d})$, $d \in \{-2, -11\}$, $F$ being a cubic 
field of suitable conductor $f = \ell_1 \ell_2$ ($\ell_1$, $\ell_2$ totally inert in 
$k_0^\cyc/\Q$). We use some polynomials defining the first layer of 
$k_0^\acyc$ given by Program \ref{P7} (some were given by Kim--Oh 
\cite[Table I]{KO2004}), whence $K^\iw_1$, and the following program, where 
the two fields $F$ of conductor $f$ are considered; ${\sf split \in \{0,1\}}$ 
indicates the splitting ($0$) or not ($1$) of $3$ in $F$ (hence in $k$ since 
for $d \in \{-2, -11\}$, $3$ splits in $k_0$); in this context, $3$ is totally 
ramified in $K^\iw/k$, which is crucial for capitulation aspects.

\smallskip
Since these examples of $K^\iw/k$ are of very hight complexity, we test
if there is, at least, partial capitulation of $\CH_k$ in $K^\iw_1$ by computing 
the algebraic norm $\Bnu_{\!K^\iw_1/k}(\CH_{K^\iw_1})$. 

\smallskip
In the list ${\sf G}$, giving $\Gal(K^\iw_1/\Q)$, one selects the
elements of order $3$ and characterizes ${\sf S0}$ fixing ${\sf k}$;
then ${\sf NuA}$ computes the algebraic norm of ideals ${\sf A}$ 
generating $\CH_{K^\iw_1}$. The polynomial ${\sf Pk}$ defines 
${\sf k}$ of degree $6$, ${\sf PK1}$ defines ${\sf K1}$ of degree 
$18$, ${\sf Hk}$ and ${\sf HK1}$ are the class groups of ${\sf k}$ 
and ${\sf K1}$, respectively; the intervals of definition of ${\sf ell1}$
and ${\sf ell2}$ may be chosen at will:

\ft\begin{verbatim}
\p 800
allocatemem(800000000)
{p=3;L=[[-2,x^3-3*x-10],[-11,x^3-3*x-46]];for(i=1,2,d=L[i][1];Pk0=x^2-d;
Pkac=L[i][2];print();print("d=",d," Pkac=",Pkac);forprime(ell1=7,13,
forprime(ell2=ell1+2,37,if(Mod(ell1-1,p)!=0 || Mod(ell2-1,p)!=0,next);
if(kronecker(d,ell1)!=-1 || kronecker(d,ell2)!=-1,next);f=ell1*ell2;
for(i=1,4,PF=polsubcyclo(f,p)[i];D=nfdisc(PF);if(D!=f^2,next);print();
\\DECOMPOSITION OF 3 in F:
split=polisirreducible(PF+O(p));if(split!=0,next);
Pk=polcompositum(PF,Pk0)[1];print("ell1=",ell1," ell2=",ell2," Pk=",Pk);
k=bnfinit(Pk,1);kcyc=k.cyc;rk=matsize(kcyc)[2];
PK1=polcompositum(Pk,Pkac)[1];K1=bnfinit(PK1,1);
print("split=",split," Hk=",kcyc," HK1^iw=",K1.cyc);
\\SEARCH OF A GENERATOR S0 OF GAL(K1/k):
G=nfgaloisconj(K1);X=nfsubfields(K1,6);
for(j=1,4,PX=X[j][1];Y=Mod(X[j][2],PK1);
if(polgalois(PX)==polgalois(polcyclo(18)),
print("j=",j," PX=",PX);break);
for(kk=1,18,S=G[kk];if(nfgaloisapply(K1,S,S)==x,next);
T=nfgaloisapply(K1,S,S);T=nfgaloisapply(K1,S,T);if(T!=x,next);
YY=nfgaloisapply(K1,S,Y);if(YY!=Y,next);S0=S;break));
\\COMPUTATION OF THE ALGEBRAIC NORMS:
rK1=matsize(component(K1.clgp,2))[2];
for(j=1,rK1,A=component(K1.clgp,3)[j];
A1=A;As=nfgaloisapply(K1,S0,A);
Ass=nfgaloisapply(K1,S0,As);NuA=idealmul(K1,A1,As);
NuA=idealmul(K1,NuA,Ass);
B=bnfisprincipal(K1,NuA)[1];print(B))))))}
\end{verbatim}\ns

In the first example below, the data means (for $\Bnu_{K_1^{\iw}/k}$ 
denoted $\Bnu_1$):
$$\Bnu_1(\CH_{K_1^{\iw}}) = \langle\, \Bnu_1(h_1), 
\Bnu_1(h_2), \ldots, \Bnu_1(h_r) \,\rangle = \langle\, h_1^{201}, 
1, \ldots, 1 \,\rangle \simeq \Z/3\Z; $$

\noindent 
since the image is cyclic of order $3$, $\CH_k \simeq \Z/3\Z$ 
does not capitulate in $K^{\iw}_1$. In the selected intervals
there is very few capitulations when $\CH_k \simeq \Z/3\Z$ and 
few partial capitulations when $\CH_k \simeq \Z/3\Z \times \Z/3\Z$,
when $3$ totally splits (except one case), which enforces the heuristic 
regarding the huge complexity of $\CH_{K^\iw_1}$; but we can not say 
if a capitulation does exist in larger layers, and the sufficient condition 
of Theorem \ref{maincoro} does not hold. 

\smallskip
The following list gives cases of non-trivial capitulation
(zero matrices are not printed):

\ft\begin{verbatim}
Cyclic 3-class groups Hk:

d=-2 Pkac=x^3-3*x-10

ell1=7 ell2=37 Pk=x^6+2*x^5-165*x^4-586*x^3+6990*x^2+38488*x+60857
split=1 Hk=[93] HK1^iw=[279,3,3,3]
j=4 PX=x^6-6*x^5-1485*x^4+15822*x^3+566190*x^2-9352584*x+44364753
TOTAL CAPITULATION in K1^iw - Inertia of 3 in F

ell1=7 ell2=37 Pk=x^6+2*x^5-165*x^4-68*x^3+7508*x^2-9168*x+17604
split=0 Hk=[222,2] HK1^iw=[666,6,3,3,3]
j=3 PX=x^6-6*x^5-1485*x^4+1836*x^3+608148*x^2+2227824*x+12833316
[444,0,0,0,0]
[444,0,0,0,0]
[444,0,0,0,0]
[0,0,0,0,0]
[0,0,0,0,0]
NO CAPITULATION

ell1=13 ell2=37 Pk=x^6+2*x^5-313*x^4+80*x^3+26008*x^2-65704*x+90124
split=1 Hk=[78,2] HK1^iw=[78,6,3,3,3]
j=2 PX=x^6-6*x^5-2817*x^4-2160*x^3+2106648*x^2+15966072*x+65700396
TOTAL CAPITULATION in K1^iw - Inertia of 3 in F

ell1=13 ell2=37 Pk=x^6+2*x^5-313*x^4+1042*x^3+26970*x^2-225396*x+508113
split=0 Hk=[1029] HK1^iw=[3087,9,3,3,3,3]
j=2 PX=x^6-6*x^5-2817*x^4-28134*x^3+2184570*x^2+54771228*x+370414377
[0,0,0,0,0,0]
[2058,3,0,0,0,0]
[0,0,0,0,0,0]
[0,0,0,0,0,0]
[0,0,0,0,0,0]
[0,0,0,0,0,0]
NO CAPITULATION

d=-11 Pkac=x^3-3*x-46

ell1=7 ell2=19 Pk=x^6+2*x^5-54*x^4+94*x^3+2459*x^2-11352*x+36639
split=0 Hk=[291] HK1^iw=[873,3,3,3,3]
j=2 PX=x^6-6*x^5-486*x^4-2538*x^3+199179*x^2+2758536*x+26709831
[291,0,0,0,0]
[0,0,0,0,0]
[0,0,0,0,0]
[0,0,0,0,0]
[0,0,0,0,0]
NO CAPITULATION

ell1=7 ell2=19 Pk=x^6+2*x^5-54*x^4-172*x^3+2193*x^2+9130*x+38900
split=1 Hk=[57] HK1^iw=[171,3,3,3]
j=4 PX=x^6-6*x^5-486*x^4+4644*x^3+177633*x^2-2218590*x+28358100
TOTAL CAPITULATION in K1^iw - Inertia of 3 in F

Non-cyclic 3-class groups Hk:

d=-2 Pkac=x^3-3*x-10

ell1=7 ell2=181 Pk=x^6+2*x^5-837*x^4+5546*x^3+184482*x^2-2733176*x+10529273
split=1 Hk=[489,3] HK1^iw=[1467,3,3,3,3,3]
j=2 PX=x^6-6*x^5-7533*x^4-149742*x^3+14943042*x^2+664161768*x+7675840017
[489,0,0,0,0,0]
[0,0,0,0,0,0]
[0,0,0,0,0,0]
[0,0,0,0,0,0]
[0,0,0,0,0,0]
[0,0,0,0,0,0]
PARTIAL CAPITULATION in K1^iw - Inertia of 3 in F

ell1=7 ell2=181 Pk=x^6+2*x^5-837*x^4-7124*x^3+171812*x^2+2689584*x+10256868
split=0 Hk=[2388,12] HK1^iw=[50148,84,3,3,3,3]
j=2 PX=x^6-6*x^5-7533*x^4+192348*x^3+13916772*x^2-653568912*x+7477256772
[16716,0,0,0,0,0]
[0,0,0,0,0,0]
[0,0,0,0,0,0]
[0,0,0,0,0,0]
[0,0,0,0,0,0]
[0,0,0,0,0,0]
PARTIAL CAPITULATION in K1^iw - Splitting of 3 in F

ell1=7 ell2=223 Pk=x^6+2*x^5-1033*x^4-2882*x^3+268566*x^2+971028*x+1404297
split=0 Hk=[993,3] HK1^iw=[2979,9,3,3,3,3]
j=2 PX=x^6-6*x^5-9297*x^4+77814*x^3+21753846*x^2-235959804*x+1023732513
[1986,6,0,0,0,0]
[1986,6,0,0,0,0]
[0,0,0,0,0,0]
[0,0,0,0,0,0]
[0,0,0,0,0,0]
[0,0,0,0,0,0]
PARTIAL CAPITULATION in K1^iw - Splitting of 3 in F

ell1=7 ell2=223 Pk=x^6+2*x^5-1033*x^4+3362*x^3+274810*x^2-2313316*x+5362993
split=1 Hk=[51,51] HK1^iw=[153,153,9,3,3,3]
j=2 PX=x^6-6*x^5-9297*x^4-90774*x^3+22259610*x^2+562135788*x+3909621897
[0,0,6,0,0,0]
[0,0,0,0,0,0]
[0,0,3,0,0,0]
[0,0,0,0,0,0]
[0,0,0,0,0,0]
[0,0,0,0,0,0]
PARTIAL CAPITULATION in K1^iw - Inertia of 3 in F
(...)
\end{verbatim}\ns

Examples with $3$ inert in $k/\Q$ and non-cyclic 3-class groups $\CH_k$:

\ft\begin{verbatim}
d=-7 Pkac=x^3-3*x-5

ell1=13 ell2=103 Pk=x^6+2*x^5-870*x^4-8402*x^3+191539*x^2+3514100*x+15694639
split=1 Hk=[651, 3] HK1^iw=[12369,57,3,3,3]
j=2 PX=x^6-6*x^5-7830*x^4+226854*x^3+15514659*x^2-853926300*x+11441391831
TOTAL CAPITULATION in K1^iw

ell1=13 ell2=103 Pk=x^6+2*x^5-870*x^4-368*x^3+199573*x^2-237778*x+1494544
split=1 Hk=[1554,6] HK1^iw=[1554,6,3,3,3]
j=2 PX=x^6-6*x^5-7830*x^4+9936*x^3+16165413*x^2+57780054*x+1089522576
TOTAL CAPITULATION in K1^iw
\end{verbatim}\ns

\subsection{Program computing \texorpdfstring{$k_1^\acyc$}{Lg} for 
\texorpdfstring{$k = \Q(\sqrt{-m})$, $p = 3$, 
$m\not\equiv 3 \pmod 9$}{Lg}}\label{P7}

The program characterizes $k_1^\acyc$ by means of a (non-Galois) 
cubic polynomial ${\sf Q}$ or a Galois polynomial ${\sf R = 
polcompositum(Q,x^2+m)[1]}$ of degree $6$, from the Kummer radical
$w$ defining ${\sf Q}$ (one may use ${\sf polredbest}$ to simplify the 
results); it also gives the class group of $k_1^\acyc$. The number 
$\Val = v_3(\order \CT_k)-v_3(m)+2$ gives the fundamental test implying 
$\log_3({\mathfrak q})^- \in 3 \log_3(I_k \otimes \Z_3)^-$ (see Diagram \ref{DLog}). 

\smallskip
We assume $m \not \equiv 3 \pmod 9$ in the program to avoid some 
complications; use the paper \cite[Section 7, \S\,7.1, 7.2, 7.3, 7.4]{Gra2026}
for calculations without any restriction, and performing programs. 
For the case $\CH_{k^*} = 1$ giving the unique radical $\varepsilon^*$
(fundamental unit of $k^*$), giving $Q = x^3- 3\,{\bf N}(\varepsilon^*)\,x
- {\bf Tr}(\varepsilon^*)$, we have introduced the instruction 
${\sf if(rkstar==0,next)}$, which can be deleted by the reader.

\smallskip
For large $m$'s, a sufficient precision is needed (e.g., 
${\sf \backslash p\ 400}$ or much more). 
The auxiliary primes $q$ are taken in the interval defined by ${\sf Maxq}$ 
and the list ${\sf ListSigma}$ gives the integers $m$ for which the program 
fails to find the unique solution; if so this means that ${\sf Maxq}$ must be 
increased. In practice, this list is always empty with ${\sf Maxq = 500}$:

\ft\begin{verbatim}
\p 400
{ListSigma=List;Maxq=500;for(m=5,2000,if(core(m)!=m || Mod(m,9)==3,next);
P=x^2+m;k=bnfinit(P);basis=k.zk;hk=k.no;vhk=valuation(hk,3);Delta=0;
Pstar=x^2-3*m;if(Mod(m,3)==0,Pstar=x^2-m/3);kstar=bnfinit(Pstar,1);
basisstar=kstar.zk;Lstar=List;Cstar=kstar.clgp[2];rkstar=0;for(j=1,
matsize(Cstar)[2],C= Cstar[j];if(Mod(C,3)==0,rkstar=rkstar+1;listput(Lstar,j)));
\\if(rkstar==0,next);\\ has been deleted for a complete table
\\STRUCTURE OF T_k AND COMPUTATION OF Val:
nu=vhk+2;Tk=List;Kpnu=bnrinit(k,3^nu);Hknu=Kpnu.cyc;ot=1;rt=0;
d=matsize(Hknu)[2];for(j=3,d,c=Hknu[j];v=valuation(c,3);if(v>0,listput(Tk,3^v);
ot=ot*3^v;rt=rt+1));hk3=3^vhk;ram=Ramified;if(hk3!=ot,ram=Unramified);
Val=valuation(ot/m,3)+2;
\\WRITINGS OF THE MAIN INVARIANTS:
kro=kronecker(-m,3);Disc=m;if(Mod(-m,4)!=1,Disc=4*m);print();
print("m=",m," Disc=",Disc," (-m,3)=",kro," hk=",hk," Hk=",k.cyc," Tk=",Tk,
" hkstar=",kstar.no," Hkstar=",kstar.cyc," #Hk=",hk3,"  k_1^ac/k is ",ram);
\\COMPUTATION OF THE LIST Lw OF RADICALS:
Lw0=List;w0=kstar.fu[1];listput(Lw0,w0);
for(i=1,rkstar,Idstar=kstar.clgp[3][Lstar[i]];
alpha=bnfisprincipal(kstar,idealpow(kstar,Idstar,Cstar[Lstar[i]])); 
wi=Mod(basisstar[1]*alpha[2][1]+basisstar[2]*alpha[2][2],Pstar);
listput(Lw0,wi));Nrad=rkstar+1;Lw=List;nw=3^Nrad-1;
LD=List;for(i=1,nw,D=digits(i,3);if(D[1]!=2,listput(LD,D)));
nw=nw/2;for(i=1,nw,d=matsize(LD[i])[2];d0=Nrad+1-d;w=1;
for(j=d0,Nrad,w=w*Lw0[j]^LD[i][j-d0+1]);listput(Lw,w));
print("Lw=",Lw);
\\RESEARCH OF THE SOLUTION AMONG THE RADICALS:
for(J=1,nw,w=Lw[J];Nw=norm(w);ispower(Nw,3,&a);
Qac=x^3-3*a*x-trace(w);Q=polredbest(Qac);Test=0;e0=1;if(Mod(m,3)==0,e0=2);
dres=1;if(kro==-1,dres=2);forprime(q=5,Maxq,if(kronecker(-m,q)!=1,next);
A=component(idealfactor(k,q),1)[1];B=idealpow(k,A,hk);b=bnfisprincipal(k,B);
pib=Mod(basis[1]*b[2][1]+basis[2]*b[2][2],P)^(3^dres-1)-1;Log=pib;t=1;
while(t<=e0*(Val+log(t)+1),t=t+1;Log=Log-(-1)^t*pib^t/t);Log=lift(Log);
C1=polcoeff(Log,1);if(valuation(C1,3)<=Val,next);Qq=Q+O(q);
if(polisirreducible(Qq)==1,print("J=",J," q=",q," Qq irreducible");Test=1;break));
if(Test==0,Delta=Delta+1;R=polcompositum(P,Q)[1];kacyc=bnfinit(polredbest(R),1);
print("SOLUTION:J=",J," w=",w);print("Q^ac=",Q," Val=",Val," Hkacyc=",kacyc.cyc)));
if(Delta!=1,listput(ListSigma,m)));print(ListSigma)}
\end{verbatim}\ns

Let's give some extracts, where ${\sf Q^{ac} \to polredbest(Q^{ac})}$:

\ft\begin{verbatim}
m=107 Disc=107 kronecker(-m,3)=1 hk=3 Hk=[3] Tk=[3] hkstar=3 
Hkstar=[3] #Hk=3  k_1^ac/k is Ramified
Lw=List([Mod(-1/2*x-17/2,x^2-321),Mod(12*x-215,x^2-321),
Mod(11/2*x-197/2,x^2-321),Mod(5/2*x-91/2,x^2-321)])
SOLUTION:J=1 w=Mod(-1/2*x-17/2,x^2-321)
Q^ac=x^3+6*x-17   Val=3    Hkacyc=[9,3]
J=2 q=79 Qq irreducible
J=3 q=79 Qq irreducible
J=4 q=79 Qq irreducible

m=262 Disc=1048 kronecker(-m,3)=-1 hk=6 Hk=[6] Tk=[3] hkstar=6 
Hkstar=[6] #Hk=3  k_1^ac/k is Ramified
Lw=List([Mod(-2*x+137,x^2-786),Mod(28*x-785,x^2-786),
Mod(5406*x-151561,x^2-786),Mod(1043744*x-29262089,x^2-786)])
J=1 q=113 Qq irreducible
J=2 q=113 Qq irreducible
J=3 q=113 Qq irreducible
SOLUTION:J=4 w=Mod(1043744*x-29262089,x^2-786)
Q^ac=x^3+42*x-40   Val=3    Hkacyc=[54,9]

m=362 Disc=1448 kronecker(-m,3)=1 hk=18 Hk=[18] Tk=[3] hkstar=6 
Hkstar=[6] #Hk=9  k_1^ac/k is Unramified
Lw=List([Mod(-20*x+743,x^2-1086),Mod(22*x+725,x^2-1086),
Mod(1846*x+60835,x^2-1086),Mod(154878*x+5105285,x^2-1086)])
J=1 q=53 Qq irreducible
J=2 q=53 Qq irreducible
J=3 q=53 Qq irreducible
SOLUTION:J=4 w=Mod(154878*x+5105285,x^2-1086)
Q^ac=x^3+5*x-14   Val=3    Hkacyc=[6]

m=586 Disc=2344 kronecker(-m,3)=-1 hk=18 Hk=[18] Tk=[9] hkstar=6 
Hkstar=[6] #Hk=9  k_1^ac/k is Ramified
Lw=List([Mod(28*x-2491,x^2-1758),Mod(14*x+587,x^2-1758),
Mod(-18438*x-773081,x^2-1758),Mod(24282790*x+1018152659,x^2-1758)])
J=1 q=137 Qq irreducible
J=2 q=137 Qq irreducible
J=3 q=137 Qq irreducible
SOLUTION:J=4 w=Mod(24282790*x+1018152659,x^2-1758)
Q^ac=x^3-99*x-630   Val=4    Hkacyc=[90,15]

m=974 Disc=3896 kronecker(-m,3)=1 hk=36 Hk=[12,3] Tk=[3] hkstar=6 
Hkstar=[6] #Hk=9  k_1^ac/k is Unramified
Lw=List([Mod(444*x+24101,x^2-2922),Mod(18*x+973,x^2-2922),
Mod(865830*x+46802897,x^2-2922),
Mod(41647855098*x+2251296756037,x^2-2922)])
J=1 q=71 Qq irreducible
J=2 q=71 Qq irreducible
J=3 q=71 Qq irreducible
SOLUTION:J=4 w=Mod(41647855098*x+2251296756037,x^2-2922)
Q^ac=x^3-x^2+16*x-8   Val=3    Hkacyc=[12,3,3]

EXAMPLES WITH rkstar=2:

m=28477 Disc=113908 (-m,3)=-1 hk=54 Hk=[18,3] Tk=List([3,3]) hkstar=18 
Hkstar=[6,3] #Hk=27  k_1^ac/k is Unramified
SOLUTION:J=4 w=Mod(167093627943*x-48839097896288,x^2-85431)
Q^ac=x^3-47*x-140 Val=4 Hkacyc=[36,12,6]

m=32573 Disc=130292 (-m,3)=1 hk=126 Hk=[42,3] Tk=List([3,3]) hkstar=18 
Hkstar=[6,3] #Hk=9  k_1^ac/k is Ramified
SOLUTION:J=4 w=Mod(9469075744*x+2960044436165,x^2-97719)
Q^ac=x^3+402*x-5428 Val=4 Hkacyc=[126,3,3,3]
(...)
m=39677 Disc=158708 (-m,3)=1 hk=144 Hk=[24,6] Tk=List([3,3]) hkstar=36 
Hkstar=[6,6] #Hk=9  k_1^ac/k is Ramified
SOLUTION:J=12 w=Mod(6552180178393*x-2260559136307388,x^2-119031)
Q^ac=x^3+585*x-13428 Val=4 Hkacyc=[72,6,3,3]

m=50983 Disc=50983 (-m,3)=-1 hk=90 Hk=[30,3] Tk=List([3,3]) hkstar=18 
Hkstar=[6,3] #Hk=9  k_1^ac/k is Ramified
SOLUTION:J=12 w=Mod(-11066570533*x+4327991281906,x^2-152949)
Q^ac=x^3-471*x-4234 Val=4 Hkacyc=[90,3,3,3]
\end{verbatim}\ns

For $k = \Q(\sqrt{-262})$, the resulting polynomial, given analytically by
Kundu--Washington in \cite[\S\,5.2]{KW2023}, is $x^3+42 x-40$, 
in perfect accordance with our result obtained in the Kummer context
and the properties of the $\Log_p$-function.

\subsection{Capitulation of \texorpdfstring{$\CH_k$}{Lg}
in \texorpdfstring{$k^\acyc$}{Lg}}\label{P8}

The following program computes $\Bnu_{k_1^\acyc/k}(\CH_{k_1^\acyc})$ 
to see if there is capitulation (or partial capitulation or none) of $\CH_k$; 
some examples confirm statements of Theorems \ref{capitule}, \ref{capitule2}:

\ft\begin{verbatim}
{L=[[107,x^3+6*x-17],[262,x^3+42*x-40],[298,x^3-99*x-522],
[302,x^3-93*x-458],[331,x^3+36*x-45],[362,x^3+5*x-14],
[367,x^3-x^2+2*x+3],[397,x^3+90*x-252],[586,x^3-99*x-630]];
for(i=1,9,m=L[i][1];Qac=L[i][2];P=x^2+m;k=bnfinit(P);hk=k.no;
hta=hk/3^valuation(hk,3);R=polcompositum(Qac,P)[1];kacyc=bnfinit(R,1);
print();print("m=",m," Qac=", Qac," Hk=",k.cyc," Hkacyc",1,"=",kacyc.cyc);
\\SEARCH OF AN ELEMENT OF ORDER 3 OF G:
G=nfgaloisconj(kacyc);for(j=1,6,S=G[j];T=nfgaloisapply(kacyc,S,S);
if(T==x,next);T=nfgaloisapply(kacyc,T,S);if(T!=x,next);S0=S;break);
\\COMPUTATION OF THE ALGEBRAIC NORMS:
d=matsize(kacyc.clgp[2])[2];for(j=1,d,
A=kacyc.clgp[3][j];As=nfgaloisapply(kacyc,S0,A);
Ass=nfgaloisapply(kacyc,S0,As);NuA=idealmul(kacyc,A,As);
NuA=idealmul(kacyc,NuA,Ass);NuA=idealpow(kacyc,NuA,hta);
B=bnfisprincipal(kacyc,NuA);print(B)))}

m=107 Q^ac=x^3+6*x-17 Hk=[3] Hkacyc=[9,3] k_1^acyc/k Ramified
[[3,0],[1,0,0,0,0,0]]
[[0,0],[2,0,0,0,0,0]]
No capitulation in k_1^acyc 

m=262 Q^ac=x^3+42*x-40 Hk=[6] Hkacyc=[54,9] k_1^acyc/k Ramified
[[36,0],[too big rationals]]
[[0,0],[4,0,0,0,0,0]]
No capitulation in k_1^acyc 

m=298 Q^ac=x^3-99*x-522 Hk=[6] Hkacyc=[6,6,2] k_1^acyc/k Ramified
[[0,0,0],[2,0,0,0,0,0]]
[[0,0,0],[24409,9324,-5650,-54,-2724,-698]]
[[0,0,0],[8,0,0,0,0,0]]
Capitulation in k_1^acyc 

m=302 Q^ac=x^3-93*x-458 Hk=[12] Hkacyc=[12,3] k_1^acyc/k Ramified
[[0,0],[22143693,1346420,-829846,2772642,-3983252,736266]]
[[0,0],[840263516,26919056,192957256,35605288,110376752,21654320]]
Capitulation in k_1^acyc 

m=331 Q^ac=x^3+36*x-45 Hk=[3] Hkacyc=[3,3,3] k_1^acyc/k Ramified
[[0,0,0],[2,0,0,0,0,0]]
[[0,0,0],[2,0,0,0,0,0]]
[[2,1,0],[-5/4,1/4,0,0,1/4,0]]
No capitulation in k_1^acyc 

m=362 Q^ac=x^3+5*x-14 Hk=[18] Hkacyc=[6] k_1^acyc/k Unramified
[[0],[187,-66,78,114,30,-18]]
Partial capitulation in k_1^acyc

m=367 Q^ac=x^3-x^2+2*x+3 Hk=[9] Hkacyc=[3] k_1^acyc/k Unramified
[[0],[2,2,2,-1,0,-1]]
Partial capitulation in k_1^acyc 

m=397 Q^ac=x^3+90*x-252 Hk=[6] Hkacyc=[6,3] k_1^acyc/k Ramified
[[0,0],[2,0,0,0,0,0]]
[[0,0],[[413346743896,-96157089046,-58525843370,-10433180696,
                                -72646724196,12020514874]]
Capitulation in k_1^acyc

m=586 Q^ac=x^3-99*x-630 Hk=[18] Hkacyc=[90,15] k_1^acyc/k Ramified
[[60,0],[too big rationals]]
[[0,0],[-6419648, 839936, 2340736, 671488, -686912, 168064]]
Partial capitulation in k_1^acyc 
\end{verbatim}\ns

The case ${\sf m = 367,\ Q^{ac} = x^3-x^2+2*x+3,\ H_k = [9],
\  H_k^{acyc} = [3]}$, is interesting since $\CH_k$, cyclic of order $9$, partially 
capitulates in $k_1^\acyc$ unramified over $k$, defined by the polynomial
$R =  x^6 - 15*x^4 - 7*x^3+148*x^2+236*x+104$. 
So, we get $\Bnu_{k_1^\acyc/k}(\CH_{k_1^\acyc}) = 
\BJ_{k_1^\acyc/k}(\CH_k{'}) = \BJ_{k_1^\acyc/k}(\CH_k^3) = 1$ 
(the image of $\BN_{k_1^\acyc/k}$ is $\CH_k{'}$), giving that the 
subgroup $\CH_k{'} = \CH_k^3$ capitulates in $k_1^\acyc$; indeed, 
we have the following data, where $\CH_k$ is generated by the class 
of ${\mathfrak p}_2 \mid 2$, with $\BJ_{k_1^\acyc/k}({\mathfrak p}_2) 
=: (\beta)$ defined by:

\ft\begin{verbatim}
beta=Mod(-195553/53501908122*x^5+80786/26750954061*x^4-66973577/17833969374*x^3 
     +339236483/53501908122*x^2-6106391899/5944656458*x+8976954289/2057765697, 
     x^6+1131*x^4+194*x^3+404292*x^2-210684*x+45482177)
\end{verbatim}\ns

To check this result, one computes the ideal decomposition of $\beta$ 
and its norm, which leads to $(\beta) = {\mathfrak P}_2^3$, with 
${\mathfrak P}_2 \mid {\mathfrak p}_2$ inert in $k_1^\acyc/k$, 
and $\beta$ of norm $512$ as expected.

\end{appendices}

\section*{Definitions and Notations}

We organize the main definitions and notations regarding various  
categories of objects; $p \geq 3$ is a prime number split into ${\mathfrak p}$, 
${\ov {\mathfrak p}} := {\mathfrak p}^\tau$ (complex conjugate of ${\mathfrak p}$), 
in the base field $k$:

\smallskip
{\it Number fields}
\begin{itemize}
\item $\tau$: complex conjugation, 
\item $g = \langle \tau \rangle$,
\item $k$: imaginary quadratic field,
\item ``Totally $p$-adic field'': field in which $p$ totally splits,
\item $k^\cyc$: cyclotomic $\Z_p$-extension of $k$,
\item $k^\acyc$: anti-cyclotomic $\Z_p$-extension of $k$,
\item $\wt k$: compositum of the $\Z_p$-extensions of $k$,
\item $K, \ov K := K^\tau$: non-Galois $\Z_p$-extension of $k$ and its conjugate,
\item $K_n$: the $n^{\rm th}$ layer of the $\Z_p$-extension $K/k$,
\item $K_e, K_{\ov e}$: inertia fields of ${\mathfrak p}$, ${\ov {\mathfrak p}}$, in $K/k$,
\item $L, \ov L$: inertia fields of ${\mathfrak p}$, ${\ov {\mathfrak p}}$, in $\wt k/k$,
\item $H_k^\nr$: $p$-Hilbert's class field of $k$,
\item $H_k^\pr$: maximal abelian $p$-ramified pro-$p$ extension of $k$,
\item $H_k^\bp$: Bertrandias--Payan field (maximal abelian
pro-$p$-extension of $k$ in which any cyclic extension is embeddable
in cyclic $p$-extensions of $k$ of arbitrary degree),
\item $H_k^\lc$: maximal abelian locally cyclotomic pro-$p$-extension of $k$
(i.e. such that $p$ totally splits in $H_k^\lc/k^\cyc$),
\item $H_k^\lac$: maximal abelian locally anti-cyclotomic pro-$p$-extension of $k$
\item $\knr = \wt k \cap H_k^\nr$, 
\item $K^\nr = K \cap H_k^\nr$,
\item $\wt e = v_p([\knr : k])$
\end{itemize}

{\it Units, $S$-units}
\begin{itemize}
\item $S_F$ or simply $S$: set of $p$-places of a number field $F$,
\item $E_F$: group of global units of $F$,
\item $E_F^S$: group of global $S$-units of $F$,
\item $\hp$: order of the class of ${\mathfrak p}$,
\item $x$: generator of ${\mathfrak p}^{\hp}$ (fundamental 
${\mathfrak p}$-unit), and $\ov x$ its complex conjugate,
\item $\CU_{\mathfrak p}$: group of principal local units of the completion 
$k_{\mathfrak p}$,
\item $\CU_p = \oplus_{{\mathfrak p} \mid p} \CU_{\mathfrak p}$,
\end{itemize}

{\it Class groups, $S$-class groups, Galois groups}
\begin{itemize}
\item $\BH_F = I_F/P_F$: class group of the number field $F$,
\item $\CH_F = \BH_F \otimes \Z_p$: $p$-class group of $F$,
\item $\Ccl : I_F \mapsto \BH_F$ or $I_F \mapsto \CH_F$,
\item ``$p$-principality of the ideal ${\mathfrak a} \in I_F$'': $\Ccl({\mathfrak a}) = 1$
in $\CH_F$,
\item $\CH_F^S = \CH_F/\Ccl(S)$: $S$-class group of $F$ for the 
set $S$ of $p$-places of $F$,
\item $\CH_n^i = \CH_{K_n}^i$: the $i^{\rm th}$-element of the filtration of $\CH_{K_n}$,
\item $\CH_{K_n} = \Ccl(\CI_n)$ for a group of ideals $\CI_n$ of $K_n$,
\item $\Lbda_{K_n/k} = \{a \in k^\times, \, (a) \in \BN_{K_n/k}(\CI_n)\}$,
\item $\CT_k$: $p$-torsion group $\Gal(H_k^\pr/ \wt k)$,
\item $\CW_k^\bp = \oplus_v \mu_p(k_v)\big / \mu_p(k) \simeq \Gal(H_k^\pr/H_k^\bp)$
(trivial in the split case),
\item $\CT_k^\bp = \Gal(H_k^\bp/\wt k)$: Bertrandias--Payan module
($= \CT_k$ in the split case),
\item $\wt \CH_k \simeq \Gal(H_k^\lc/ k^\cyc)$: logarithmic class group,
\item $\CH_{K_n}^\ram = \Ccl( \langle {\mathcal P}_n, \ov {\mathcal P}_n \rangle)
\subset \CH_n$; ${\mathcal P}_n =$ product of the ${\mathfrak P}_n \mid p$ in $K_n$,
\item $G_n = \Gal(K_n/k)$,
\item $X^- = X^{\frac{1-\tau}{2}}$, $X^+ = X^{\frac{1+\tau}{2}}$, for a $g$-module $X$,
\item $\Gamma = \Gal(\wt k/k)$, $\Gamma^- = \Gal(\wt k/k^\cyc)$, $\Gamma^+
= \Gal(\wt k/k^\acyc)$
\end{itemize}

{\it Canonical maps}
\begin{itemize}
\item $\BN_{M/F}$, $\BN_{m/n}$: arithmetic norm in $M/F$, in $K_m/K_n$,
\item $\Bnu_{M/F}$: algebraic norm in $M/F$,
\item $\BJ_{M/F}$, $\BJ_{m/n}$: transfer map in $M/F$, in $K_m/K_n$,
\item $\Log_p({\mathfrak a}) \equiv \frac{1}{m} \log_p(\alpha) \pmod 
{\Q_p \log_p(E_k)}$, if ${\mathfrak a}^m = (\alpha)$ in $k$, 
\item $\log_{\mathfrak p}$, $\log_{\ov {\mathfrak p}}$: usual $p$-adic
logarithms defined on the completions $k_{\mathfrak p}^\times$, 
$k_{\ov {\mathfrak p}}^\times$ of $k$,
\item $v_p$: $p$-adic valuation, $v_{\mathfrak p}$: ${\mathfrak p}$-adic valuation, 
\end{itemize}

{\it Symbols}
\begin{itemize} 
\item $\big(\fffrac{x\,,\,F/k}{\mathfrak q}\big)$: Hasse norm residue symbols 
of $x$ at ${\mathfrak q}$,
\item $\big(\fffrac{F/k}{\mathfrak A}\big)$: Artin symbol of the ideal ${\mathfrak A}$,
\item $G_{n,{\mathfrak p}}$, $G_{n,{\ov {\mathfrak p}}}$: inertia groups of the 
$p$-places in $K_n/k$,
\item $\Omega_{K_n/k} = \big\{(s, \ov s) \in G_{n,\mathfrak p} \times 
G_{n,{\ov {\mathfrak p}}},\ \ s\cdot \ov s = 1 \big\}$,
\item $\CN_{K_n/k}$: subgroup of elements $a \in k^\times$ 
local norms in $K_n/k$ at all place $v \nmid p$,
\item $\omega_{K_n/k} : \CN_{K_n/k} \to \Omega_{K_n/k}$
\end{itemize}

{\it Iwasawa's theory}
\begin{itemize}
\item $\lambda_p(K/k)$, $\mu_p(K/k)$, $\nu_p(K/k)$: usual Iwasawa's 
invariants for $K/k$,
\item $\lambda_p^S(K/k)$, $\mu_p^S(K/k)$, $\nu_p^S(K/k)$: Iwasawa's 
invariants for the $S$-class groups in $K/k$
\item $\ds \CC_{\,\wt k} = \limproj_{F \subset \wt k} \CH_F$ for
the arithmetic norms
\end{itemize}

{\it Class numbers and numerical invariants}
\begin{itemize}
\item $\hk = \order \BH_k$: class number of $k$,
\item $\hp \mid \hk$: order of the class of ${\mathfrak p} \mid p$,
\item $\delta_p(k) = v_{\mathfrak p}(\ov x^{\,p-1}-1) - 1$ 
(${\mathfrak p}$-valuation of the ${\mathfrak p}$-Fermat quotient of $\ov x$),
\item $\wt \delta_p(k) = \delta_p(k) + \big [v_p (\hk) 
- v_p(\hp) \big] = \delta_p(k) + v_p(\order \CH_k^{S_k}) =
v_p(\wt \CH_k)$
\end{itemize}

{\it Kummer radicals and extensions}
\begin{itemize}
\item $M = k(\mu_3)$, $k^* = \Q(\sqrt{3m})$,
\item $H_M^\pr[p]$: maximal elementary abelian $p$-extension of $M$ in $H_M^\pr$,
\item $\Rad(\hbox{\ft$H_M^\pr[p]$}) := \big\{w M^{\times p}, 
\ M(\sqrt[p]w) \subseteq H_M^\pr[p] \big \}$,
\item $Q^\acyc = x^3 - 3a x - t$ ($a^3 = w w'$, $t = w + w'$):
irreducible polynomial defining $k_1^\cyc$ from the radical $w \in k^*$
\end{itemize}

{\it Characters}
\begin{itemize} 
\item $\psi$ (degree 1), $\varphi$ ($p$-adic), $\chi$ (rationnel),
\item $\Delta$: Galois group of prime-to-$p$ order,
\item $\Z_\varphi = \Z_p[\Delta] e_\varphi$ for the idempotent 
$e_\varphi = \frac{1}{d}\sum_{\tau\in\Delta}
\varphi(\tau)\tau^{-1}$
\item $\wt h_\varphi = \order \wt \CH_k^{\,e_\varphi}$
\end{itemize}

\nocite{*}
\bibliographystyle{cdraifplain}
\bibliography{xampl}

\begin{thebibliography}{10}
\bibitem{random-note-crossref}
Volume~2 is listed under Knuth \cite{book-full}.
\end{thebibliography}

\end{document}